%% file: quasi_newton_minimax.tex
\documentclass[11pt,letterpaper]{article}

\usepackage[utf8]{inputenc}
\usepackage{comment}
\usepackage{microtype}
\usepackage{graphicx}
\usepackage{subfigure}
\usepackage{booktabs} %
\usepackage{scalerel}
\usepackage[margin=1in]{geometry}

\usepackage[style=alphabetic, maxbibnames=15, giveninits=true, maxcitenames=15, natbib=true,
  maxalphanames=10, backend=biber, sorting=nty, backref=true, uniquename=false]{biblatex}
\DefineBibliographyStrings{english}{backrefpage = {page},backrefpages = {pages},}
\DeclareNameAlias{sortname}{given-family}
\renewbibmacro{in:}{%
  \ifentrytype{article}{}{\printtext{\bibstring{in}\intitlepunct}}}

\usepackage{tikz}
\usetikzlibrary{decorations.pathreplacing,calc}
\newcommand{\tikzmark}[1]{\tikz[overlay,remember picture] \node (#1) {};}

\newcommand*{\AddNote}[4]{%
    \begin{tikzpicture}[overlay, remember picture]
        \draw [decoration={brace,amplitude=0.3em},decorate,thick,black]
            ($(#3.east)!(#1.north)!($(#3.east)-(0,1)$)$) --  
            ($(#3.east)!(#2.south)!($(#3.east)-(0,1)$)$)
                node [align=left, text width=4cm, pos=0.5, anchor=west] {#4};
    \end{tikzpicture}
}
\usepackage{bm}
\usepackage{amssymb}
\usepackage{amsthm}
\usepackage{amsmath}
\usepackage{mathtools}

\usepackage{nicefrac}

\usepackage{algorithmic}
\usepackage{algorithm}
\usepackage[shortlabels]{enumitem}
\usepackage{eqparbox}
\definecolor{comment}{RGB}{2,128, 9}

\floatstyle{ruled}
\newfloat{subroutine}{htbp}{lok}
\floatname{subroutine}{Subroutine}

\usepackage{url}
\usepackage{float}

\usepackage{pifont}
\usepackage{hhline}
\usepackage{multirow}

\newcommand{\myalert}[1]{{\vspace{2mm}\noindent\textbf{#1}}}

\input{math_commands.tex}

\newcommand{\calB}{\mathcal{B}}
\newcommand{\calC}{\mathcal{C}}

\newcommand{\calD}{\mathcal{D}}
\newcommand{\calL}{\mathcal{L}}
\newcommand{\calX}{\mathcal{X}}
\newcommand{\calY}{\mathcal{Y}}
\newcommand{\calZ}{\mathcal{Z}}
\newcommand{\mydef}{\triangleq}
\newcommand{\gap}{\mathrm{Gap}}
\newcommand{\reg}{\mathrm{Reg}}
\newcommand{\dreg}{\textrm{D-Reg}}
\newcommand{\vF}{\Vector{F}}

\newcommand{\blue}[1]{{\textcolor{blue}{#1}}}

\renewcommand{\blue}[1]{#1}

\newcommand{\vspan}{\mathrm{span}}
\newcommand{\SEP}{\mathsf{SEP}}

\newtheorem{assumption}{Assumption}
\newtheorem{definition}{Definition}
\newtheorem{theorem}{Theorem}
\newtheorem{lemma}[theorem]{Lemma}
\newtheorem{proposition}[theorem]{Proposition}
\newtheorem{corollary}[theorem]{Corollary}
\newtheorem{remark}{Remark}
\numberwithin{assumption}{section}
\numberwithin{definition}{section}
\numberwithin{theorem}{section}
\numberwithin{remark}{section}

\let\origtheassumption\theassumption

\usepackage{hyperref}
\hypersetup{colorlinks,citecolor=blue,linktocpage,breaklinks=true}

\usepackage{cancel}

\title{Online Learning Guided Quasi-Newton Methods with Global Non-Asymptotic Convergence\thanks{A preliminary version of part of this work appeared in \cite{jiang2023online}.}
}

\author{Ruichen Jiang\thanks{Department of Electrical and Computer Engineering, The University of Texas at Austin, Austin, TX, USA\qquad\qquad\{rjiang@utexas.edu, mokhtari@austin.utexas.edu\}} \and
        Aryan Mokhtari$^\dagger$
}

\date{}

\begin{document}

\maketitle

\begin{abstract}%
In this paper, we propose a quasi-Newton method for solving smooth and monotone nonlinear equations, including unconstrained minimization and minimax optimization as special cases. For the strongly monotone setting, we establish two global convergence bounds: (i) a linear convergence rate that matches the rate of the celebrated extragradient method, and (ii) an explicit global superlinear convergence rate that provably surpasses the linear convergence rate after at most $\mathcal{O}(d)$ iterations, where~$d$ is the problem's dimension. 
In addition, for the case where the operator is only monotone, we prove a global convergence rate of $\mathcal{O}(\min\{\frac{1}{k},\frac{\sqrt{d}}{k^{1.25}}\})$ in terms of the duality gap. This matches the rate of the extragradient method when $k = \mathcal{O}(d^2)$ and is faster when $k = \Omega(d^2)$. 
These results are the first global convergence results to demonstrate a provable advantage of a quasi-Newton method over the extragradient method, without querying the Jacobian of the operator. Unlike classical quasi-Newton methods, we achieve this by using the hybrid proximal extragradient framework and a novel online learning approach for updating the Jacobian approximation matrices. Specifically, guided by the convergence analysis, we formulate the Jacobian approximation update as an online convex optimization problem over non-symmetric matrices, relating the regret of the online problem to the convergence rate of our method. To facilitate efficient implementation, we further develop a tailored online learning algorithm based on an approximate separation oracle, which preserves structures such as symmetry and sparsity in the Jacobian matrices. 
\end{abstract}

\newpage
\section{Introduction}

In this paper, we consider the problem of solving a system of nonlinear equations:
\begin{equation}\label{eq:monotone}
  \vF(\vz) = 0,
\end{equation}
where $\vF : \reals^d \rightarrow \reals^d$ is a continuously differentiable operator.
This class of problems emerges from finite-dimensional approximations to nonlinear differential and integral equations and has applications in diverse fields of mathematical physics \cite{more1990computational,averick1992minpack,Nocedal2006}. Problem~\eqref{eq:monotone} is also closely related to variational inequalities (VIs), as it can be regarded as a VI without constraints. Moreover, this formulation captures the optimality conditions of unconstrained minimization and minimax optimization problems \cite{Facchinei2004}. 

One well-known method for solving the problem in~\eqref{eq:monotone}  is Newton's method, noted for its fast convergence. When the Jacobian $\nabla \vF(\vz)$ is non-singular, Newton's method achieves quadratic convergence near the optimal solution~\cite{Nocedal2006}. However, a major challenge of its implementation is computing the Jacobian $\nabla \vF$, especially in high-dimensional settings. Hence, various modifications to Newton's method have been considered to improve computational efficiency, and among them quasi-Newton methods are the most popular.
They were first introduced in~\cite{davidon1959variable}, for minimization problems, and in~\cite{broyden1965class}, for systems of equations. They aim to emulate Newton's method while approximating the Jacobian $\nabla \vF$ solely through the operator information $\vF$, and they require $\mathcal{O}(d^2)$ arithmetic operations per iteration.  
For minimization problems, several popular update rules include the Davidon-Fletcher-Powell (DFP) method \cite{davidon1959variable,fletcher1963rapidly}, the Broyden-Fletcher-Goldfarb-Shanno (BFGS) method \cite{broyden1970convergence,fletcher1970new,goldfarb1970family,shanno1970conditioning}, and the symmetric rank-one (SR1) method \cite{conn1991convergence,khalfan1993theoretical}. For solving systems of equation, Broyden's methods~\cite{broyden1965class} remain the most popular choices. In the following paragraphs, we summarize the known theoretical guarantees for these methods.

\myalert{Asymptotic superlinear convergence.}
The classical analysis of quasi-Newton methods aims to establish  
their Q-superlinear convergence under suitable conditions, meaning $\lim_{k \rightarrow +\infty}\frac{\|\vz_{k+1}-\vz^*\|}{\|\vz_k-\vz^*\|} =0$ where $\vz^*$ is the optimal solution of \eqref{eq:monotone}.  
For the special case of minimization problems, when the function is locally smooth and strongly convex, it was established 
in~\cite{broyden1973local,dennis1974characterization} that DFP and BFGS are locally and Q-superlinearly convergent. 
To ensure global convergence, it is necessary to incorporate quasi-Newton updates with a line search or a trust-region method. When the function is smooth and strongly convex, it was shown in~\cite{powell1971convergence,dixon1972variable} that DFP and BFGS with an exact line search converge globally and Q-superlinearly. 
Subsequently, it was shown in~\cite{powell1976some} that BFGS with an inexact line search retains global and superlinear convergence, and this result was later extended to the restricted Broyden class except for DFP~\cite{byrd1987global}. Along another line of research, the SR1 method with trust region techniques was studied in~\cite{conn1991convergence,khalfan1993theoretical,byrd1996analysis} and they also proved its global and superlinear convergence.
For general nonlinear equations, it was shown in \cite{broyden1973local} that when the Jacobian $\nabla \vF$ is non-singular and Lipschitz continuous, Broyden's method converges locally and superlinearly.   
Subsequent work in~\cite{more1976global} established the local superlinear convergence of a modified Broyden's method by Powell~\cite{powell1970hybrid}. Moreover, these superlinear convergence results have also been extended to Hilbert spaces~\cite{griewank1987local,kelley1991new}.  In addition, several works proposed line search strategies for Broyden's method~\cite{griewank1986global,li2000derivative} to ensure global and superlinear convergence. However, these results  
are all \textit{asymptotic} and they do not provide an explicit convergence rate.

\myalert{Local non-asymptotic superlinear convergence.}
Recent work has attempted to establish non-asymptotic guarantees for quasi-Newton methods when applied to solving Problem~\eqref{eq:monotone} and the special case of minimization problems.
\blue{For smooth and strongly convex minimization problems, 
this line of research was initiated by \cite{rodomanov2021greedy}, which studied a greedy variant of quasi-Newton methods
and demonstrated a \textit{local non-asymptotic} superlinear convergence rate of $(1-\frac{\mu}{d L_1})^{\nicefrac{k^2}{2}} (\frac{dL_1}{\mu})^k$. Here, $d$ is the problem's dimension, and $L_1$ and $\mu$ are the smoothness and strong convexity parameters, respectively.   
Later, two concurrent works~\cite{rodomanov2021new,jin2022non} examined local non-asymptotic superlinear convergence rates of \emph{classical} quasi-Newton methods. 
Specifically, the authors in~\cite{rodomanov2021new} proved that in a local neighborhood of the optimal solution, if the initial Hessian approximation is set as $L_1\mI$, BFGS with unit step size converges at a superlinear rate of  
$\left(\frac{d}{k}\log\frac{L_1}{\mu}\right)^{\nicefrac{k}{2}}$. 
Concurrently, the authors in~\cite{jin2022non} showed that if the initial Hessian approximation is {close to the Hessian at the optimal solution or selected as the Hessian at the initial point,}
BFGS with unit step size achieves a local superlinear convergence rate of $\left(\frac{1}{k}\right)^{\nicefrac{k}{2}}$. Further details and follow-up works can be found in \cite{jin2022sharpened,lin2021greedy,ye2022towards}.} 
Subsequently, similar results have been established for solving general nonlinear equations. 
When the Jacobian $\nabla \vF$ is Lipschitz continuous and $\nabla \vF(\vz^*)$ is non-singular, under the condition that the initial point and the initial Jacobian approximation matrix are sufficiently close to $\vz^*$ and $\nabla \vF(\vz^*)$, respectively, the authors in~\cite{lin2021explicit} demonstrated that Broyden's method achieves an explicit local non-asymptotic superlinear rate of~$\left(\frac{1}{k}\right)^{\nicefrac{k}{2}}$.
Moreover, under similar assumptions and a stronger initial condition, the greedy and random variants of Broyden's method were proposed in~\cite{ye2021greedy} and shown to achieve a local superlinear rate of $(1-\frac{1}{d})^{\nicefrac{k^2}{4}}$. Building upon this, the block Broyden's method was later presented in~\cite{liu2023block}. 
However, these non-asymptotic results are limited to \textit{local neighborhoods of the solution}, and several of them, especially those for solving nonlinear equations, also require the initial Hessian/Jacobian approximation matrix to be sufficiently close to the actual one at $\vz^*$~\cite{jin2022non,lin2021explicit,ye2021greedy,liu2023block}. 
Thus, they do not lead to a global convergence guarantee.

\myalert{\blue{Convergence without strong convexity/non-singularity.}} \blue{Note that all the results above only apply under the restrictive assumption that the objective is strongly convex in the minimization setting, or that the operator is non-singular in the nonlinear equation setting. There are a few works that study quasi-Newton methods when the objective is merely convex or when the operator is possibly singular. However, to the best of our knowledge, no theoretical result demonstrates the advantage of quasi-Newton methods in these more general settings. Specifically, for minimizing smooth convex functions, it has been shown that classical quasi-Newton methods such as BFGS converge asymptotically~\cite{byrd1987global,powell1972some} but no explicit rates have been provided. Several recent works~\cite{scheinberg2016practical,ghanbari2018proximal,kamzolov2023accelerated,scieur2024adaptive} have combined quasi-Newton updates with variable metric proximal gradient or cubic regularization techniques and proved sublinear rates. However, their results are no better than the first-order counterparts such as gradient descent or accelerated gradient descent.}  %

\myalert{\blue{Contributions.}} In summary, 
\blue{for the strongly monotone case,}
while quasi-Newton methods have the potential to achieve a faster rate, the current theory is either asymptotic or applicable only in a local neighborhood of the solution, with a stringent requirement on the initial Jacobian. \blue{Furthermore, for the monotone case, no theoretical advantage of quasi-Newton methods has been demonstrated in the literature.  
In this paper, we aim to close these gaps and present a novel quasi-Newton proximal extragradient (QNPE) method.}
When the operator is $\mu$-strongly monotone,  
it attains the following global convergence bounds:
\begin{equation}\label{eq:rate}
    \frac{\|\vz_k-\vz^*\|^2}{\|\vz_0-\vz^*\|^2} \leq  \min\left\{ \left(1+\frac{\mu}{30L_1}\right)^{-k}, \left(1+\frac{\mu}{16L_1}\sqrt{\frac{k}{M}}\right)^{-k}\right\},
    \vspace{-1mm}
\end{equation}
where $M = \bigO(\frac{\|\mB_0-\nabla \vF(\vz^*)\|^2_F}{L_1^2}+\frac{L_2^2\|\vz_0-\vz^*\|^2}{\mu L_1}) = {\bigO ( d + \frac{L_2^2\|\vz_0-\vz^*\|^2}{\mu L_1} )}$ \blue{with $L_2$ being the Jacobian's Lipschitz constant}. 
As we observe from \eqref{eq:rate}, when $k = \bigO(d)$, the first upper bound in \eqref{eq:rate} implies that the convergence rate of our method matches that of the classical extragradient (EG) method~\cite{Korpelevich1976,nemirovski2004prox}. This is the best rate we can hope for in this regime, since there are lower bound results~\cite{Nemirovsky1983Problem,zhang2022lower} showing that the complexity bound of $\bigO(\frac{L_1}{\mu}\log\frac{1}{\epsilon})$ achieved by EG is optimal up to constants when the number of iterations $k$ is $\mathcal{O}(d)$. Further, based on the second term in the upper bound, it provably outperforms EG once the number of iterations satisfies $k\geq M= \mathcal{O}(d)$ and attains a superlinear rate of the form $\left(\frac{L_1^2 M}{\mu^2 k}\right)^{k/2}$.
\blue{In addition, when the operator is only monotone, our method achieves a global convergence rate of 
\begin{equation}
  \mathcal{O}\left(\min\left\{\frac{1}{k},\frac{\sqrt{d}}{k^{1.25}}\right\}\right).
\end{equation}
In particular, this implies that our method matches the $\mathcal{O}(\frac{1}{k})$ rate by EG, known to be optimal in the regime where $k = \bigO(d)$ \cite{ouyang2021lower}. Moreover, it converges at a faster rate of $\bigO\left(\frac{\sqrt{d}}{k^{1.25}}\right)$ when $k = \Omega(d^2)$. Finally, for both strongly monotone and monotone settings, we fully characterize the computational cost in terms of the total number of operator evaluations and matrix-vector products (see Theorems~\ref{thm:line_search},~\ref{thm:computational_cost}, and~\ref{thm:computational_cost_monotone}).
}

Our proposed method is 
distinct from classical quasi-Newton methods, such as BFGS and Broyden's method, in two key aspects. 
First, our method adopts the hybrid proximal extragradient (HPE) framework~\cite{solodov1999hybrid}, resembling a quasi-Newton approximation of the Newton proximal extragradient method for solving monotone variational inequalities~\cite{monteiro2010complexity}. 
Another notable distinction lies in the update of the Jacobian approximation matrix. Classical quasi-Newton methods adhere to the secant condition while maintaining proximity to the previous approximation matrix. Conversely, our update rule is solely driven by our convergence analysis of the HPE framework. 
Specifically, according to our analysis, a better upper bound on 
the cumulative loss $ \sum_k  \ell_k(\mB_k)$ implies a faster convergence rate for our proposed QNPE method, 
where $\mB_k$ is the Jacobian approximation matrix and $\ell_k : \reals^{d\times d} \rightarrow \reals_+$ is a loss function that in some sense measures the approximation error. As a result, the update of $\mB_k$ boils down to running an online algorithm for solving an \emph{online convex optimization problem} in the space of matrices. 

\blue{Finally, we address the challenge of computational efficiency by presenting a tailored online learning algorithm for the update of $\mB_k$. Note that most online learning algorithms for constrained problems are based on a projection oracle, but in our specific setting, such projections either do not admit a closed-from solution or require expensive eigendecomposition. In contrast, our online learning algorithm utilizes an \textit{approximate separation oracle} that can be efficiently constructed using the classical Lanczos method. Additionally, when the Jacobian of $\vF$ possesses certain structures, such as symmetry or sparsity, our algorithm mirrors the same structure on the Jacobian approximation matrices with no additional cost. This allows us to fully leverage the structural information of the Jacobian matrices to reduce the computational cost. }

\vspace{-0.5mm}

\subsection{Overview}
\vspace{-0.5mm}

In this section, we provide a brief overview of our proposed algorithm. Our QNPE algorithm has a hierarchical structure of three levels. 
\begin{itemize}
\vspace{-1mm}
  \item At the highest level, it is based on the HPE framework~\cite{solodov1999hybrid,monteiro2010complexity}, which can be regarded as an inexact variant of the proximal point method~\cite{Rockafellar1976,Martinet1970}. 
  We fully describe the HPE framework in Section~\ref{sec:quasi_NPE}.
  \item At the intermediate level, the algorithm requires two subroutines in the HPE framework: a line search subroutine for selecting the step size (Section~\ref{subsec:ls}) and a Jacobian approximation update subroutine for selecting the matrix $\mB_k$ (Section~\ref{sec:jacobian_online_learning}). To implement the line search, we adopt the standard backtracking algorithm with a warm start strategy, where the initial step size at the $k$-th iteration is chosen as a multiple of the step size in the previous iteration. In addition, motivated by our convergence analysis, we formulate our Jacobian approximation update as an online learning problem (Section~\ref{subsec:convergence}). We first present a general online learning algorithm based on the approximate separation oracle in Section~\ref{subsec:projection_free_online_learning} and then instantiate it to our specific settings in Sections~\ref{subsec:nonlinear_eqs} and~\ref{subsec:special_cases}. 
  \item Furthermore, the line search subroutine requires implementing a $\mathsf{LinearSolver}$ oracle, while the Jacobian approximation update subroutine requires an approximate separation oracle. Their implementations are discussed in Section~\ref{sec:implementation}.
\end{itemize}

We present our complexity analysis in Section~\ref{sec:complexity}, which is divided into the strongly monotone setting~(Section~\ref {subsec:strongly_monotone}) and the monotone setting~(Section~\ref{subsec:monotone}). Finally, the concluding remarks are presented in Section~\ref{sec:conclusion}.

\vspace{-1mm}
\section{Preliminaries}
\label{sec:prelims}
\vspace{-1mm}

\edef\oldassumption{\the\numexpr\value{assumption}+1}

In this section, we present our assumptions, discuss three special cases of Problem~\eqref{eq:monotone} that are of interest, and introduce the considered measures of suboptimality for our convergence analyses.

\subsection{Assumptions}

\blue{In this paper, we focus on two specific classes of nonlinear equations: (i) strongly monotone and (ii) monotone, formally defined in Assumption~\ref{assum:strong_monotone} and Assumption~\ref{assum:monotone}.}  

\edef\oldassumption{\the\numexpr\value{assumption}+1}

\setcounter{assumption}{0}
\renewcommand{\theassumption}{\oldassumption.\alph{assumption}}

\begin{assumption}\label{assum:strong_monotone}  
The operator $\vF$ is $\mu$-strongly monotone, i.e., $\langle \vF(\vz)-\vF(\vz'), \vz-\vz' \rangle \geq \mu \|\vz-\vz'\|^2$,  $\forall \vz,\vz' \in \reals^d$. 
\end{assumption}

\begin{assumption}\label{assum:monotone}  
  $\vF$ is monotone, i.e., $\langle \vF(\vz)-\vF(\vz'), \vz-\vz' \rangle \geq 0$,  $\forall \vz,\vz' \in \reals^d$. 
\end{assumption}
\let\theassumption\origtheassumption
\addtocounter{assumption}{-1}

Note that under Assumption~\ref{assum:strong_monotone}, Problem \eqref{eq:monotone} has a unique solution, which we denote by $\vz^*$ throughout the paper.
Moreover, Assumption~\ref{assum:monotone} can be regarded as a special case of Assumption~\ref{assum:strong_monotone} by setting $\mu=0$. Therefore, we will sometimes use this observation to present our results in both settings in a united manner, avoiding unnecessary repetition. In addition, we make the following two assumptions that the operator $\vF$ and its Jacobian are Lipschitz. 

\begin{assumption}\label{assum:operator_lips}
The operator $\vF$ is $L_1$-Lipschitz, i.e., $\|\vF(\vz)-\vF(\vz')\| \leq L_1\|\vz-\vz'\|$, $\forall \vz,\vz' \in \reals^d$.  
\end{assumption}

\begin{assumption}\label{assum:jacobian_lips}
  The Jacobian of $\vF$ is $L_2$-Lipschitz, i.e., $\|\nabla \vF(\vz)-\nabla \vF(\vz')\|_{\op} \leq L_2\|\vz-\vz'\|$, $\forall \vz,\vz'\in \reals^d$, where $\|\mA\|_{\op} \mydef \sup_{\|\vz\| = 1} \|\mA \vz\|$. 
\end{assumption}

\begin{remark}
  In our convergence analysis for strongly monotone equations, we can relax Assumption~\ref{assum:jacobian_lips} and   $\nabla \vF$ requires to only satisfy a Lipschitz property at the optimal solution $\vz^*$, i.e., $\|\nabla \vF(\vz)-\nabla \vF(\vz^*)\|_{\op} \leq L_2\|\vz-\vz^*\|$ for any $\vz\in \reals^d$. 
\end{remark}

We remark that Assumptions~\ref{assum:strong_monotone}, ~\ref{assum:monotone} and~\ref{assum:operator_lips} are common in the convergence analysis of first-order methods for solving nonlinear equations and monotone variational inequalities. Moreover, Assumption~\ref{assum:jacobian_lips} is commonly used in the study of quasi-Newton methods~\cite{dennis1977quasi} and second-order methods~\cite{nesterov2006cubic,monteiro2010complexity}. It provides the necessary regularity condition on $\vF$ that enables us to prove a superlinear convergence rate for strongly monotone problems and a faster sublinear rate for monotone problems. 

\subsection{Jacobian structures}\label{subsec:structures}

Assumptions~\ref{assum:strong_monotone},~\ref{assum:monotone} and~\ref{assum:operator_lips} are formulated as properties of the operator $\vF$, but they can also be presented as conditions on the Jacobian $\nabla \vF$, as shown in the next lemma (see \cite[Section 2]{ryu2022large} for the proof).

\begin{lemma}\label{lem:jacobian}
Under Assumption~\ref{assum:strong_monotone}, we have $\frac{1}{2}(\nabla \vF(\vz) + \nabla \vF(\vz)^\top) \succeq \mu \mI$ for any $\vz\in \reals^d$, while under Assumption~\ref{assum:monotone}, we have $\frac{1}{2}(\nabla \vF(\vz) + \nabla \vF(\vz)^\top) \succeq 0$ for any $\vz\in \reals^d$. Further, under Assumption~\ref{assum:operator_lips}, we have $\|\nabla \vF(\vz)\|_{\op} \leq L_1$ for any $\vz\in \reals^d$. 
\end{lemma}
Lemma~\ref{lem:jacobian} provides some properties of $\nabla \vF(\vz)$, and as we discuss in Section~\ref{sec:jacobian_online_learning}, it helps us choose Jacobian approximation matrices. Moreover, in some cases, $\nabla \vF(\vz)$ has an additional specific structure that we want our Jacobian approximation matrix to mimic. This could improve the approximation accuracy and reduce storage requirements. 
We give three such examples below. 

\vspace{2mm}

\myalert{{Minimization.}} As a special instance of Problem \eqref{eq:monotone}, consider the unconstrained minimization problem 
\begin{equation}\label{eq:minimization}
  \min_{\vz \in \reals^d} f(\vz),
\end{equation}
where $f\!:\! \reals^d\! \rightarrow \!\reals$ is a (strongly) convex function. By the first-order optimality condition, the optimal solution of \eqref{eq:minimization} satisfies the nonlinear equation $\nabla f(\vz) = 0$. Thus, in this case, $\vF$ is  the gradient operator, i.e., $\vF(\vz) = \nabla f(\vz)$, and its Jacobian is the Hessian matrix of $f$, i.e., $\nabla \vF(\vz) = \nabla^2 f(\vz)$. In particular, we observe that the Jacobian is a symmetric matrix, which allows us to simplify the conditions in Lemma~\ref{lem:jacobian}. Specifically, under Assumptions~\ref{assum:strong_monotone} and~\ref{assum:operator_lips},  $\nabla \vF$ belongs to the set $\{\mA \in \mathbb{S}^d\!: \mu \mI \preceq\! \mA\! \preceq L\mI\}$. Meanwhile, under Assumptions~\ref{assum:monotone} and~\ref{assum:operator_lips}, it belongs to the set $\{\mA \in \mathbb{S}^d\!: 0 \preceq\! \mA\! \preceq L\mI\}$. As discussed in Section~\ref{subsec:nonlinear_eqs}, in our method, we ensure that the Jacobian approximation matrix stays symmetric and positive semidefinite, which is a key property for the efficient implementation of the subroutines.

\vspace{2mm}
\myalert{Minimax optimization.} Another important instance of Problem \eqref{eq:monotone} is the unconstrained minimax optimization problem 
\begin{equation}\label{eq:minimax}
  \min_{\vx \in \reals^m} \max_{\vy \in \reals^n} f(\vx,\vy),
\end{equation}
where $f$ is (strongly) convex with respect to $\vx$ and (strongly) concave with respect to $\vy$. To reformulate this problem as \eqref{eq:monotone}, we define the saddle point operator as $\vF(\vz) = (\nabla f_{\vx}(\vx,\vy),-\nabla f_{\vy}(\vx,\vy)) \in \reals^{m+n}$, where $\vz = (\vx,\vy) \in \reals^{m+n}$. Indeed, finding a solution of $\vF(\vz)$ is equivalent to finding a stationary point of the minimax problem, which is the optimal solution. In this case, the Jacobian is   
$
    \nabla \vF(\vz) = 
    \begin{bmatrix}
        \nabla^2 f_{\vx\vx}(\vx,\vy) & \nabla^2 f_{\vx \vy}(\vx,\vy) \\
        -\nabla^2 f_{\vx \vy}(\vx,\vy)^\top & -\nabla^2 f_{\vy\vy}(\vx,\vy)
    \end{bmatrix}$.
As noted in~\cite{asl2023j}, $\nabla \vF(\vz)$ exhibits $\mJ$-symmetric structure~\cite{mackey2003structured}, where it is symmetric on the main diagonal blocks and anti-symmetric on the off-diagonal blocks. 
Specifically, for  $
    \mJ = \begin{bmatrix}
        \mI_{m\times m} & 0 \\
        0 & - \mI_{n\times n}
    \end{bmatrix}$,  matrix $\mM$ is $\mJ$-symmetric if $ \mJ\mM$ is symmetric, i.e., $ \mJ \mM = \mM^\top \mJ$. 
Similarly, as discussed in Section~\ref{subsec:nonlinear_eqs}, we can ensure that our Jacobian approximation matrix in our method respects the $\mJ$-symmetry property, simplifying the implementation of some subroutines.  

\vspace{2mm}
\myalert{Sparse nonlinear equations.}
The last special structure that we consider is when the Jacobian $\nabla \vF(\vz)$ is sparse. This naturally arises in several applications of scientific computing~\cite{Broyden1971,chen1984solving} and several works have modified Broyden's methods to leverage this sparsity structure~\cite{Schubert1970,Broyden1971,Martinez2000}.   
Specifically, assume that we know the sparsity pattern of the Jacobian matrices, defined as the set
$
  \Omega = \{(i,j):\,\exists\, \vz \in \reals^{d}\,\mathrm{s.t.}\,[\nabla \vF(\vz)]_{ij} \neq 0, \;1\leq i \leq d,\; 1\leq j \leq d\}$,
where $[\nabla \vF(\vz)]_{ij}$ denotes the $(i,j)$-the entry of the matrix $\nabla \vF(\vz)$. 
Moreover, we say a matrix $\mM\in \reals^{d \times d}$ has the sparsity pattern $\Omega$ if $[\mM]_{ij} = 0$ for all $(i,j) \notin \Omega$ and $i \neq j$. 
As we discuss in Section~\ref{subsec:nonlinear_eqs}, we can incorporate these sparsity constraints on our Jacobian approximation matrix in a seamless way. 
This enables us to reduce the storage requirement as well as the computational cost of the subroutines. 

\subsection{Suboptimality measures}
\label{subsec:gap}
In this section, we discuss our choice for measuring suboptimality, which is needed to characterize the convergence rate of our proposed method. 
In the strongly monotone setting (under Assumption~\ref{assum:strong_monotone}), 
recall that there is a unique solution $\vz^*$ to Problem~\eqref{eq:monotone}. Hence, we measure the suboptimality in terms of distance to $\vz^*$. 

In the monotone setting (under Assumption~\ref{assum:monotone}), a common suboptimality measure is the \emph{weak gap function}~\cite{nesterov2007dual,monteiro2010complexity}, defined as $\gap_{w}(\vz) = \max_{\vz' \in \reals^d } \langle \vF(\vz'),   \vz - \vz' \rangle$. Since $\vF$ is monotone and Lipschitz, it holds that $\gap_w(\vz) \geq 0$ for all $\vz\in \reals^d$ and $\gap_w(\vz) = 0$ if and only if $\vz$ solves the nonlinear equation in~\eqref{eq:monotone}. However, in some cases, the weak gap function could always be infinite except at the solutions of \eqref{eq:monotone}, rendering it vacuous. To remedy this issue, we use the \emph{restricted gap function} 
\begin{equation}
  \gap_{\text{w}}(\vz; \mathcal{D}) = \max_{\vz' \in \mathcal{D}} \langle \vF(\vz'),   \vz - \vz' \rangle,
\end{equation}
where $\mathcal{D}$ is a given compact set in $\reals^d$. It is shown in \cite{nesterov2007dual} that: (i) $\gap_w(\vz; \mathcal{D})\geq 0$ for all $\vz \in \mathcal{D}$. (ii) If $\vz^*$ is a solution of \eqref{eq:monotone} and $\vz^* \in \mathcal{D}$, then $\gap_w(\vz^*; \mathcal{D}) = 0$. (ii) Conversely, if $\gap_w(\vz; \mathcal{D}) = 0$ and $\vz$ is in the interior of $\mathcal{D}$, then $\vz$ is a solution of~\eqref{eq:monotone}. Thus, $\gap_w(\vz; \mathcal{D})$ is a valid merit function when  $\mathcal{D}$ is  sufficiently large. 

In addition, we can use a more customized measure of suboptimality for the special instances discussed in Section~\ref{subsec:structures}. Specifically, for the minimization problem in~\eqref{eq:minimization}, we can consider the function value gap:  
\begin{equation}
  \gap_f(\vz; \vz^*) = f(\vz) - f(\vz^*). 
\end{equation} 
For the minimax problem in~\eqref{eq:minimax}, we consider the restricted primal-dual gap: 
\begin{equation}
  \gap_{\text{pd}}(\vz;\calX\times \calY) = \max_{\vy' \in \calY} f(\vx,\vy') - \min_{\vx'\in \calX} f(\vx',\vy),
\end{equation}
where $\vz = (\vx,\vy)$ and $\calX$ and $\calY$ are given compact sets in $\reals^m$ and $\reals^n$, respectively. 
The following classical lemma plays a key role in our analysis, since it provides an upper bound on the gap at the averaged iterate. 
\begin{lemma}\label{lem:averaging}
  Suppose Assumption~\ref{assum:monotone} holds. 
    Let $\theta_0, \dots, \theta_{T-1} \geq 0$ with $\sum_{t=0}^{T-1} \theta_t = 1$ and let $\vz_0,\dots,\vz_{T-1}\in \reals^d$. Define the averaged iterates as $\bar{\vz}_T = \sum_{t=0}^{T-1} \theta_t \vz_t$. Then: 
    \begin{enumerate}[(i)]
      \item For Problem~\eqref{eq:monotone}, $\gap_{\mathrm{w}}(\vz; \mathcal{D}) \leq \max_{\vz \in \calD}\sum_{t=0}^{T-1} \theta_t \langle \vF(\vz_t), \vz_t-\vz\rangle$. 
      \item For Problem~\eqref{eq:minimization}, $\gap_{f}(\vz; \vz^*) \leq \sum_{t=0}^{T-1} \theta_t \langle \vF(\vz_t), \vz_t-\vz^*\rangle$. 
      \item For Problem~\eqref{eq:minimax}, $\gap_{\mathrm{pd}}(\vz;\calX\times \calY) \leq \max_{\vz \in \calX \times \calY}\sum_{t=0}^{T-1} \theta_t \langle \vF(\vz_t), \vz_t-\vz\rangle$. 
    \end{enumerate} 
\end{lemma}
In Lemma~\ref{lem:averaging}, we observe that the gap functions in the three cases can all be upper bounded by a similar quantity. Thus, to unify these different notions, we define  
\begin{equation}\label{eq:gap}
  \gap(\vz; \mathcal{D}) = \begin{cases}
    \gap_{\mathrm{w}}(\vz; \mathcal{D}), & \text{for Problem~\eqref{eq:monotone};}\\
    \max_{\vz' \in \mathcal{D}}\{f(\vz) - f(\vz')\}, & \text{for Problem~\eqref{eq:minimization}};\\
    \max_{(\vx',\vy') \in \calD} \{f(\vx,\vy') - f(\vx',\vy)\}, & \text{for Problem~\eqref{eq:minimax}}.
  \end{cases}
\end{equation}
Specifically, note that $\gap(\vz; \{\vz^*\})$ reduces to $\gap_f(\vz;\vz^*)$ for Problem~\eqref{eq:minimization} and $\gap(\vz; \calX\times \calY)$ reduces to $\gap_{\text{pd}}(\vz;\calX\times \calY)$ for Problem~\eqref{eq:minimax}. 
Hence, for ease of exposition, in our main theorem, we report our results in terms of $\gap(\vz; \calD)$.  
\section{Quasi-Newton proximal extragradient algorithm}\label{sec:quasi_NPE}

In this section, we present our quasi-Newton proximal extragradient (QNPE) method, which is based on the hybrid proximal extragradient (HPE) framework  \cite{solodov1999hybrid,monteiro2010complexity}. Thus, to lay the groundwork, we first briefly recap the HPE framework. 

\vspace{1mm}
\myalert{HPE framework.} The HPE framework is a principled scheme to approximate the proximal point method with a fast convergence rate for solving variational inequality problems. Its update rule consists of two steps in each iteration. In HPE, given the current iterate $\vz_k$ and the step size $\eta_k$, we first take an \emph{approximate proximal point step} as $\hat{\vz}_{k} \approx \vz_k - \eta_k \vF(\hat{\vz}_{k})$. Specifically, we require $\hat{\vz}_k$ to approximately solve the proximal subproblem $\vz - \vz_k + \eta_k \vF(\vz) = 0$ satisfying the error criterion
\begin{equation}\label{eq:HPE_error_condition}
    \|\hat{\vz}_k - \vz_k + \eta_k \vF(\hat{\vz}_k)\| \leq \alpha \|\hat{\vz}_k-\vz_k\|,%
\end{equation}
where $\alpha \in (0,1)$. After computing $\hat{\vz}_k$, we take an \emph{extragradient step} and compute
\begin{equation}\label{eq:extragradient}
    \vz_{k+1} = \vz_k - \eta_k \vF(\hat{\vz}_k).
\end{equation}
The iteration complexity of HPE was first analyzed in \cite{monteiro2010complexity} in terms of finding a \emph{$(\rho,\epsilon)$-weak solution} of \eqref{eq:monotone}. Specifically,  $\vz$ is a $(\rho,\epsilon)$-weak solution if there exists a vector $\vr \in \reals^d$ such that $\|\vr\| \leq \rho $ and $\sup_{\vz' \in \reals^d}\langle \vF(\vz')-\vr, {\vz} - \vz' \rangle \leq {\epsilon}$; see \cite[Section 3]{monteiro2010complexity} for more discussions on this notion. Under monotonicity of $\vF$, it was shown that the averaged iterate $ \bar{\vz}_k = \sum_{i=0}^{k-1} \hat{\vz}_i / \sum_{i=0}^{k-1} \eta_i$ is a $(\rho,\epsilon)$-weak solution with $\max\{\rho,\epsilon\} = \bigO({1}/{\sum_{i=0}^{k-1} \eta_i})$.

\vspace{1mm}
\myalert{Newton Proximal Extragradient.}
The HPE method should be regarded as a conceptual algorithmic framework. Indeed, to turn it into an implementable algorithm, we need to specify how to compute $\hat{\vz}_k$ such that the condition in \eqref{eq:HPE_error_condition} is satisfied. Assuming access of the operator $\vF$ and the Jacobian $\nabla \vF$, \cite{solodov1999hybrid} proposed using a single iteration of Newton's method to solve the proximal subproblem. A variant of this method, known as Newton proximal extragradient (NPE) method, was analyzed by~\cite{monteiro2010complexity,monteiro2012iteration} for monotone settings. 
Specifically, in the first step, the NPE method chooses a step size $\eta_k$ and an iterate $\hat{\vz}_k$ such that 
\begin{gather}
    \hat{\vz}_k = \vz_k -\eta_k(\mI + \eta_k \nabla \vF({\vz}_k))^{-1} \vF(\vz_k), \label{eq:Newton_step}
    \\
    \alpha' \leq \frac{L_2}{2}\eta_k \|\hat{\vz}_k-\vz_k\| \leq  \alpha ,\label{eq:large_step_size_condition}
\end{gather} 
 where  $0 < \alpha' < \alpha < 1$. Note that \eqref{eq:Newton_step} corresponds to applying one iteration of Newton's method on the proximal subproblem $\vz - \vz_k + \eta_k \vF(\vz) = 0$, while \eqref{eq:large_step_size_condition} 
imposes that the step size $\eta_k$ should be on the same order as $\frac{2}{L_2 \|\hat{\vz}_k-\vz_k\|}$. 
Together, the conditions in \eqref{eq:Newton_step} and \eqref{eq:large_step_size_condition} ensure that the error criterion \eqref{eq:HPE_error_condition} is satisfied, and thus NPE is an instance of HPE. %
Using the convergence theory developed for HPE,
\cite{monteiro2010complexity} showed  a convergence rate of $\bigO(1/k^{1.5})$ for NPE when $\vF$ is monotone. %

\begin{algorithm}[t!]\small
  \caption{Quasi-Newton Proximal Extragradient (QNPE) Method (informal)}\label{alg:Full_Equasi-Newton}
    \begin{algorithmic}[1]
        \STATE \textbf{Input:} strong monotonicity parameter $\mu \geq 0$, Jacobian feasible set $\mathcal{Z}$, line search parameters $\alpha_1\geq 0$ and $\alpha_2>0$ such that $\alpha_1+\alpha_2 < 1$,  and initial trial step size $\sigma_0>0$ 
        \STATE \textbf{Initialization:} initial point $\vz_0\in \mathbb{R}^d$ and initial Jacobian approximation $\mB_0\in \mathcal{Z}$ %
        \FOR{iteration $k=0,\ldots,N-1$}
        \STATE Let $\eta_k$ be the largest possible step size in $\{\sigma_k\beta^i:i\geq 0\}$ such that  \tikzmark{top} 
        \begin{gather*}
          \hspace{-8em}\|{\hat{\vz}_k}-\vz_k+{\eta_k}(\vF(\vz_k)+\mB_k({\hat{\vz}_k}-\vz_k))\| \leq \alpha_1 \sqrt{1+\eta_k \mu} \|\hat{\vz}_k-\vz_k\|, \tikzmark{right} %
          \\
          \hspace{-8em}\|\hat{\vz}_k - \vz_k + \eta_k \vF(\hat{\vz}_k)\| \leq { (\alpha_1+\alpha_2)\sqrt{1+\eta_k \mu} \|{\hat{\vz}_k}-\vz_k\|}
      \end{gather*}%

        \STATE Set $\sigma_{k+1} \leftarrow \eta_{k}/\beta$ \tikzmark{bottom}
        \STATE Update 
        $
          \vz_{k+1} \leftarrow \theta_k (\vz_k - \eta_k \vF(\hat{\vz}_k))+(1-\theta_k) \hat{\vz}_k
        $, where $\theta_k = \frac{1}{1+2\eta_k\mu}$
        \IF[Line search accepted the initial trial step size]{$\eta_k = \sigma_k$ \tikzmark{top2}}
        \STATE Set $\mB_{k+1} \leftarrow \mB_k$ \label{line:Hessian_approx_unchanged}
        \ELSE[Line search bactracked]
        \STATE Let $\tilde{\vz}_k$ be the last rejected iterate in the line search  %
        \STATE Set $\vu_k \leftarrow \vF(\tilde{\vz}_k) -\vF({\vz_k})$, $\vs_k \leftarrow \tilde{\vz}_k-\vz_k$
        \STATE Define the loss function $\ell_k(\mB) = \frac{\|\vu_k-\mB\vs_k\|^2}{\|\vs_k\|^2}$ 
        \STATE Feed $\ell_k(\mB)$ to an online learning algorithm $\mathcal{A}$ and obtain $\mB_{k+1}$ \quad \tikzmark{right2}
        \ENDIF \tikzmark{bottom2}
        \ENDFOR
                 
    \end{algorithmic}
    \AddNote{top}{bottom}{right}{\hspace{-.5em}\color{comment}\textit{\quad Line search subroutine;\\
             \hspace{.2em} see Section~\ref{subsec:ls}}}
    \AddNote{top2}{bottom2}{right2}{\hspace{-.5em}\color{comment}\textit{\quad Jacobian approximation\\ \hspace{.2em} update subroutine; see \\
    \hspace{.2em} Section~\ref{sec:jacobian_online_learning}}}
  \end{algorithm}

  \vspace{1mm}
\myalert{Quasi-Newton Proximal Extragradient.} 
In our setting of interest, computing $\nabla \vF(\cdot)$ is not feasible, preventing us from deploying NPE. Instead, we introduce our quasi-Newton proximal extragradient (QNPE) method, which, unlike NPE, does not rely on access to $\nabla \vF(\cdot)$. Specifically, in the initial step of HPE, we simply perform a single \emph{quasi-Newton} iteration on the proximal subproblem. 
Besides the fact that QNPE does not require access to $\nabla \vF(\cdot)$ and only relies on $\vF(\cdot)$, we should emphasize two other key differences from NPE. First, the analysis of NPE in \cite{monteiro2010complexity,monteiro2012iteration} is done for the setting where $\vF$ is merely monotone. When we introduce our QNPE method for the strongly monotone setting, we not only replace the $\nabla \vF(\cdot)$ with its quasi-Newton approximation but also modify the update rules in both the first and second stages. These changes are made to leverage the strong monotonicity of $\vF$. 
Second, the Newton update in \eqref{eq:Newton_step} requires the inverse of a $d \times d$ matrix, with a complexity of $\bigO(d^3)$. In comparison, we allow the linear system of equations arising from the quasi-Newton step to be solved inexactly up to some prescribed accuracy, thus reducing the computational cost and leading to an algorithm with a cost per iteration of $\bigO(d^2)$. 

Now we formally describe the procedure of our QNPE method, consisting of three stages. %
In the \textbf{first stage}, given a Jacobian approximation matrix $\mB_k$ and the iterate~$\vz_k$, we select a step size $\eta_k$ and a point $\hat{\vz}_k$ that satisfy the following conditions:
\begin{align}
    \|{\hat{\vz}_k}-\vz_k+{\eta_k}(\vF(\vz_k)+\mB_k({\hat{\vz}_k}-\vz_k))\| &\leq {\alpha_1} \sqrt{1+\eta_k \mu} \|\hat{\vz}_k-\vz_k\|, \label{eq:inexact_linear_solver} \\
    \|\hat{\vz}_k - \vz_k + \eta_k \vF(\hat{\vz}_k)\| &\leq {(\alpha_1+\alpha_2) \sqrt{1+\eta_k \mu} \|{\hat{\vz}_k}-\vz_k\|} \label{eq:inexact_proximal_point}, 
\end{align}
where $\alpha_1\in [0,1)$ and $\alpha_2 \in (0,1)$ are user-specified parameters with $\alpha_1+\alpha_2<1$.  
When $\alpha_1=0$, the condition in \eqref{eq:inexact_linear_solver} reduces to $\hat{\vz}_k = \vz_k -\eta_k(\mI+\eta_k  \mB_k)^{-1} \vF(\vz_k)$, which corresponds to one iteration of quasi-Newton update on the proximal subproblem. In general, this quasi-Newton step can be performed inexactly and $\alpha_1$ determines the accuracy of solving the resulting linear system of equations. Moreover, unlike NPE that constrains the step size in terms of the displacement $\|\hat{\vz}_k-\vz_k\|$ in \eqref{eq:large_step_size_condition}, we directly impose the condition in \eqref{eq:inexact_proximal_point} to ensure a small error of solving the proximal subproblem. Compared with the error criterion in \eqref{eq:HPE_error_condition} in HPE, we observe that the condition is relaxed by a factor of $\sqrt{1+\eta_k\mu}$. This relaxation, which is inspired by~\cite{barre2022note}, allows us to take a potentially larger step size. 

To find a pair of $(\eta_k, \hat{\vz}_k)$ that satisfies both conditions in \eqref{eq:inexact_linear_solver} and \eqref{eq:inexact_proximal_point}, we propose a backtracking line search scheme. Specifically, given a parameter $\beta \in (0,1)$, we iteratively test the step size from the set $\{\sigma_k \beta^i:i \geq 0\}$, where the initial trial step size $\sigma_k$ is chosen as $\sigma_k = \eta_{k-1}/\beta$ for $k\geq 1$. 
We elaborate on the implementation of our line search scheme in Section~\ref{subsec:ls}. 

In the \textbf{second stage}, %
we compute  $\vz_{k+1}$ via
\begin{equation}\label{eq:extragradient_mixing}
  \vz_{k+1} = \theta_k(\vz_k - \eta_k \vF(\hat{\vz}_k))+(1-\theta_k)\hat{\vz}_k,
\end{equation}
where $\theta_k \in [0,1]$ is chosen based on our convergence analysis. Specifically, in the monotone setting (under Assumption~\ref{assum:monotone}), we choose $\theta_k = 1$ and \eqref{eq:extragradient_mixing} reduces to the extragradient step in \eqref{eq:extragradient}. Moreover, in the strongly monotone setting (under Assumption~\ref{assum:strong_monotone}), we choose $\theta_k  = \frac{1}{1+2\eta_k\mu}$.

Finally, in the \textbf{third stage}, we update $\mB_{k}$, the most important module of our algorithm. Rather than following classical quasi-Newton methods {such as Broyden's method or BFGS}, our update rule of $\mB_k$ is purely motivated by our convergence analysis. 
 Specifically, as we explain in Section~\ref{sec:jacobian_online_learning}, we need to maintain the Jacobian approximation matrix in a certain feasible set $\mathcal{Z} \subset \reals^{d\times d}$, to ensure that $\mB_k$ is properly conditioned. Moreover, the convergence rate of our method is related to the cumulative loss $\sum_{k \in \calB} \ell_k(\mB_k)$, and a smaller cumulative loss implies a faster convergence rate. Here,  $\calB = \{k: \eta_k <\sigma_k\}$ denotes the indices where the line search scheme backtracks and the loss function $\ell_k(\mB_k)$ is given by $\ell_k(\mB_k) = \frac{\|\vu_k-\mB_k\vs_k\|^2}{\|\vs_k\|^2}$, where $\vu_k = \vF(\tilde{\vz}_k) -\vF({\vz_k})$, $\vs_k = \tilde{\vz}_k-\vz_k$,  and $\tilde{\vz}_k$ is an auxiliary iterate returned by the line search scheme. 
Hence, this motivates us to use tools from online learning to minimize the cumulative loss. Specifically, when the line search scheme accepts the initial trial step size $\sigma_k$ (i.e., $k \notin \calB$), the Jacobian approximation matrix remains unchanged since it does not contribute to the cumulative loss. Otherwise, when the line search scheme backtracks, we use a projection-free online learning algorithm to update the matrix $\mB_{k+1}$. 
We will present the details of this procedure in  Section~\ref{sec:jacobian_online_learning}. 

\subsection{Backtracking line search}\label{subsec:ls}

\begin{subroutine}[!t]\small
  \caption{Backtracking line search}\label{alg:ls}
  \begin{algorithmic}[1]
      \STATE \textbf{Input:} iterate $\vz \in \mathbb{R}^d$,  operator $\vg\in \reals^d$,  Jacobian approximation $\mB$, initial trial step size $\sigma>0$
      \STATE \textbf{Parameters:} line search parameters $\beta\in (0,1)$, $\alpha_1\geq 0$ and $\alpha_2>0$ such that $\alpha_1+\alpha_2<1$
      \STATE Set ${\eta}_{+} \leftarrow \sigma$, {$\vs_{+} \leftarrow \mathsf{LinearSolver}(\mI+\eta_{+}\mB, -\eta_{+}\vg; {\alpha_1}\sqrt{1+\eta_+ \mu})$ and $\hat{\vz}_{+} \leftarrow \vz+\vs_{+}$}
      \WHILE{$\|\hat{\vz}_+ - \vz + \eta_+ \vF(\hat{\vz}_+)\| \geq (\alpha_1+\alpha_2)\sqrt{1+\eta_+ \mu} \|{\hat{\vz}_{+}}-\vz\|$\label{line:check_condition}}
        \STATE Set $\tilde{\vz} \leftarrow \hat{\vz}_{+}$ %
          and $\eta_{+} \leftarrow \beta\eta_+ $ \label{line:decreasing_stepsize}
      \STATE Compute $\vs_{+} \leftarrow \mathsf{LinearSolver}(\mI+\eta_{+}\mB, -\eta_{+}\vg; \alpha_1\sqrt{1+\eta_+ \mu})$ and $\hat{\vz}_{+} \leftarrow \vz+\vs_{+}$
      \ENDWHILE
      \IF{$\eta_+ = \sigma$}
      \STATE \textbf{Return} $\eta_{+}$ and  $\hat{\vz}_{+}$ \label{line:return_1}
      \ELSE 
      \STATE \textbf{Return} $\eta_{+}$, $\hat{\vz}_{+}$ and $\tilde{\vz}$ \label{line:return_2}
      \ENDIF
  \end{algorithmic}
\end{subroutine}

Next, we present the backtracking line search scheme for selecting $\eta_k$ and the iterate $\hat{\vz}_k$ in the first stage of QNPE. For brevity, we denote $\vF(\vz_k)$ by $\vg$ and omit the iteration subscript $k$ from both $\vz_k$ and $\mB_k$. %
In light of \eqref{eq:inexact_linear_solver} and \eqref{eq:inexact_proximal_point}, our goal at the $k$-th iteration is to identify a pair $(\eta_{+}, \hat{\vz}_{+})$ satisfying  
\begin{align}
  \|\hat{\vz}_{+}-\vz+{\eta_{+}}(\vg+\mB({\hat{\vz}_{+}}-\vz))\| &\leq  {\alpha_1}\sqrt{1+\eta_+ \mu} \|{\hat{\vz}_{+}}-\vz\|,
  \label{eq:x_plus_update} \\
  \|\hat{\vz}_+ - \vz + \eta_+ \vF(\hat{\vz}_+)\| &\leq (\alpha_1+\alpha_2)\sqrt{1+\eta_+ \mu} \|{\hat{\vz}_{+}}-\vz\|  \label{eq:step size_condition}. 
\end{align}
As previously discussed, with $\eta_{+}$ fixed, the first condition in \eqref{eq:x_plus_update} can be met by solving the linear system $(\mI+\eta_{+}\mB)(\hat{\vz}_+-\vz) = -\eta_+\vg$ to a desired accuracy.
To formalize,  we let 
\begin{equation}\label{eq:linear_solver_update}
  \vs_{+} = \mathsf{LinearSolver}(\mI+\eta_{+}\mB, -\eta_{+}\vg; {\alpha_1}\sqrt{1+\eta_+ \mu}) \quad \text{and} \quad \hat{\vz}_{+} = \vz+\vs_{+},
\end{equation}
where the $\mathsf{LinearSolver}$ oracle is defined as follows. %

\begin{definition}\label{def:linear_solver}
  The oracle $\mathsf{LinearSolver}(\mA,\vb; \rho)$ takes  a matrix $\mA \in \reals^{d\times d}$, a vector $\vb\in \reals^d$ and $\rho>0$ as input, and returns an approximate solution $\vs_{+}$ 
  satisfying $\|\mA\vs_{+}-\vb\| \leq \rho \|\vs_{+}\|$. 
\end{definition}

By Definition~\ref{def:linear_solver}, the pair $(\eta_+,\hat{\vz}_+)$ is guaranteed to satisfy \eqref{eq:x_plus_update} when $\hat{\vz}_+$ is computed from~\eqref{eq:linear_solver_update}. 
Moreover, to implement the  $\mathsf{LinearSolver}(\mA,\vb; \rho)$ oracle, a direct way is to compute the exact solution $\vs_{+} = \mA^{-1}\vb$, but this has a cost of $\bigO(d^3)$. To avoid this, we rely on conjugate gradient-type methods
to inexactly solve the linear system, which only requires computing matrix-vector products. The detailed implementation will be further discussed in Section~\ref{sec:implementation}. 

With the $\mathsf{LinearSolver}$ oracle, we introduce our backtracking line search scheme, detailed in Subroutine~\ref{alg:ls}. We first set $\eta_{+}$ to be the initial trial step size $\sigma$ and then compute $\hat{\vz}_{+}$ from \eqref{eq:linear_solver_update}. If the pair $(\eta_+,\hat{\vz}_+)$ satisfies  \eqref{eq:step size_condition}, we accept $\eta_+$ and $\hat{\vz}_+$ as the final step size and iterate, respectively. Note that both conditions in \eqref{eq:x_plus_update} and \eqref{eq:step size_condition} are indeed satisfied. Otherwise, we multiply the step size by $\beta \in (0,1)$ and repeat the process, until $\eta_+$ and $\hat{\vz}_+$ satisfy \eqref{eq:step size_condition}. We show that this procedure terminates after finite steps (see Appendix~\ref{appen:terminate}), and we characterize the overall computational cost in Theorem~\ref{thm:line_search}. Moreover, in this case, we also return an auxiliary iterate $\tilde{\vz}$, which is computed from \eqref{eq:linear_solver_update} using the step size $\eta_+/\beta$. In other words, $\tilde{\vz}$ is the last point rejected by our line search scheme before accepting $(\eta_+,\hat{\vz}_+)$. This iterate $\tilde{\vz}$ will be used to establish a lower bound on $\eta_k$, which is further used to define the loss function $\ell_k$ for our online learning algorithm.

\section{Jacobian approximation update via online learning}\label{sec:jacobian_online_learning}

Next, we discuss the update for $\{\mB_k\}_{k\geq 0}$. Our update rule deviates from classical quasi-Newton methods and is guided by our convergence analysis. 
In Section~\ref{subsec:convergence}, we discuss key convergence results of QNPE for monotone and strongly monotone operators. We quantify how the choice of Jacobian approximation matrices impacts our method's convergence rate, turning the update of these matrices into an \textit{online convex optimization problem} over a feasible set of matrices.
Given the complexity of the feasible set and the computational intractability of computing its projection, most common projection-based online learning algorithms are precluded from use. In Section~\ref{subsec:projection_free_online_learning}, we tackle this issue by proposing a projection-free online learning approach inspired by \cite{mhammedi2022efficient}, laying the groundwork for our subsequent Jacobian approximation update.
Later in Sections~\ref{subsec:nonlinear_eqs} and~\ref{subsec:special_cases}, we specialize the online learning algorithm in Section~\ref{subsec:projection_free_online_learning} to different settings, including nonlinear monotone equations, minimization, minimax optimization, and monotone equations with sparse Jacobians. 
\subsection{From convergence rate to online learning}
\label{subsec:convergence}
As the first step, we relate the convergence rate of QNPE to the step size $\eta_k$. Before presenting this result in Proposition~\ref{prop:HPE}, we first introduce the following key lemma, which plays a crucial role in the convergence analysis.  

\begin{lemma}\label{lem:one_step}
  Suppose that the operator $\vF$ in \eqref{eq:monotone} is $\mu$-strongly monotone with $\mu \geq 0$.  
  Let $\{\vz_k\}_{k\geq 0}$ and $\{\hat{\vz}_k\}_{k\geq 0}$ be the iterates generated by Algorithm~\ref{alg:Full_Equasi-Newton}, where $\alpha_1\in [0,1)$, $\alpha_2 \in (0,1)$ and $\alpha_1+\alpha_2 <1$. Then for any $\vz \in \reals^d$ and $k \geq 0$, it holds that
  \begin{equation}\label{eq:one_step}
    \begin{aligned}
      \eta_k \langle \vF(\hat{\vz}_k),\hat{\vz}_k - \vz \rangle &\leq \frac{\|\vz_k\!-\! \vz\|^2}{2} - \frac{1+2\eta_k\mu}{2}\|\vz_{k+1} - \vz\|^2 + \eta_k\mu \|\hat{\vz}_k - \vz\|^2 \\
      & \phantom{{}\leq{}} 
      - \frac{1-\alpha_1-\alpha_2}{2}(\|\hat{\vz}_k - \vz_k\|^2 + \|\hat{\vz}_k-\vz_{k+1}\|^2).
    \end{aligned}
  \end{equation} 

\end{lemma}
\begin{proof}
  To simplify the notation, let $\alpha \triangleq \alpha_1+\alpha_2\in (0,1)$. For any $\vz\in \reals^d$, we have 
\begin{equation}\label{eq:inexact_pp_decomp}
  \eta_k \langle \vF(\hat{\vz}_k), \hat{\vz}_k-\vz \rangle =  \langle \hat{\vz}_k-\vz_k+\eta_k \vF(\hat{\vz}_k), \hat{\vz}_k-\vz \rangle
  +\langle \vz_k-\hat{\vz}_k, \hat{\vz}_k-\vz\rangle. %
\end{equation}
For the first term in \eqref{eq:inexact_pp_decomp}, we can bound it by
\begin{align}
  \langle \hat{\vz}_k-\vz_k+\eta_k \vF(\hat{\vz}_k), \hat{\vz}_k-\vz \rangle &\leq \|\hat{\vz}_k-\vz_k+\eta_k\vF(\hat{\vz}_k)\|\|\hat{\vz}_k-\vz\| \nonumber\\
  &\leq \alpha\sqrt{1+\eta_k\mu}\|\hat{\vz}_k-\vz_k\|\|\hat{\vz}_k-\vz\| \nonumber\\
  &\leq {\frac{\alpha}{2}\|\hat{\vz}_k-\vz_k\|^2+\frac{\alpha(1+\eta_k \mu)}{2}\|\hat{\vz}_k-\vz\|^2}.\label{eq:triangle_ipp}
\end{align}
Here, the first inequality is due to the Cauchy-Schwarz inequality, the second inequality comes from \eqref{eq:inexact_proximal_point}, and the last inequality is a result of Young's inequality. 
Additionally, to handle the second term in \eqref{eq:inexact_pp_decomp}, we apply the three-point equality, yielding:
\begin{equation}\label{eq:three_point}
  \langle \vz_k-\hat{\vz}_k, \hat{\vz}_k-\vz\rangle = \frac{1}{2}\|\vz_k-\vz\|^2-\frac{1}{2}\|\hat{\vz}_k - \vz_k\|^2-\frac{1}{2}\|\hat{\vz}_k-\vz\|^2.
\end{equation}
By combining \eqref{eq:inexact_pp_decomp}, \eqref{eq:triangle_ipp} and \eqref{eq:three_point}, we obtain that    
\begin{equation}\label{eq:intermediate_1}
  \eta_k \langle \vF(\hat{\vz}_k), \hat{\vz}_k-\vz \rangle \leq \frac{\|\vz_k-\vz\|^2 }{2}-\frac{1-\alpha}{2}(\|\hat{\vz}_k - \vz_k\|^2 + \|\hat{\vz}_k-\vz\|^2) + \frac{\alpha\eta_k \mu}{2}\|\hat{\vz}_k-\vz\|^2.
\end{equation}
Furthermore, it follows from the update rule in \eqref{eq:extragradient_mixing} that $\eta_k\vF(\hat{\vz}_k) = \vz_k-\vz_{k+1}+2\eta_k\mu (\hat{\vz}_k-\vz_{k+1})$. Thus, for any $\vz\in \reals^d$, it also holds that
\begin{equation}
  \begin{aligned}\label{eq:intermediate_2}
    \phantom{{}={}}\eta_k \langle \vF(\hat{\vz}_k), \vz_{k+1}-\vz \rangle &= \langle \vz_k-\vz_{k+1}, \vz_{k+1}-\vz \rangle+2\eta_k\mu \langle \hat{\vz}_k-\vz_{k+1}, \vz_{k+1}-\vz\rangle %
    \\
    & = \frac{\|\vz_{k}-\vz\|^2}{2}-\frac{\|\vz_k-\vz_{k+1}\|^2}{2}-\frac{1+2\eta_k\mu}{2}\|\vz_{k+1}-\vz\|^2\\
    & \phantom{{}={}}+ {\eta_k\mu}\|\hat{\vz}_{k}-\vz\|^2-{\eta_k\mu}\|\hat{\vz}_k-\vz_{k+1}\|^2, 
  \end{aligned}
\end{equation}
where we applied the three-point equality twice in the last equality. Hence, by adding the inequality in \eqref{eq:intermediate_1} with $\vz = \vz_{k+1}$ to the inequality in \eqref{eq:intermediate_2}, we obtain 
\begin{equation}\label{eq:linearized_loss}
  \begin{aligned}
    \eta_k \langle \vF(\hat{\vz}_k),\hat{\vz}_k - \vz \rangle 
    &= \eta_k \langle \vF(\hat{\vz}_k),{\vz}_{k+1} - \vz \rangle 
    +\eta_k \langle \vF(\hat{\vz}_k), \hat{\vz}_{k} - \vz_{k+1} \rangle \\
    &\leq \frac{\|\vz_{k}-\vz\|^2}{2}-\bcancel{\frac{\|\vz_k-\vz_{k+1}\|^2}{2}}-\frac{1+2\eta_k\mu}{2}\|\vz_{k+1}-\vz\|^2 \\
    &\phantom{{}={}} + {\eta_k\mu}\|\hat{\vz}_{k}-\vz\|^2-{{\eta_k\mu}\|\hat{\vz}_k-\vz_{k+1}\|^2} + \bcancel{\frac{\|\vz_k-\vz_{k+1}\|^2}{2}}\\
    &\phantom{{}={}} -\frac{1\!-\!\alpha}{2}(\|\hat{\vz}_k \!-\! \vz_k\|^2 \!+\! \|\hat{\vz}_k\!-\!\vz_{k+1}\|^2 )+{\frac{\alpha \eta_k\mu}{2}\|\hat{\vz}_k\!-\!\vz_{k+1}\|^2} \!. %
  \end{aligned}
\end{equation} 
Moreover, since $\alpha <1$, we further have $-{{\eta_k\mu}\|\hat{\vz}_k-\vz_{k+1}\|^2} + \frac{\alpha \eta_k \mu}{2}\|\hat{\vz}_k-\vz_{k+1}\|^2 \leq -\frac{\eta_k\mu}{2}\|\hat{\vz}_k-\vz_{k+1}\|^2 \leq 0$.  Combining this with \eqref{eq:linearized_loss}  and rearranging the terms, we arrive at the desired result in~\eqref{eq:one_step}. 
\end{proof}

Building on Lemma~\ref{lem:one_step}, we obtain the following convergence result for Algorithm~\ref{alg:Full_Equasi-Newton}. 

\begin{proposition}\label{prop:HPE}
  Let $\{\vz_k\}_{k\geq 0}$ and $\{\hat{\vz}_k\}_{k\geq 0}$ be the iterates generated by Algorithm~\ref{alg:Full_Equasi-Newton}, where $\alpha_1\in [0,1)$, $\alpha_2 \in (0,1)$ and $\alpha_1+\alpha_2 <1$. 
  \begin{enumerate}[(a)]
    \item Under Assumption~\ref{assum:strong_monotone}, we have $\|\vz_{k+1}-\vz^*\|^2 \leq \|\vz_k-\vz^*\|^2 (1+2\eta_k \mu)^{-1}$ for any $k \geq 0$. 
    \item Under Assumption~\ref{assum:monotone}, we have  $\|\vz_{k+1}-\vz^*\| \leq \|\vz_k-\vz^*\|$ for any $k \geq 0$. Moreover, define the averaged iterate $\bar{\vz}_N$ by $\bar{\vz}_N = \frac{\sum_{k=0}^{N-1} \eta_k \hat{\vz}_k}{\sum_{k=0}^{N-1} \eta_k}$. Then for any compact set $\calD \subset \reals^d$, we have 
    $\gap(\bar{\vz}_N; \calD) \leq \frac{\max_{\vz \in \calD}\,\|\vz_0-\vz\|^2}{2 \sum_{k=0}^{N-1} \eta_k}$.  
  \end{enumerate}
\end{proposition}

\begin{proof}
Let $\alpha = \alpha_1 + \alpha_2$. First, we prove Part (a). Since $\mF(\vz^*) = 0$ and $\vF$ is $\mu$-strongly monotone by Assumption~\ref{assum:strong_monotone}, it holds that  
\begin{equation}\label{eq:strong_monotone}
\langle \vF(\hat{\vz}_k),\hat{\vz}_k - \vz^* \rangle = \langle \vF(\hat{\vz}_k)-\vF({\vz^*}),\hat{\vz}_k - \vz^* \rangle \geq \mu\|\hat{\vz}_k - \vz^*\|^2.
\end{equation}
By setting $\vz = \vz^*$ in \eqref{eq:one_step} and applying \eqref{eq:strong_monotone}, it yields 
\begin{equation}\label{eq:one_step_strongly}
    0 \leq \frac{1}{2}\|\vz_k- \vz^*\|^2 - \frac{1+2\eta_k\mu}{2}\|\vz_{k+1} - \vz^*\|^2  - \frac{1-\alpha}{2}(\|\hat{\vz}_k - \vz_k\|^2 + \|\hat{\vz}_k-\vz_{k+1}\|^2).
\end{equation} 
Since $\alpha = \alpha_1 + \alpha_2 <1$, we can drop the non-positive last term in \eqref{eq:one_step_strongly} and obtain $0 \leq \frac{1}{2}\|\vz_k- \vz^*\|^2 - \frac{1+2\eta_k\mu}{2}\|\vz_{k+1} - \vz^*\|^2$, which is equivalent to $\|\vz_{k+1} - \vz^*\|^2 \leq \|\vz_k- \vz^*\|^2 (1+2\eta_k\mu)^{-1}$. This completes the proof of Part (a). 

Next, we prove Part (b). Since $\vF$ is monotone by Assumption~\ref{assum:monotone}, we have
  \begin{equation}\label{eq:monotone_star}
    \langle \vF(\hat{\vz}_k),\hat{\vz}_k - \vz^* \rangle = \langle \vF(\hat{\vz}_k)-\vF({\vz^*}),\hat{\vz}_k - \vz^* \rangle \geq 0.
    \end{equation}
    By setting $\vz = \vz^*$ in \eqref{eq:one_step} with $\mu = 0$ and applying \eqref{eq:monotone_star}, it yields 
    \begin{equation}\label{eq:one_step_monotone}
      0 \leq \frac{1}{2}\|\vz_k- \vz^*\|^2 - \frac{1}{2}\|\vz_{k+1} - \vz^*\|^2  - \frac{1-\alpha}{2}(\|\hat{\vz}_k - \vz_k\|^2 + \|\hat{\vz}_k-\vz_{k+1}\|^2).
    \end{equation}
    Again, since the last term in \eqref{eq:one_step_monotone} is non-positive, this immediately implies $\frac{1}{2}\|\vz_{k+1} - \vz^*\|^2 \leq \frac{1}{2}\|\vz_{k} - \vz^*\|^2$, which is equivalent to $\|\vz_{k+1} - \vz^*\| \leq \|\vz_{k} - \vz^*\|$. Additionally, by dropping the non-positive last term in \eqref{eq:one_step} and noting that $\mu = 0$, we also have 
    $\eta_k \langle \vF(\hat{\vz}_k),\hat{\vz}_k - \vz \rangle \leq \frac{1}{2}\|\vz_k- \vz\|^2 - \frac{1}{2}\|\vz_{k+1} - \vz\|^2$.
    By summing this inequality from $k=0$ to $k=N-1$, we further get 
    $\sum_{k=0}^{N-1} \eta_k \langle \vF(\hat{\vz}_k),\hat{\vz}_k - \vz \rangle \leq \frac{1}{2}\|\vz_0- \vz\|^2 - \frac{1}{2}\|\vz_{N} - \vz\|^2 \leq \frac{1}{2}\|\vz_0- \vz\|^2$. 
    Hence, it follows from Lemma~\ref{lem:averaging} that $\gap(\bar{\vz}_N; \calD) \leq \frac{\max_{\vz \in \calD}\,\|\vz_0-\vz\|^2}{2 \sum_{k=0}^{N-1} \eta_k}$, where $\bar{\vz}_N$ is the averaged iterate given by $\bar{\vz}_N = \frac{\sum_{k=0}^{N-1} \eta_k \hat{\vz}_k}{\sum_{k=0}^{N-1} \eta_k}$.  
\end{proof}

Proposition~\ref{prop:HPE} highlights the role of  $\eta_k$ in the convergence analysis. In particular, in the strongly monotone setting, if~$\eta_k$ tends to infinity as the number of iterations $k$ increases, then $\lim_{k\rightarrow \infty} \frac{\|\vz_{k+1}-\vz^*\|}{\|\vz_{k}-\vz^*\|} = 0$ by Proposition~\ref{prop:HPE}(a), i.e., QNPE converges superlinearly. Similarly, in the monotone setting, a larger step size results in a faster convergence rate according to Proposition~\ref{prop:HPE}(b). That said, note that the step size $\eta_k$ cannot be arbitrarily selected, since it is constrained by the line search conditions in \eqref{eq:inexact_linear_solver} and~\eqref{eq:inexact_proximal_point}. \textit{Intuitively, $\eta_k$ depends on how well our Jacobian approximation matrix $\mB_k$ captures the local curvature of the operator $\vF$}. This intuition is made precise in the following lemma, where we present a lower bound on $\eta_k$. 
Recall that $\calB$ is the set of indices where the line search subroutine backtracks, i.e., $\calB = \{k: \eta_k < \sigma_k\}$. 

\begin{lemma}\label{lem:step size_lb}
    Recall that ${\tilde{\vz}_k}$ is the auxiliary iterate returned by our line search scheme in Subroutine~\ref{alg:ls}. For $k\notin \mathcal{B}$ we have $\eta_k = \sigma_k$, while for $k\in \mathcal{B}$ we have 
     $ \eta_k \!> \!
      {\frac{{\alpha_2} \beta\|{\tilde{\vz}_k}-\vz_k\|}{\|\vF({\tilde{\vz}_k})-\vF(\vz_k)-\mB_k({\tilde{\vz}_k}-\vz_k)\|}}$.
    Moreover, if $\frac{1}{2}(\mB_k + \mB_k^\top) \succeq \frac{1}{2}\mu \mI$, we have $\|{\tilde{\vz}_k}-\vz_k\| \leq \frac{1+\alpha_1}{\beta(1-\alpha_1)} \|{\hat{\vz}_k}-\vz_k\|$. 
\end{lemma}

\begin{proof}
  If $k\notin \mathcal{B}$, as per the definition, the line search scheme adopts the initial trial step size at the $k$-th iteration and thus $\eta_k = \sigma_k$. On the other hand, if $k \in \mathcal{B}$, we go through the backtracking procedure in Subroutine~\ref{alg:ls}.
  Let $\tilde{\vz}_k$ denote the last rejected point in the line search scheme, which is calculated from \eqref{eq:linear_solver_update} using the step size $\tilde{\eta}_k = \eta_k/\beta$.
  This means that the pair $(\tilde{\vz}_k,\tilde{\eta}_k)$ satisfies \eqref{eq:x_plus_update} but not  \eqref{eq:step size_condition}, i.e.,$\|\tilde{\vz}_k-\vz_k+\tilde{\eta}_k(\vF(\vz_k)+\mB_k(\tilde{\vz}_k-\vz))\| \leq  {\alpha_1}\sqrt{1+\tilde{\eta}_k \mu} \|\tilde{\vz}_k-\vz_k\|$ and $\|\tilde{\vz}_k - \vz_k + \tilde{\eta}_k \vF(\tilde{\vz}_k)\| > (\alpha_1+\alpha_2)\sqrt{1+\tilde{\eta}_k \mu} \|{\tilde{\vz}_{k}}-\vz_k\|$. Moreover, note that $\tilde{\eta}_k (\vF(\tilde{\vz}_k) - \vF(\vz_k)-\mB_k({\tilde{\vz}_k}-\vz_k)) = \left(\tilde{\vz}_k - \vz_k + \tilde{\eta}_k \vF(\tilde{\vz}_k)\right) - ({\tilde{\vz}_k}-\vz_k+\tilde{\eta}_k(\vF(\vz_k)+\mB_k({\tilde{\vz}_k}-\vz_k)))$.  
  Thus, it follows from the triangle inequality that $\tilde{\eta}_k \|\vF(\tilde{\vz}_k) - \vF(\vz_k)-\mB_k({\tilde{\vz}_k}-\vz_k)\| > (\alpha_1+\alpha_2)\sqrt{1+\tilde{\eta}_k \mu} \|{\tilde{\vz}_{k}}-\vz_k\| - \alpha_1\sqrt{1+\tilde{\eta}_k \mu} \|{\tilde{\vz}_{k}}-\vz_k\| = \alpha_2\sqrt{1+\tilde{\eta}_k \mu} \|{\tilde{\vz}_{k}}-\vz_k\|$.
  This further implies that 
  $\tilde{\eta}_k > \frac{\alpha_2\sqrt{1+\tilde{\eta}_k \mu} \|{\tilde{\vz}_{k}}-\vz_k\|}{\|\vF(\tilde{\vz}_k) - \vF(\vz_k)-\mB_k({\tilde{\vz}_k}-\vz_k)\|}$. 
    Since $1 + \tilde{\eta}_k \mu \geq 1$, 
    we obtain the first result in Lemma~\ref{lem:step size_lb}.
    To prove the second result in Lemma~\ref{lem:step size_lb}, 
    recall from~\eqref{eq:x_plus_update} that $\hat{\vz}_k$ and $\tilde{\vz}_k$ are inexact solutions of the linear system of equations:  
    \begin{equation*}
      (\mI+{\eta_k}\mB_k)({\vz}-\vz_k) = -\eta_k\vF(\vz_k) \quad \text{and} \quad (\mI+\tilde{\eta}_k\mB_k)({\vz}-\vz_k) = -\tilde{\eta}_k\vF(\vz_k),
    \end{equation*}
    respectively. 
    Define $\hat{\vz}_k^* = \vz_k - {\eta}_k(\mI+ {\eta}_k\mB_k)^{-1} \vF(\vz_k)$ and $\tilde{\vz}_k^* = \vz_k - \tilde{\eta}_k(\mI+ \tilde{\eta}_k\mB_k)^{-1} \vF(\vz_k)$, i.e., the exact solutions of the above linear systems. Since $(\hat{\vz}_k,\eta_k)$ and $(\tilde{\vz}_k,\tilde{\eta}_k)$ satisfy the condition in \eqref{eq:x_plus_update}, we have 
    \begin{align}\label{eq:inexact_condition_linear}
      \|(\mI+\eta_k\mB_k)(\hat{\vz}_k-\hat{\vz}_k^*)\| &\leq \alpha_1 \sqrt{1+\eta_k \mu} \|\hat{\vz}_k - \vz_k\| \\
      \quad \text{and} \quad  \|(\mI+\tilde{\eta}_k\mB_k)(\tilde{\vz}_k-\tilde{\vz}_k^*)\| &\leq \alpha_1 \sqrt{1+\tilde{\eta}_k \mu}  \|\tilde{\vz}_k - \vz_k\|.
    \end{align} 
    We divide the proof of the second result {in Lemma~\ref{lem:step size_lb}} into the following three steps.
      First, we  show that 
        \begin{align}
                    (1-\alpha_1)\|\hat{\vz}_k - \vz_k\| &\leq \|\hat{\vz}_k^*-{\vz}_k\| \leq (1+\alpha_1)\|\hat{\vz}_k - \vz_k\|, \label{eq:relation_exact_inexact1}\\ (1-\alpha_1)\|\tilde{\vz}_k - \vz_k\| &\leq \|\tilde{\vz}_k^*-{\vz}_k\| \leq (1+\alpha_1)\|\tilde{\vz}_k - \vz_k\|.\label{eq:relation_exact_inexact2}
        \end{align}
      In the following, we will only prove \eqref{eq:relation_exact_inexact1}, since the proof of \eqref{eq:relation_exact_inexact2} follows similarly.
      Since we assume $\frac{1}{2}(\mB_k + \mB_k^\top) \succeq \frac{\mu}{2} \mI$, we have $\vv^\top \mB_k \vv \geq \frac{\mu}{2} \|\vv\|^2$ for any $\vv \in \reals^d$. Therefore, by using Cauchy-Schwarz inequality, we get   
      $\|\hat{\vz}_k-\hat{\vz}_k^*\| \|(\mI+\eta_k\mB_k)(\hat{\vz}_k-\hat{\vz}_k^*)\| \geq (\hat{\vz}_k-\hat{\vz}_k^*)^\top (\mI+\eta_k\mB_k)(\hat{\vz}_k-\hat{\vz}_k^*) \geq (1+ \frac{\eta_k \mu}{2})\|\hat{\vz}_k-\hat{\vz}_k^*\|^2$,
      which further implies that $\|(\mI+\eta_k\mB_k)(\hat{\vz}_k-\hat{\vz}_k^*)\| \geq (1+ \frac{1}{2}\eta_k \mu)\|\hat{\vz}_k-\hat{\vz}_k^*\|$. 
      Moreover, since $ \sqrt{1+\eta_k\mu} \leq 1+ \frac{1}{2} \eta_k\mu $, by~\eqref{eq:inexact_condition_linear} we also have $\|(\mI+\eta_k\mB_k)(\hat{\vz}_k-\hat{\vz}_k^*)\| \leq \alpha_1 {(1+ \frac{1}{2}\eta_k \mu)} \|\hat{\vz}_k - \vz_k\|$. 
      Hence, combining these two inequalities, we get $\|\hat{\vz}_k-\hat{\vz}_k^*\| \leq \alpha_1 \|\hat{\vz}_k-\vz_k\|$. It then follows from the triangle inequality that
      \begin{align*}
        \|\hat{\vz}_k^*-\vz_k\| \leq \|\hat{\vz}_k - \vz_k\| + \|\hat{\vz}_k^*-\hat{\vz}_k\| \leq (1+\alpha_1)\|\hat{\vz}_k - \vz_k\|, \\
        \|\hat{\vz}_k^*-\vz_k\| \geq \|\hat{\vz}_k - \vz_k\| -  \|\hat{\vz}_k^*-\hat{\vz}_k\| \geq (1-\alpha_1)\|\hat{\vz}_k - \vz_k\|,
      \end{align*}
      which proves \eqref{eq:relation_exact_inexact1}.
  Next, we show that 
     \begin{equation}\label{eq:relation_displacement}
      \|{\tilde{\vz}^*_k}-\vz_k\|\leq \frac{1}{\beta}\|{\hat{\vz}^*_k}-\vz_k\|.
     \end{equation}
     This follows from \cite[Lemma 7.8]{monteiro2010complexity}. For completeness, we present its proof below. Note that by definition, we have $(\mI+{\eta_k}\mB_k)(\hat{\vz}_k^*-\vz_k) = -\eta_k\vF(\vz_k)$ and $(\mI+\tilde{\eta}_k\mB_k)(\tilde{\vz}_k^*-\vz_k) = -\tilde{\eta}_k\vF(\vz_k)$.
     Hence, we further have 
     \begin{equation*}
      \mB_k(\hat{\vz}_k^* - \tilde{\vz}_k^*) = \frac{1}{\tilde{\eta}_k}(\tilde{\vz}_k^*-\vz_k) - \frac{1}{\eta_k}(\hat{\vz}_k^*-\vz_k) = \frac{1}{\tilde{\eta}_k} (\tilde{\vz}_k^* - \hat{\vz}_k^*) + \left(\frac{1}{\tilde{\eta}_k}-\frac{1}{\eta_k}\right) (\hat{\vz}_k^*-\vz_k).
     \end{equation*}
     Since $(\hat{\vz}_k^* - \tilde{\vz}_k^*)^\top \mB_k(\hat{\vz}_k^* - \tilde{\vz}_k^*) \geq 0$, by taking the inner product with $\hat{\vz}_k^* - \tilde{\vz}_k^*$ on both sides of the above inequality, we obtain
    $\left(\frac{1}{\tilde{\eta}_k}-\frac{1}{\eta_k}\right) (\hat{\vz}_k^*-\vz_k)^\top (\hat{\vz}_k^* - \tilde{\vz}_k^*)  \geq \frac{1}{\tilde{\eta}_k} \|\hat{\vz}_k^* - \tilde{\vz}_k^*\|^2$.
     Since ${\eta}_k = \beta \tilde{\eta}_k$, using the Cauchy-Schwarz inequality, this further leads to 
    $\|\hat{\vz}_k^* - \tilde{\vz}_k^*\|^2 \leq \left(\frac{1}{\beta}-1\right)(\vz_k - \hat{\vz}_k^*)^\top (\hat{\vz}_k^* - \tilde{\vz}_k^*) \leq  \left(\frac{1}{\beta}-1\right) \|\hat{\vz}_k^* - \vz_k\| \|\hat{\vz}_k^* - \tilde{\vz}_k^*\|$.
     Hence, we further have $\|\hat{\vz}_k^* - \tilde{\vz}_k^*\| \leq \left(\frac{1}{\beta}-1\right) \|\hat{\vz}_k^* - \vz_k\|$, which implies that $\|{\tilde{\vz}^*_k}-\vz_k\| \leq \|{\tilde{\vz}^*_k}-\hat{\vz}^*_k\|+\|\hat{\vz}^*_k-\vz_k\| \leq \frac{1}{\beta} \|\hat{\vz}^*_k-\vz_k\|$. 
    This completes the proof of \eqref{eq:relation_displacement}. Finally, by combining \eqref{eq:relation_exact_inexact1}, \eqref{eq:relation_exact_inexact2}, and \eqref{eq:relation_displacement}, it follows that $\|{\tilde{\vz}_k}-\vz_k\| \leq \frac{1}{1-\alpha_1}\|\tilde{\vz}_k^*-{\vz}_k\| \leq \frac{1}{(1-\alpha_1)\beta}\|{\hat{\vz}^*_k}-\vz_k\| \leq \frac{1+\alpha_1}{1-\alpha_1}\|\tilde{\vz}_k^*-{\vz}_k\| \leq \frac{1+\alpha_1}{(1-\alpha_1)\beta}\|{\hat{\vz}_k}-\vz_k\|$. This proves the second result {in Lemma~\ref{lem:step size_lb}}. 
  \end{proof}

Lemma~\ref{lem:step size_lb} demonstrates that the step size $\eta_k$ depends inversely on the relative approximation error $\frac{\|\vF({\tilde{\vz}_k})-\vF(\vz_k)-\mB_k({\tilde{\vz}_k}-\vz_k)\|}{\|\tilde{\vz}_k-\vz_k\|}$. Moreover, a smaller approximation error leads to a larger step size, which in turn implies faster convergence. %
Note that the lower bound in Lemma \ref{lem:step size_lb} for $\eta_k$ is expressed in terms of $\tilde{\vz}_k$, which is not accepted as the actual iterate. Hence, we use the second result in Lemma~\ref{lem:step size_lb} to relate $\|\tilde{\vz}_k-\vz_k\|$ with $\|\hat{\vz}_k-\vz_k\|$.   
Building on Proposition~\ref{prop:HPE} and Lemma~\ref{lem:step size_lb}, we are ready to quantify the relationship between the convergence rate of QNPE and the choice of the Jacobian approximation matrix $\mB_k$. 

To begin with, \blue{in the strongly monotone setting,} by repeatedly applying Proposition~\ref{prop:HPE}(a), we obtain that 
\begin{equation}\label{eq:after_jensen}
   \frac{\|\vz_{N}-\vz^*\|^2}{\|\vz_0-\vz^*\|^2} \leq \prod_{k=0}^{N-1} (1+2\eta_k\mu)^{-1} \leq \biggl(1+{\frac{2\mu N}{\sum_{k=0}^{N-1} 1/\eta_k}}\biggr)^{-N}, 
\end{equation}
where the last inequality is obtained by applying Jensen's inequality to the function $\log(1+\frac{1}{t})$. 
\blue{Similalry, in the monotone setting, it follows from Proposition~\ref{prop:HPE}(b) that  
\begin{equation}\label{eq:after_Cauchy_Schwarz}
  \mathrm{Gap}(\bar{\vz}_N; \vz) \leq \frac{\|\vz_0-\vz\|^2}{2 \sum_{k=0}^{N-1} \eta_k} \leq  \frac{\|\vz_0-\vz\|^2}{2 N^2}\sum_{k=0}^{N-1} \frac{1}{\eta_k}, 
\end{equation}
where the last inequality is due to Cauchy-Schwarz inequality. 
}
Thus, in light of both \eqref{eq:after_jensen} and \eqref{eq:after_Cauchy_Schwarz}, our goal is to establish an upper bound on $\sum_{k=0}^{N-1}1/\eta_k$ and this is achieved in the next lemma. %

\begin{lemma}\label{lem:stepsize_bnd}
Let $\{\eta_k\}_{k=0}^{N-1}$ be the step sizes in Algorithm~\ref{alg:Full_Equasi-Newton} obtained by the line search in Subroutine~\ref{alg:ls}.~Then,
\begin{equation}\label{eq:goal}
     \sum_{k=0}^{N-1}\frac{1}{\eta_k} \leq \frac{1}{(1-\beta)\sigma_0}+ \frac{1}{(1-\beta)\alpha_2\beta} 
     \sqrt{N \sum_{k\in \mathcal{B}} \frac{\|\vu_k-\mB_k\vs_k\|^2}{\|\vs_k\|^2}}, 
\end{equation}
where $\vu_k \triangleq \vF(\tilde{\vz}_k) -\vF({\vz_k})$, $\vs_k \triangleq \tilde{\vz}_k-\vz_k$, and $\tilde{\vz}_k$ is the auxiliary iterate in Subroutine~\ref{alg:ls}. 
\end{lemma}
\begin{proof}
  Let $\vu_k \triangleq \vF(\tilde{\vz}_k) -\vF({\vz_k})$ and $\vs_k \triangleq \tilde{\vz}_k-\vz_k$. 
  In Lemma~\ref{lem:step size_lb}, we showed that $\eta_k = \sigma_k$ if $k\notin \mathcal{B}$ and $\eta_k> \frac{\alpha_2 \beta \|\vs_k\|}{\|\vu_k-\mB_k\vs_k\|}$ otherwise. %
Using the observations above, we can write 
\begin{equation}\label{eq:intermediate_bound}
    \sum_{k=0}^{N-1}\frac{1}{\eta_k} 
  = \sum_{k\notin \mathcal{B}}\frac{1}{\eta_k} + \sum_{k\in \mathcal{B}}\frac{1}{\eta_k} \leq
  \sum_{k\notin \mathcal{B}}\frac{1}{\sigma_k} + \sum_{k\in \mathcal{B}}\frac{1}{\eta_k}
  \leq \frac{1}{\sigma_0} + 
  \beta\sum_{k\notin \mathcal{B},k\geq 1}\frac{1}{\eta_{k-1}}+\sum_{k\in \mathcal{B}}\frac{1}{\eta_k},
\end{equation}
where we used $\sigma_k = \frac{\eta_{k-1}}{\beta}$ for $k\geq 1$ in the last equality. 
Since
$\sum_{k\notin \mathcal{B},k\geq 1}\frac{1}{\eta_{k-1}} 
\leq \sum_{k= 0}^{N-1}\frac{1}{\eta_{k}}$,
rearranging and simplifying the terms in \eqref{eq:intermediate_bound}, we arrive at 
\begin{equation}\label{eq:all_vs_B}
  \sum_{k=0}^{N-1}\frac{1}{\eta_k} \leq \frac{1}{(1-\beta)\sigma_0} + \frac{1}{1-\beta} \sum_{k\in \mathcal{B}}\frac{1}{\eta_k}. 
\end{equation}
Now using $\eta_k > \frac{\alpha_2 \beta \|\vs_k\|}{\|\vu_k-\mB_k\vs_k\|}$ for $k \in \calB$, we further have 
\begin{align}
  \sum_{k \in \calB} \frac{1}{\eta_k} \leq \sum_{k \in \calB} \frac{\|\vu_k-\mB_k\vs_k\|}{\alpha_2 \beta \|\vs_k\|} &\leq \frac{1}{\alpha_2\beta} \sqrt{|\calB| \sum_{k\in \mathcal{B}} \frac{\|\vu_k-\mB_k\vs_k\|^2}{\|\vs_k\|^2}} 
  \leq \frac{1}{\alpha_2\beta} 
  \sqrt{N \sum_{k\in \mathcal{B}} \frac{\|\vu_k-\mB_k\vs_k\|^2}{\|\vs_k\|^2}}, \label{eq:B_bound1}
\end{align}
where the second inequality is due to the generalized mean inequality. 
Finally, the inequality in \eqref{eq:goal} follows from \eqref{eq:all_vs_B} and \eqref{eq:B_bound1}. 
This completes the proof.
\end{proof}

 Now given the above results we observe the connection between the choice of $\{\mB_k\}$ and our method's convergence rate. Specifically, by combining the results in \eqref{eq:after_jensen} and \eqref{eq:after_Cauchy_Schwarz} with Lemma~\ref{lem:stepsize_bnd}, we observe that our goal is to update the Jacobian approximation matrices $\{\mB_k\}$ such that $ \sum_{k\in \mathcal{B}} \frac{\|\vu_k-\mB_k\vs_k\|^2}{\|\vs_k\|^2}$ is as small as possible to achieve the fastest possible convergence rate. 
\blue{Moreover, note that the variables $u_k = \vF(\tilde{\vz}_k) - \vF(\vz_k)$ and $\vs_k = \tilde{\vz}_k - \vz_k$ are determined by Subroutine~\ref{alg:ls} only \emph{after} the matrix $\mB_k$ has been chosen. This implies that the matrix $\mB_k$ must be selected first, after which we can compute the approximation error as $\frac{\|\vu_k - \mB_k \vs_k\|^2}{\|\vs_k\|^2}$. Our key insight is that this exactly fits into the framework of an \emph{online learning problem} by regarding the sum in \eqref{eq:goal} as the cumulative loss incurred by our choice of $\{\mB_k\}_{k \geq 0}$.}

Formally, define the loss function $\ell_k : \reals^{d \times d} \rightarrow \reals$ at iteration $k$ as
\begin{equation}\label{eq:loss_of_jacobian}
  \ell_k(\mB) \triangleq 
  \begin{cases}
    0, & \text{if } k\notin \mathcal{B}, \\
    \frac{\|\vu_k-\mB\vs_k\|^2}{\|\vs_k\|^2}, & \text{otherwise}.
  \end{cases}
\end{equation}
Then we consider an online learning problem as follows: (i) At the beginning of the $k$-th iteration, we commit to a Jacobian approximation matrix $\mB_k$ from a given convex feasible set $\mathcal{Z}$; (ii) We receive the loss function $\ell_k(\mB)$ defined in \eqref{eq:loss_of_jacobian}; (iii) We update our Jacobian approximation matrix to $\mB_{k+1}$.
Bounding the sum in \eqref{eq:goal} is equivalent to constraining the cumulative loss in the preceding online learning problem. Therefore, we are motivated to utilize an online learning algorithm to update~$\{\mB_k\}_{k\geq 0}$.

\textbf{The choice of feasible set} $\mathcal{Z}$.  Before discussing the online learning algorithm, let us address the choice of the feasible set $\mathcal{Z}$ for selecting $\mB_k$. We provide a high-level overview of why the feasible set structure is crucial in constraining $\sum_{k=0}^{N-1} \ell_k(\mB_k)$. 

Specifically, \blue{in the strongly monotone case}, we decompose the cumulative loss as 
$
    \sum_{k=0}^{N-1} \ell_k(\mB_k) = \sum_{k=0}^{N-1} \ell_k(\mH) +(\sum_{k=0}^{N-1} \ell_k(\mB_k) -\sum_{k=0}^{N-1} \ell_k(\mH)  )$,  
where the first part $\sum_{k=0}^{N-1} \ell_k(\mH)$ is the cumulative loss incurred by choosing a fixed matrix $\mH \in \calZ$, and the second part is known as the \emph{regret} with respect to $\mH$ in the online learning literature. By using tools from online learning, we prove a regret bound in the form of $\sum_{k=0}^{N-1} \ell_k(\mB_k) -\sum_{k=0}^{N-1} \ell_k(\mH)   \leq \reg_N(\mH; \calZ)$, \blue{where $\reg_N(\mH; \calZ)$ denotes a function depending on $N$, the matrix $\mH$ and the feasible set $\calZ$}. Thus, by choosing any $\mH \in \mathcal{Z}$, we obtain an upper bound as $\sum_{k=0}^{N-1} \ell_k(\mB_k) \leq \sum_{k=0}^{N-1} \ell_k(\mH) + \reg_N(\mH)$. Intuitively, a natural choice is setting $\mH = \nabla\mF(\vz^*)$, the true Jacobian matrix at the solution $\vz^*$. 
Thus, we require $\nabla \vF(\vz^*) \in \mathcal{Z}$. 
\blue{In the monotone setting, we use a similar approach but decompose the cumulative loss as $
    \sum_{k=0}^{N-1} \ell_k(\mB_k) = \sum_{k=0}^{N-1} \ell_k(\mH_k) +(\sum_{k=0}^{N-1} \ell_k(\mB_k) -\sum_{k=0}^{N-1} \ell_k(\mH_k)  )$, where $\{\mH_k\}_{k=0}^{N-1}$ is an arbitrary sequence of matrices in $\calZ$. In this decomposition, the second part is known as the \emph{dynamic regret} and we can again utilize tools from online learning to prove a regret bound in the form of $\sum_{k=0}^{N-1} \ell_k(\mB_k) -\sum_{k=0}^{N-1} \ell_k(\mH_k)   \leq \dreg_N(\mH_0,\dots,\mH_{N-1}; \calZ)$. It turns out that the proper choice for $\{\mH_k\}$ is to set $\mH_k = \nabla \vF(\vz_k)$ for $k=0,\dots,N-1$, and hence, we require that $\nabla \vF(\vz_k) \in \calZ$ for any $k \geq 0$.    }

Additionally, the set $\mathcal{Z}$ constrains our approximation matrices $\mB_k$, ensuring that their operator norm $\|\mB_k\|_{\op}$ remains bounded as $\bigO(L_1)$. This constraint, as demonstrated later in Lemma~\ref{lem:stepsize_const_bound}, is essential for achieving a convergence rate similar to EG.
 Moreover, we require $\mB_k$ to satisfy $\frac{1}{2}(\mB_k + \mB_k^\top) \succeq \Omega(\mu) \mI$ in the strongly monotone setting or $\frac{1}{2}(\mB_k + \mB_k^\top) \succeq 0$ in the  monotone setting. This condition is pivotal in Lemma~\ref{lem:step size_lb} and ensures our $\mathsf{LinearSolver}$ oracle can be implemented efficiently.

 Finally, when $\nabla \vF(\vz^*)$ possesses a certain structure, we \blue{can include additional constraints in $\mathcal{Z}$ to} enforce the same structure on $\mB_k$, making the subroutines more efficient and reducing storage requirements. For instance, as discussed in Section~\ref{sec:prelims}, \blue{if $\vF$ is derived from a minimization problem in~\eqref{eq:minimization} or a minimax optimization problem in~\eqref{eq:minimax}, then $\nabla \vF(\vz)$ is symmetric or $\mJ$-symmetric (i.e., $\mJ\nabla \vF(\vz) =  \nabla \vF(\vz)^\top \mJ$).} 
 In addition, if $\vF$ has a sparse Jacobian with sparsity pattern $\Omega$, then $[\nabla \vF(\vz)]_{ij} =0$ for all $(i,j)$ that satisfy $(i,j) \notin \Omega$ and $i\neq j$.

Given these points and the discussions in Section~\ref{subsec:structures}, we choose the feasible set~$\calZ$ as follows: 
\begin{enumerate}[(a)]
  \item For the minimization problem in \eqref{eq:minimization}, we choose 
  \begin{equation}\label{eq:feasible_set_min}
    \calZ = \{\mB \in \sS^d: \; \mu \mI \preceq \mB \preceq L_1 \mI,\; \mB \in \calL\},
  \end{equation}
  where $\calL \subset \sS^d$ is a linear subspace. in the general case, we set $\calL = \sS^d$. However, when the Hessians of $f$ exhibit the sparsity pattern $\Omega$, we define $\calL$ as $\calL=\{\mB\in \sS^{d}: [\mB]_{ij} = 0 \text{ for all } (i,j)\notin \Omega \text{ and }i\neq j\}$.
  \item For the minimax problem in \eqref{eq:minimax} and the nonlinear equation in \eqref{eq:monotone}, we choose 
  \begin{equation}\label{eq:feasible_set}
    \mathcal{Z} = \left\{ \mB \in \reals^{d\times d}:\mu \mI \preceq \frac{1}{2}(\mB + \mB^\top )\preceq L_1\mI , \;\|\mB\|_{\op} \leq L_1, \mB \in \mathcal{L} \right\},
  \end{equation}
  where $\mathcal{L} \subset \reals^{d\times d}$ is a linear subspace. For general nonlinear equations $\mathcal{L} = \reals^{d \times d}$, for minimax problems $\mathcal{L} = \{\mB\in \reals^{d\times d}: \mB = \mJ \mB^\top \mJ\}$, and for sparse nonlinear equations $\calL=\{\mB\in \reals^{d\times d}: [\mB]_{ij} = 0 \text{ for all } (i,j)\notin \Omega \text{ and }i\neq j\}$.
\end{enumerate}

Finally, we note that while the feasible set $\mathcal{Z}$ defined in \eqref{eq:feasible_set_min} or \eqref{eq:feasible_set} satisfies all desired properties, it poses a major computational challenge. Specifically, most online learning algorithms, such as projected online gradient descent~\cite{zinkevich2003online}, require performing Euclidean projection onto the feasible set $\mathcal{Z}$, which would be computationally costly due to the structure of $\mathcal{Z}$. To address this, we employ a projection-free online learning algorithm that is built on the approximate separation oracle in \cite{mhammedi2022efficient}, which is described in the next section.

\subsection{Online learning with an approximate separation oracle}
\label{subsec:projection_free_online_learning}

\begin{algorithm}[!t]\small
  \caption{Projection-Free Online Learning}\label{alg:projection_free_online_learning}
  \begin{algorithmic}[1]
      \STATE \textbf{Input:} Initial point $\vw_0 \in  \vspan(\calC)$, the orthogonal projection matrix $\mP$ associated with the subspace $\vspan(\calC)$, step size $\rho>0$, $\delta>0$, radius $R>0$
      \FOR{$t=0,1,\dots T-1$}
      \STATE Query the oracle $(\gamma_t,\vs_t) \leftarrow \mathsf{SEP}_{\calC}(\vw_t;\delta)$
      \IF[Case I: we have $\vw_t\in (1+\delta)\mathcal{C}$]{$\gamma_t \leq 1$}
        \STATE Set $\vx_t \leftarrow \begin{cases} \vw_t & (\textbf{Option I}) \\ \frac{\vw_t}{1+\delta} & (\textbf{Option II}) \end{cases}$ and play the action $\vx_t$ 
        \STATE Receive the loss $\ell_t(\vx_t)$ and the subspace gradient $\vg_t = \mP\nabla \ell_t(\vx_t)$
        \STATE Set $\tilde{\vg}_t \leftarrow \vg_t$
      \ELSE[Case II: we have $\vw_t/\gamma_t\in (1+\delta)\mathcal{C}$] 
        \STATE Set $\vx_t \leftarrow \begin{cases} \frac{\vw_t}{\gamma_t} & (\textbf{Option I}) \\ \frac{\vw_t}{(1+\delta)\gamma_t} & (\textbf{Option II}) \end{cases}$ and play the action $\vx_t$ 
      \STATE Receive the loss $\ell_t(\vx_t)$ and the subspace gradient $\vg_t = \mP\nabla \ell_t(\vx_t)$
      \STATE Set $\tilde{\vg}_t \leftarrow \vg_t+\max\left\{0, - \frac{1}{\gamma_t}\langle \vg_t, \vw_t \rangle\right\} \vs_t$
      \ENDIF
      \STATE Update $\vw_{t+1} \leftarrow \frac{R(\vw_t-\rho \tilde{\vg}_t)}{\max\{\|\vw_t-\rho \tilde{\vg}_t\|_2,R\}} $ 
      \COMMENT{Online projected gradient descent}
      \ENDFOR
  \end{algorithmic}
\end{algorithm}

To set the stage for our Jacobian approximation update algorithm, we take a detour and consider a general online learning problem over a compact feasible set $\calC \subset \reals^D$. For $T$ consecutive rounds $t=0,\dots,T-1$, a learner chooses an action $\vx_t \in \reals^D$ and then observes a convex loss function $\ell_t:\reals^D \rightarrow \reals$. 
\blue{The goal is to minimize either the \emph{static regret} defined by $\mathrm{Reg}_T(\vu; \calC) \triangleq \sum_{t=0}^{T-1} \ell_t(\vx_t) - \sum_{t=0}^{T-1} \ell_t(\vu)$ with a fixed competitor $\vx \in \calC$, or the \emph{dynamic regret} defined by $\dreg_T(\vu_0,\dots,\vu_{T-1}; \calC) \triangleq \sum_{t=0}^{T-1} \ell_t(\vx_t) - \sum_{t=0}^{T-1} \ell_t(\vu_t)$ with a sequence of competitors $\vu_t \in \calC$ for any $t \geq 0$.}
We assume $0\in \mathcal{C}$ without loss of generality. Moreover, we assume that the set $\calC$ is contained in the Euclidean ball $B_R(0):= \{\vw\in \reals^D: \|\vw\|\leq R\}$ for some radius $R>0$.
Most online learning algorithms require projection onto the feasible set $\calC$. However, the projection oracle is computationally intractable in our setting. To address this issue, 
inspired by \cite{mhammedi2022efficient}, we propose a projection-free algorithm that relies on an approximate separation oracle. 
To start, we will define the approximate separation oracle $\SEP_{\cal{C}}$ for a compact set $\calC$. We use $\vspan(\mathcal{C})$ to denote the linear span of $\mathcal{C}$, i.e., $\vspan(\calC) = \left\{\sum_{i=1}^{k}  \alpha_i \vx_i : k \in \mathbb{Z}_+,\;\vx_i \in \calC,\; \alpha_i \in \reals  \right\}$. Intuitively, $\vspan(\calC)$ is the smallest linear subspace that contains the set $\calC$. %

\begin{definition}\label{def:gauge}
  Let $\calC$ be a compact convex set in $\reals^D$ containing the origin. 
  The oracle $\mathsf{SEP}_{\mathcal{C}}(\vw; \delta)$ takes $\vw\!\in\! \vspan(\calC)$ and $\delta>0$ and returns $\gamma>0$ and a vector $\vs\in \vspan(\calC)$ with one of these outcomes:
    \begin{itemize}
      \item Case I: $\gamma \leq 1$, which implies that $\vw \in (1+\delta)\mathcal{C}$;
      \item Case II: $\gamma>1$, which implies that $\vw/\gamma \in (1+\delta)\mathcal{C} \ $ and $\ \langle \vs, \vw-\vx \rangle \geq {\gamma-1},$  $ \forall\vx\in \mathcal{C}$. 
    \end{itemize}
    \end{definition}

  By Definition~\ref{def:gauge}, given an input $\vw \in \vspan(\calC)$, the $\SEP_{\calC}(\vw;\delta)$ has two possible outcomes: either it certifies that $\vw$ is approximately feasible and lies in $(1\!+\!\delta) \calC$ (Case I), or it produces a scaled version of $\vw$ that is in $(1\!+\!\delta)\calC$ and provides a strict separating hyperplane between $\vw$ and $\calC$ (Case II). 
  We note that the requirement of $\vw \in \vspan(\calC)$ is necessary. Otherwise, we have $\vw/\gamma \notin (1+\delta)\calC$ for any $\gamma\in \reals$ and thus neither of the two cases in Definition~\ref{def:gauge} holds. 

Once equipped with the $\SEP_{\mathcal{C}}$ oracle, 
we are ready to present our projection-free online learning algorithm, detailed in Algorithm~\ref{alg:projection_free_online_learning}. We remark that Algorithm~\ref{alg:projection_free_online_learning} has two different options: in \textbf{Option I}, we slightly relax the feasibility requirement and allow the iterates $\{\vx_t\}_{t\geq 0}$ to be in a larger set $(1+\delta)\mathcal{C}$, whereas in \textbf{Option II}, the iterate satisfies $\vx_t \in \calC$ for all $t \geq 0$.  As we shall see later in Section~\ref{sec:complexity}, we will use \textbf{Option I} in the strongly monotone setting and \textbf{Option II} in the monotone setting.  

\vspace{-.2em}
\begin{remark}
  There are several differences between Algorithm 1 in \cite{mhammedi2022efficient} and our presentation here. First,  a standard online learning setup was considered in~\cite{mhammedi2022efficient} where the action $\vx_t$ must be in the feasible set $\mathcal{C}$, while in our setting $\vx_t$ can be chosen from a larger set $(1+\delta)\mathcal{C}$ in \textbf{Option I}. Second, their algorithm relied on an oracle that approximates the gauge function $\gamma_{\mathcal{C}}(\vw) \triangleq \inf\{\lambda \geq 0: \vw \in \lambda \mathcal{C}\}$ and its subgradient, which is further explicitly constructed using a membership oracle. Our oracle in Definition~\ref{def:gauge} is different but related, in the sense that its output $\gamma$ and $\vs$ may also be regarded as an approximation of the gauge function and its subgradient. It is also more general, since in \cite{mhammedi2022efficient} they assume that $\calC$ is full dimensional, i.e., $\vspan(\calC) = \reals^D$.  Finally, we focus on the specific set used in our Jacobian approximation update and offer a more refined regret analysis along with an efficient construction of the oracle.  
 \end{remark}
 \vspace{-.2em}
To shed light on the design of Algorithm~\ref{alg:projection_free_online_learning}, 
we remark that it can be regarded as a black-box reduction that transforms the original online learning problem over the feasible set $\calC$ into an auxiliary online learning problem on the larger set $B_R(0)$. Specifically, the auxiliary online learning problem is defined by surrogate loss functions $\tilde{\ell}_t (\vw) = \langle \tilde{\vg}_t,\vw\rangle$ for $0\leq t \leq T-1$, where $\tilde{\vg}_t$ is the surrogate gradient to be defined later. Instead of updating the iterates $\{\vx_t\}_{t\geq 0}$ in the original online learning problem directly, we will run online projected gradient descent on this auxiliary problem to obtain the iterates $\{\vw_t\}_{t\geq 0}$ (note that the projection onto $B_R(0)$ is easy to compute), and then generate the iterates $\{\vx_t\}_{t\geq 0}$ by calling $\SEP_{\calC}(\vw_t;\delta)$. 
More precisely, we initialize $\vw_0 \in \vspan(\calC)$, and as we shall prove in Lemma~\ref{lem:regret_reduction}, we can guarantee that $\vw_t \in \vspan(\calC)$ for any $t\geq 0$.
 Consider the iterate $\vw_t$ at round $t$. 
Since $\vw_t \in \vspan(\calC)$, $\SEP_{\calC}(\vw_t;\delta)$ is well-defined and let $\gamma_t>0$ and $\vs_t \in \reals^D$ be its output. Now we consider two cases depending on the value of $\gamma_t$. 
\vspace{-.1em}
\begin{enumerate}[(a)]
  \item In Case I where $\gamma_t \leq 1$, we set either $\vx_t = \vw_t$ in \textbf{Option I}, or $\vx_t = \frac{\vw_t}{1+\delta}$ in \textbf{Option II}. Further, denote by $\mP$ the orthogonal projection matrix associated with the subspace $\vspan(\calC)$, and we compute the subspace gradient $\vg_t = \mP\nabla \ell_t(\vx_t) \in \vspan(\calC)$. Then the surrogate gradient is chosen as $\tilde{\vg}_t = \vg_t$.
  \item Otherwise, in Case II where $\gamma_t > 1$, we set either $\vx_t = \frac{\vw_t}{\gamma_t}$ in \textbf{Option I}, or $\vx_t = \frac{\vw_t}{(1+\delta)\gamma_t}$ in \textbf{Option II}. We further compute the subspace gradient $\vg_t = \mP\nabla \ell_t(\vx_t)$ and define the surrogate gradient by $\tilde{\vg}_t = \vg_t + \max\{0, - \frac{1}{\gamma_t}\langle \vg_t,\vw_t \rangle\}\vs_t$.
\end{enumerate}
\vspace{-.1em}
Finally, we update $\vw_{t+1}$ following the standard online projected gradient descent with step size $\rho>0$: 
\vspace{-.1em}
\begin{equation}\label{eq:w_update}
  \vw_{t+1}  = \mathrm{\Pi}_{\mathcal{B}_R(0)}\bigl(\vw_t-\rho \nabla \tilde{\ell}_t(\vw_t)\bigr) = \frac{R}{\max\{\|\vw_t-\rho \tilde{\vg}_t\|_2,R\}} (\vw_t-\rho \tilde{\vg}_t).
\end{equation}
We note that the surrogate loss functions $\{\tilde{\ell}_t(\vw)\}_{t = 0}^{T-1}$ are constructed explicitly to guarantee that the immediate regret $\tilde{\ell}_t(\vw_t) - \tilde{\ell}_t(\vu)= \langle \tilde{\vg}_t, \vw_t- \vu \rangle$ serves as an (approximate) upper bound on $\ell_t(\vx_t) - \ell_t(\vu)$ for any $\vu \in \calC$. As a result, the regret of the original problem can be upper bounded by the regret of the auxiliary problem, which can be further bounded by standard analysis for online projected gradient descent. This is formalized in the following lemma. %

\begin{lemma}\label{lem:regret_reduction}
 Suppose the loss function $\ell_t$ is convex for any $t\geq 0$ and let $\{\vx_t\}_{t=0}^{T-1}$ and $\{\vw_t\}_{t=0}^{T-1}$ be the iterates generated by Algorithm~\ref{alg:projection_free_online_learning}. Then we have $\vw_t \in \vspan(\calC)$ for any $t \geq 0$. Moreover: 
  \begin{enumerate}[(a)]
    \item In \textbf{Option I}, we have $\vx_t \in (1+\delta)\calC$ for $t=0,1,\dots,T-1$. Also, for any $\vu\in \mathcal{C}$, it holds that 
    \begin{equation}
    \ell_t(\vx_t) - \ell_t(\vu) %
    \leq \langle \tilde{\vg}_t, \vw_t-\vu \rangle %
    \leq \frac{\|\vw_t-\vu\|^2_2}{2\rho}-\frac{\|\vw_{t+1}-\vu\|^2_2}{2\rho}+ \frac{\rho}{2}\|\tilde{\vg}_t\|_2^2,\label{eq:regret_reduction} %
\end{equation}
\begin{equation}\label{eq:surrogate_loss}
   \|\tilde{\vg}_t\| \leq \|\vg_t\|+|\langle \vg_t, \vx_t\rangle|\|\vs_t\|.
\end{equation}
\item In \textbf{Option II}, we have $\vx_t \in \calC$ for $t=0,1,\dots,T-1$. Also, for any $\vu\in \mathcal{C}$, it holds that 
  \begin{align}
  \ell_t(\vx_t) - \ell_t(\vu) 
  &\leq \langle \tilde{\vg}_t, \vw_t-\vu \rangle - \delta \langle \vg_t, \vx_t \rangle \label{eq:regret_reduction_II_1}\\
  &\leq \frac{\|\vw_t-\vu\|^2_2}{2\rho}-\frac{\|\vw_{t+1}-\vu\|^2_2}{2\rho}+ \frac{\rho}{2}\|\tilde{\vg}_t\|_2^2 - \delta \langle \vg_t, \vx_t \rangle,\label{eq:regret_reduction_II_2} %
\end{align}
\begin{equation}\label{eq:surrogate_loss_II}
  \|\tilde{\vg}_t\| \leq \|\vg_t\|+(1+\delta)|\langle \vg_t, \vx_t\rangle|\|\vs_t\|.
\end{equation}
  \end{enumerate}
\end{lemma}
\begin{proof}
To begin with, we prove by induction that $\vw_t \in \vspan(\calC)$. By initialization, we have $\vw_0 \in \vspan(\calC)$. Now suppose $\vw_t \in \vspan(\calC)$ for some $t \geq 0$. Since $\vg_t =  \mP\nabla \ell_t(\vx_t) \in \vspan(\calC)$ and $\vs_t \in \vspan(\calC)$ by Definition~\ref{def:gauge}, we obtain that $\tilde{\vg}_t \in \vspan(\calC)$ in both \textbf{Options I} and \textbf{II}. Moreover, since $\vw_{t+1}$ is expressed as a linear combination of $\vw_t$ and $\tilde{\vg}_t$ according to \eqref{eq:w_update}, we obtain that $\vw_{t+1} \in \vspan(\calC)$. Hence, we conclude by induction that $\vw_t \in \vspan(\calC)$ for any $t\geq 0$.  

Next, we consider the result in Part (a) for \textbf{Option I}.  We distinguish two cases depending on the outcome of $\mathsf{SEP}(\vw_t;\delta)$. 
\begin{itemize}
  \item If $\gamma_t \leq 1$,  we have $\vw_t \in (1+\delta) \mathcal{C}$ by Definition~\ref{def:gauge}. Thus, we have $\vx_t=\vw_t\in (1+\delta) \mathcal{C}$ and $\tilde{\vg}_t = \vg_t$, which immediately implies \eqref{eq:surrogate_loss}. Moreover, since $\ell_t$ is convex, we have $\ell_t(\vx_t) - \ell_t(\vu) \leq \langle \nabla \ell_t(\vx_t), \vx_t - \vu \rangle$. Note that both $\vx_t$ and $\vu$ are in $\vspan(\calC)$ and recall that $\mP$ denotes the orthogonal projection matrix associated with $\vspan(\calC)$. Thus, we further have $\vx_t - \vu = \mP(\vx_t - \vu)$, which implies that $\langle \nabla \ell_t(\vx_t), \vx_t - \vu \rangle = \langle \nabla \ell_t(\vx_t), \mP(\vx_t - \vu )\rangle = \langle \mP\nabla \ell_t(\vx_t), \vx_t - \vu \rangle = \langle \vg_t, \vx_t - \vu \rangle$. 
  Therefore, we have $\ell_t(\vx_t) - \ell_t(\vu) \leq \langle \vg_t, \vx_t - \vu \rangle = \langle \tilde{\vg}_t, \vw_t-\vu \rangle$, which proves the first inequality in \eqref{eq:regret_reduction}. 
  \item Otherwise, if $\gamma_t>1$,  we have $\frac{\vw_t}{\gamma_t} \in (1+\delta)\mathcal{C}$ by Definition~\ref{def:gauge} and $\langle \vs_t, \vw_t-\vx \rangle \geq {\gamma_t-1}$,  $ \forall\vx\in \mathcal{C}$. In this case, Algorithm~\ref{alg:projection_free_online_learning} chooses $\vx_t  = \frac{\vw_t}{\gamma_t} \in (1+\delta)\mathcal{C}$ and $\tilde{\vg}_t = \vg_t+\max\{0, -\frac{1}{\gamma_t}\langle \vg_t, \vw_t \rangle\} \vs_t = \vg_t+\max\{0, -\langle \vg_t, \vx_t \rangle\} \vs_t$. To prove the first inequality in \eqref{eq:regret_reduction}, 
  first note that we also have $\ell_t(\vx_t) - \ell_t(\vu) \leq \langle \vg_t, \vx_t - \vu \rangle$ using similar arguments as above. Moreover, for any $\vu\in \mathcal{C}$,
  \begin{align*}
    \langle \tilde{\vg}_t, \vw_t-\vu \rangle &= \langle\vg_t+\max\{0, -\langle \vg_t, \vx_t \rangle\} \vs_t, \vw_t-\vu \rangle \\
    & = \langle \vg_t, \gamma_t\vx_t-\vu \rangle + \max\{0, -\langle \vg_t, \vx_t \rangle\} \langle \vs_t, \vw_t-\vu\rangle \\
    & \geq \langle \vg_t, \vx_t-\vu\rangle + (\gamma_t-1)\langle\vg_t,\vx_t\rangle + (\gamma_t-1)\max\{0, -\langle \vg_t, \vx_t \rangle\}  \\
    & \geq \langle \vg_t, \vx_t-\vu\rangle, 
  \end{align*}
  where we used $\vw_t = \gamma_t \vx_t$ in the second equality and $\langle \vs_t, \vw_t-\vu\rangle \geq \gamma_t -1$ in the first inequality. 
  Also, by the triangle inequality we obtain 
  \begin{equation*}
    \|\tilde{\vg}_t\|_2 =\|\vg_t+\max\{0, -\langle \vg_t, \vx_t \rangle\} \vs_t\|_2 %
    \leq  \|\vg_t\|_2+|\langle \vg_t, \vx_t \rangle|\|\vs_t\|_2,
  \end{equation*}
  which proves \eqref{eq:surrogate_loss}. 
\end{itemize}
Finally, from the update rule of $\vw_{t+1}$ in \eqref{eq:w_update}, for any $\vu\in \mathcal{C} \subset \mathcal{B}_R(0)$, we have 
$
  \langle\vw_t-\rho\tilde{\vg}_t-\vw_{t+1}, \vw_{t+1}-\vu\rangle \geq 0
$. This further implies that 
\begin{align}
  \langle \tilde{\vg}_t, \vw_{t}-\vu \rangle &\leq  \langle \tilde{\vg}_t, \vw_{t}-\vw_{t+1}\rangle+ \frac{1}{\rho}\langle \vw_t-\vw_{t+1}, \vw_{t+1}-\vu \rangle \nonumber\\
  & =  \langle \tilde{\vg}_t, \vw_{t}-\vw_{t+1}\rangle+ \frac{\|\vw_t-\vu\|_2^2}{2\rho}-\frac{\|\vw_{t+1}-\vu\|_2^2}{2\rho}-\frac{\|\vw_t-\vw_{t+1}\|_2^2}{2\rho} \nonumber\\
  &\leq \frac{1}{2\rho}\|\vw_t-\vu\|_2^2-\frac{1}{2\rho}\|\vw_{t+1}-\vu\|_2^2 +\frac{\rho}{2}\|\tilde{\vg}_t\|_2^2, \label{eq:regret_wt}
\end{align}
where we used the Young's inequality $\langle \tilde{\vg}_t, \vw_{t}-\vw_{t+1}\rangle \leq \frac{\|\vw_{t}-\vw_{t+1}\|^2}{2\rho} + \frac{\rho}{2}\|\tilde{\vg}_t\|^2$ in the last inequality. 
This proves the second inequality in \eqref{eq:regret_reduction}. 

\blue{Next, we consider Part (b) for \textbf{Option II}. Similarly, we distinguish two cases depending on the outcome of $\mathsf{SEP}(\vw_t;\delta)$. 
\begin{itemize}
  \item If $\gamma_t \leq 1$,  we have $\vw_t \in (1+\delta) \mathcal{C}$ by Definition~\ref{def:gauge}. According to \textbf{Option II} in Algorithm~\ref{alg:projection_free_online_learning}, we have $\vx_t= \frac{\vw_t}{1+\delta}\in \mathcal{C}$ and $\tilde{\vg}_t = \vg_t$, which immediately implies~\eqref{eq:surrogate_loss_II}. Using similar arguments as in \textbf{Option I}, we have $\ell_t(\vx_t) - \ell_t(\vu) \leq \langle \vg_t, \vx_t - \vu \rangle$. Moreover, since $\vw_t = (1+\delta)\vx_t$, we have $\langle {\vg}_t, \vw_t - \vu \rangle = \langle {\vg}_t, (1+\delta)\vx_t - \vu \rangle = \langle {\vg}_t, \vx_t-\vu \rangle + \delta \langle {\vg}_t, \vx_t\rangle$. Combining these two, we further have $\ell_t(\vx_t) - \ell_t(\vu) \leq \langle \vg_t, \vx_t - \vu \rangle \leq\langle {\vg}_t, \vw_t - \vu \rangle- \delta \langle {\vg}_t, \vx_t\rangle$. This proves the inequality in \eqref{eq:regret_reduction_II_1}. 
  \item Otherwise, if $\gamma_t>1$, By Definition~\ref{def:gauge} we have $ \frac{\vw_t}{\gamma_t} \in (1+\delta)\mathcal{C}$ and $\langle \vs_t, \vw_t-\vx \rangle \geq {\gamma_t-1}$  $ \forall\vx\in \mathcal{C}$. According to \textbf{Option II} in Algorithm~\ref{alg:projection_free_online_learning}, we have $\vx_t  = \frac{\vw_t}{\gamma_t(1+\delta)} \in \mathcal{C}$ and $\tilde{\vg}_t = \vg_t+\max\{0, -\frac{1}{\gamma_t}\langle \vg_t, \vw_t \rangle\} \vs_t = \vg_t+\max\{0, -(1+\delta)\langle \vg_t, \vx_t \rangle\} \vs_t$. 
  Following similar arguments as in \textbf{Option I}, we have $\ell_t(\vx_t) - \ell_t(\vu) \leq \langle \vg_t, \vx_t - \vu \rangle$. Moreover, for any $\vu\in \mathcal{C}$, 
  \begin{align*}
    \langle \tilde{\vg}_t, \vw_t-\vu \rangle &= \langle\vg_t+\max\{0, -(1+\delta)\langle \vg_t, \vx_t \rangle\} \vs_t, \vw_t-\vu \rangle \\
    & = \langle \vg_t, (1+\delta)\gamma_t\vx_t-\vu \rangle + \max\{0, -(1+\delta)\langle \vg_t, \vx_t \rangle\} \langle \vs_t, \vw_t-\vu\rangle \\
    & \geq \langle \vg_t, (1+\delta)\vx_t-\vu\rangle, 
  \end{align*}
  where we used $\langle \vs_t, \vw_t-\vu\rangle \geq \gamma_t -1$ in the first inequality. By rearranging the above inequality, we obtain that $\ell_t(\vx_t) - \ell_t(\vu) \leq \langle \vg_t, \vx_t - \vu \rangle \leq \langle \tilde{\vg}_t, \vw_t-\vu \rangle - \delta \langle \vg_t, \vx_t \rangle$, which proves~\eqref{eq:regret_reduction_II_1}.  
  Also, by the triangle inequality we obtain 
  \begin{equation*}
    \|\tilde{\vg}_t\|_2 =\|\vg_t+\max\{0, -(1+\delta)\langle \vg_t, \vx_t \rangle\} \vs_t\|_2 %
    \leq  \|\vg_t\|_2+(1+\delta)|\langle \vg_t, \vx_t \rangle|\|\vs_t\|_2,
  \end{equation*}
  which proves \eqref{eq:surrogate_loss_II}. 
\end{itemize}
Finally, the second inequality in \eqref{eq:regret_reduction_II_2} can be shown similarly as in \textbf{Option I}. 
 }
\end{proof}

Hence, our remaining task is to construct such an approximate separation oracle for the sets in \eqref{eq:feasible_set_min} and \eqref{eq:feasible_set}, which are given as the intersection of one or two convex compact sets and a linear subspace, each of which is relatively simple. %
Our key insight is that the approximate separation oracle is intersection-friendly, i.e., 
the separation oracle for $\mathcal{C}_1 \cap \mathcal{C}_2$ can be constructed using the individual oracles for $\mathcal{C}_1$ and $\mathcal{C}_2$, along with projection on the linear subspace~$\vspan(\calC_1\cap \calC_2)$. This is formalized in the next lemma. %

\begin{lemma}\label{lem:sep_intersect}
  Assume that $\mathcal{C} = \mathcal{C}_1 \cap \mathcal{C}_2 \cap \mathcal{L}$, where $0 \in \calC_1 \cap \calC_2$ and $\mathcal{L}= \vspan(\calC)$ is a linear subspace. Suppose we have access to $\mathsf{SEP}_{\mathcal{C}_1}(\vw; \delta)$, $\mathsf{SEP}_{\mathcal{C}_2}(\vw; \delta)$ and the orthogonal projection matrix $\mP_{\calL}$ associated with the subspace $\calL$. Then given inputs $\vw \in \calL$ and $\delta>0$, the oracle $\mathsf{SEP}_{\mathcal{C}}(\vw; \delta)$ can be constructed in the following ways: 
  \begin{enumerate}[(a)]
    \item  let $(\gamma_1,\vs_1) = \mathsf{SEP}_{\mathcal{C}_1}(\vw; \delta)$ and $(\gamma_2,\vs_2) = \mathsf{SEP}_{\mathcal{C}_2}(\vw; \delta)$.
    \item Let $i = \argmax_{i \in \{1,2\}} \gamma_i$ and then output $(\gamma,\vs) = (\gamma_i,\mP_{\calL}\vs_i)$.
  \end{enumerate}
\end{lemma}
\begin{proof}
Without loss of generality, we assume $\gamma_1 \geq \gamma_2$ and thus the output is given by $(\gamma,\vs) = (\gamma_1,\mP_{\calL}\vs_1)$. We consider two cases depending on whether $\gamma_1 \leq 1$ or not. 

In the first case where $\gamma = \gamma_1 \leq 1$, we have $\gamma_2 \leq \gamma_1 \leq 1$. According to the definitions of $\SEP_{\calC_1}(\vw;\delta)$ and $ \SEP_{\calC_2}(\vw;\delta)$, we have $ \vw \in (1+\delta)\calC_1$ and $\vw \in (1+\delta)\calC_2$. This implies that $\vw \in (1+\delta) (\calC_1 \cap \calC_2)$. Moreover, since we assume that $\vw \in \calL$, this leads to $ \vw \in (1+\delta) (\calC_1 \cap \calC_2 \cap \calL) = (1+\delta) \calC$. 

In the second case where $ \gamma = \gamma_1 > 1$, according to the definition of $\SEP_{\calC_1}(\vw;\delta)$ we have $\vw/\gamma_1 \in (1+\delta)\calC_1$ and $ \langle \vs_1, \vw-\vx \rangle \geq \gamma_1 - 1$ for all $\vx \in \calC_1$. Moreover, by the definition of $\SEP_{\calC_2}(\vw;\delta)$, we have $\vw/\gamma_2 \in (1+\delta)\calC_2$.
First, we show that $\vw/\gamma \in (1+\delta)\calC$. 
Since $\gamma_2 \leq \gamma_1$, we can write $ \vw/\gamma_1$ as a convex combination of the origin and $\vw/\gamma_2$, and thus it follows from the convexity of $\calC_2$ that $ \vw/\gamma_1 \in (1+\delta) \calC_2$. Similar to the arguments in the first case, we obtain $\vw/\gamma =  \vw/\gamma_1 \in (1+\delta) \calC$. 
Next, we will prove that $\langle \vs, \vw-\vx \rangle \geq \gamma-1$.  Since $\calC_1 \subset \calC$, we have $\vx \in \calC_1$ for any $\vx \in \calC$. Therefore,  it implies that $\langle \vs_1, \vw - \vx \rangle \geq \gamma_1-1 = \gamma-1$. Moreover, since $\vw \in \calL$ and $\vx \in \calL$, we have $\mP_{\calL}(\vw-\vx) = \vw-\vx$. Therefore, we conclude that $\langle \vs, \vw-\vx \rangle = \langle \mP_{\calL}\vs_1, \vw-\vx \rangle = \langle \vs_1, \mP_{\calL}(\vw-\vx) \rangle = \langle \vs_1, \vw-\vx \rangle \geq \gamma_1 -1 = \gamma-1$. 

Combining both cases, we observe that the procedure described in Lemma~\ref{lem:sep_intersect} indeed satisfies the definition in Definition~\ref{def:gauge}. This completes the proof. 
\end{proof}

If $\SEP_{\calC_1}$ and $\SEP_{\calC_2}$ can be computed efficiently and projecting onto $\calL$ is easy, Lemma~\ref{lem:sep_intersect} shows that constructing $\mathsf{SEP}_{\mathcal{C}}$ from these oracles incurs minimal overhead. This is indeed the case for our setting, as discussed in the following sections.

\subsection{Projection-free Jacobian approximation update}
\label{subsec:nonlinear_eqs}

\begin{subroutine}[!t]\small
  \caption{Online Learning Guided Jacobian Approximation Update}\label{alg:jacobian_approx}
  \begin{algorithmic}[1]
      \STATE \textbf{Input:} Initial matrix $\mB_0\in \mathcal{Z}$, the orthogonal projection matrix $\mP$, step size $\rho>0$, $\delta>0$, failure probabilities $\{q_t\}_{t=1}^{T-1}$, Lipschitz constant $L_1>0$ and strong monotonicity parameter $\mu \geq 0$
      \STATE \textbf{Initialize:} set $\mW_0 \leftarrow \frac{1}{L_1}(\mB-(L_1+\mu)\mI)$, $\mG_0 \leftarrow \frac{1}{L_1} \mP\nabla \ell_0(\mB_0)$ and $\tilde{\mG}_0 \leftarrow \mG_0$
      \STATE {Update $\mW_{1} \leftarrow \frac{\sqrt{d}(\mW_0 - \rho \tilde{\mG}_0)}{\max\{\sqrt{d},\|\mW_0 - \rho \tilde{\mG}_0\|_F\}}$ \COMMENT{Projection onto $\mathcal{B}_{\sqrt{d}}(0)$}\vspace{-.5em}
      \FOR{$t=1,\dots,T-1$}
      \STATE Query $(\gamma^{(1)}_t,\mS^{(1)}_t) \leftarrow \mathsf{ExtEvec}(\frac{\mW_t+\mW_t^\top}{2};\delta, \frac{q_t}{2} )$ and $(\gamma^{(2)}_t,\mS^{(2)}_t) \leftarrow \mathsf{MaxSvec}(\mW_t;\delta, \frac{q_t}{2} )$ %
      \STATE Let $i = \argmax_{i\in \{1,2\}} \gamma_t^{(i)}$, and set $\gamma_t \leftarrow \gamma_t^{(i)}$ and $\mS_t \leftarrow \mP \mS_t^{(i)}$
      \IF[Case I]{$\gamma_t \leq 1$}
        \STATE Set $\hat{\mB}_t \leftarrow \begin{cases} \mW_t & \text{if }\mu>0\;(\textbf{Option I})\\ \frac{\mW_t}{1+\delta} & \text{if }\mu=0\; (\textbf{Option II})\end{cases}$ and $\mB_t \leftarrow {L_1}\hat{\mB}_t+{(L_1+\mu)}\mI$
        \STATE Set $\mG_t \leftarrow \frac{1}{L_1}\mP\nabla \ell_t(\mB_t)$ and $\tilde{\mG}_t \leftarrow \mG_t$
      \ELSE[Case II]
        \STATE Set $\hat{\mB}_t \leftarrow \begin{cases} \frac{\mW_t}{\gamma_t} & \text{if }\mu>0\;(\textbf{Option I})\\ \frac{\mW_t}{(1+\delta)\gamma_t} & \text{if }\mu=0\; (\textbf{Option II})\end{cases}$ and $\mB_t \leftarrow {L_1}\hat{\mB}_t+{(L_1+\mu)}\mI$ 
        \STATE Set $\mG_t \leftarrow \frac{1}{L_1}\mP\nabla \ell_t(\mB_t)$ and $\tilde{\mG}_t \leftarrow \mG_t+\max\{0,-\frac{1}{\gamma_t}\langle \mG_t, \mW_t \rangle\} \mS_t$
      \ENDIF
      \STATE %
      Update 
      $\mW_{t+1} \leftarrow \frac{\sqrt{d}(\mW_t - \rho \tilde{\mG}_t)}{\max\{\sqrt{d},\|\mW_t - \rho \tilde{\mG}_t\|_F\}}$ \COMMENT{Projection onto $\mathcal{B}_{\sqrt{d}}(0)$}
      \ENDFOR}
  \end{algorithmic}
\end{subroutine}

Next, we present our online learning algorithm for updating $\{\mB_k\}_{k \geq 0}$. We follow the projection-free online learning algorithm outlined in Section~\ref{subsec:projection_free_online_learning} and apply it to the online learning problem in Section~\ref{subsec:convergence}.
\blue{Moreover, recall that $\calB$ denotes the set of indices where the line search subroutine backtracks, and we have $\ell_k(\mB)=0$ when $k\notin \mathcal{B}$. Thus, we can simply keep $\mB_{k+1}$ unchanged for these iterations (cf. Line~\ref{line:Hessian_approx_unchanged} in Algorithm~\ref{alg:Full_Equasi-Newton}). Suppose $\calB = \{k_0,k_1,\dots,k_{T-1}\}$, where $T \leq N$. With a slight abuse of notation, in the following, we relabel the indices in $\mathcal{B}$ as $t=0,\dots,T-1$.}

The remaining question is how to construct the approximate separation oracle, required by the projection-free scheme and defined in Definition~\ref{def:gauge}, for our specific set $\calZ$.
To start, we consider the nonlinear monotone equation in \eqref{eq:monotone}; we will further discuss the special cases in the next section.  
Recall that the feasible set~$\calZ$ for our online learning problem in Section~\ref{subsec:convergence} is given in \eqref{eq:feasible_set} with a linear subspace $\calL$. In the following, we assume that the orthogonal projection matrix $\mP$ associated with the linear subspace $\calL$ is given. Specifically, for a general nonlinear monotone equation, we have $\calL = \reals^d$ and the matrix $\mP$ is simply the identity matrix. Moreover, for those special cases discussed in Section~\ref{subsec:structures}, the linear subspace is specified at the end of Section~\ref{subsec:convergence} and we will discuss the corresponding projection matrix in Section~\ref{subsec:special_cases}. 
Since the approximate separation oracle in Definition~\ref{def:gauge} requires the feasible set $\calC$ to contain the origin, we translate and rescale $\mB$ via the transform $\hat{\mB} \mydef \frac{1}{L_1}(\mB-(L_1+\mu)\mI)$. We have the following result, which shows that after this one-to-one correspondence, $\hat{\mB}$ lies in a set $\calC$ centered at the origin defined below.    

\begin{lemma}\label{lem:transform}
   Recall the set $\calZ$ defined in \eqref{eq:feasible_set} and suppose that the linear subspace $\calL$ satisfies $\mI \in \calL$.  Define the set $\calC$ as the following: 
   \begin{equation}\label{eq:translated_set}
  {\mathcal{C}} \mydef \left\{\hat{\mB} \in \reals^{d\times d}: -\mI \preceq \frac{1}{2}(\hat{\mB} + \hat{\mB}^\top ) \preceq \mI , \ \|\hat{\mB}\|_{\op} \leq 3, \ \hat{\mB} \in \calL \right\}. 
  \end{equation}
    Let $\hat{\mB} \mydef \frac{1}{L_1}(\mB-(L_1+\mu)\mI)$. If $\mB \in \calZ$, then $\hat{\mB} \in \mathcal{C}$. Conversely, if $\hat{\mB} \in (1+\delta)\mathcal{C}$, then  $\frac{1}{2}(\mB + \mB^\top )\succeq (\mu -L_1\delta) \mI$ and $\|\mB\|_{\op} \leq 4L_1 + \mu+ 3\delta L_1$. 
\end{lemma}
\begin{proof}
  First, we show that $\mB\in \calZ$ implies $\hat{\mB} \in \calC$. By the definition of $\hat{\mB}$, we have  
$\frac{1}{2}(\hat{\mB} + \hat{\mB}^\top ) = \frac{1}{2L_1}(\mB + \mB^\top) - \frac{L_1+\mu}{L_1}\mI \succeq \frac{\mu}{L_1} \mI - \frac{L_1+\mu}{L_1}\mI = -\mI$. 
Moreover, we also have 
$\frac{1}{2}(\hat{\mB} + \hat{\mB}^\top ) \preceq  \frac{1}{2L_1}(\mB + \mB^\top) \preceq \mI$.
Combining these two, we obtain that $ -\mI \preceq \frac{1}{2}(\hat{\mB} + \hat{\mB}^\top ) \preceq \mI$. Additionally, we can bound $\|\hat{\mB}\|_{\op} \leq \frac{1}{L_1} (\|\mB\|_{\op} + (L_1+\mu)) \leq 3$ since $\|\mB\|_{\op} \leq L_1$ and $\mu \leq L_1$. Finally, note that $\hat{\mB}$ is a linear combination of $\mB$ and $\mI$. Since $\mI \in \calL$ and $\mB \in \calL$, we also have $\hat{\mB} \in \calL$. Hence, we conclude that $\hat{\mB} \in \calC$. 

For the other direction, assume $\hat{\mB} \in (1+\delta)\calC$ and note that $ \mB = L_1 \hat{\mB} + (L_1+\mu)\mI$. Since $\hat{\mB} \in (1+\delta)\calC$ implies that $\frac{1}{2}(\hat{\mB} + \hat{\mB}^\top) \succeq -(1+\delta)\mI$, we have 
$\frac{1}{2}(\mB + \mB^\top) \succeq \frac{L_1}{2}(\hat{\mB} + \hat{\mB}^\top) + (L_1 + \mu) \mI \succeq (\mu-L_1 \delta)\mI$.
Finally, $\hat{\mB} \in (1+\delta)\calC$ also implies that $\|\hat{\mB}\|_{\op} \leq 3(1+\delta)$, and thus it holds that $\|\mB\|_{\op}  \leq L_1\|\hat{\mB}\|_{\op} + (L_1+\mu) \leq 4L_1 + \mu+ 3\delta L_1$. This completes the proof.  
\end{proof}

Hence, we will work with the new set $\calC$ in our online learning algorithm. Moreover, note that we can write $\calC = \calC_1 \cap \calC_2 \cap \calL$, where we define 
  $\mathcal{C}_1 \mydef \{\hat{\mB} \in \mathbb{R}^{d\times d}:  -\mI \preceq \frac{1}{2}(\hat{\mB} + \hat{\mB}^\top )\preceq \mI \}$ and $\mathcal{C}_2 \mydef \{\hat{\mB} \in \mathbb{R}^{d\times d}: \|\hat{\mB}\|_{\op} \leq 3 \}$.  
By Lemma~\ref{lem:sep_intersect}, it suffices to construct the approximate separation oracles for $\calC_1$ and $\calC_2$, respectively. 

The separation oracle for $\calC_1$ is closely related to the problem of computing extreme eigenvalues and eigenvectors of a symmetric matrix. Specifically, given an input matrix $\mW \in \reals^{d\times d}$, let $\lambda_{\max} \in \reals$ and $\vv_{\max}\in \reals^d$  be the largest magnitude eigenvalue and its associated unit eigenvector of the symmetrized matrix $\bar{\mW} = \frac{1}{2}(\mW + \mW^\top)$, respectively.
We deduce that: (i) If $|\lambda_{\max}| \leq 1$, then this implies that $ -\mI \preceq \frac{1}{2}(\mW+\mW^\top) \preceq \mI$, which certifies $\mW \in \calC$. (ii) Otherwise, if $|\lambda_{\max}|>1$, then we set $\gamma = |\lambda_{\max}|$ and $\mS = \sign(\lambda_{\max})\vv_{\max}\vv_{\max}^\top \in \mathbb{S}^d$. We claim that the pair $(\gamma,\mS)$ satisfies the conditions in Definition~\ref{def:gauge}. Indeed, note that 
$ -\gamma \mI \preceq \frac{1}{2}(\mW+\mW^\top) \preceq \gamma \mI$
and thus $\frac{\mW}{\gamma} \in \calC$. Moreover, since the matrix $\mS$ is symmetric, it holds that $\langle \mS, \mW\rangle = \langle\mS, \bar{\mW} \rangle =  \sign(\lambda_{\max})\vv_{\max}^\top \bar{\mW} \vv_{\max} = |\lambda_{\max}|$. Similarly, for any $\hat{\mB} \in \calC$, we have $\langle \mS, \hat{\mB}\rangle = \langle \mS, \frac{1}{2}(\hat{\mB}+\hat{\mB}^\top)\rangle  \leq |\frac{1}{2}\vv_{\max}^\top (\hat{\mB}+\hat{\mB}^\top) \vv_{\max}| \leq 1$. 
Combining these two results, we obtain that $\langle \mS, \mW - \hat{\mB} \rangle \geq \gamma-1$. Hence, we can build the separation oracle for ${\calC}_1$ by computing the extreme eigenvalues and eigenvectors of a given symmetric matrix. 
However, computing the exact values of $\lambda_{\max}$ and $\vv_{\max}$ can be costly. Therefore, we propose to employ the randomized Lanczos method~\cite{kuczynski1992estimating} to compute the extreme eigenvalues and the corresponding eigenvectors inexactly. This leads to the randomized oracle, $\mathsf{ExtEvec}$, defined below. We defer its implementation details to Section~\ref{sec:implementation}. 

\begin{definition}\label{def:extevec}
  The oracle $\mathsf{ExtEvec}(\mW;\delta,q)$ takes $\mW \in \reals^{d\times d}$, $\delta>0$, and $q\in (0,1)$ as input and returns a scalar $\gamma>0$ and a matrix $\mS\in \mathbb{S}^d$. With probability at least $1-q$, the returned $\gamma$ and $\mS$ satisfy one of the following properties::
  \begin{itemize}
    \item Case I: $\gamma \leq 1$, which implies that
    $-(1+\delta)\mI \preceq \frac{1}{2}(\mW+\mW^\top) \preceq (1+\delta)\mI$;
    \item Case II: $\gamma>1$, which implies that %
    $-(1+\delta)\mI \preceq \frac{1}{2\gamma}(\mW+\mW^\top) \preceq (1+\delta)\mI$, 
    $\|\mS\|_F \leq 1$ and $\langle \mS,\mW-\hat{\mB}\rangle \geq \gamma -1$ for any $\hat{\mB}$ such that $-\mI \preceq \frac{1}{2}(\hat{\mB}+\hat{\mB}^\top) \preceq \mI$.
  \end{itemize}
\end{definition}
Note that $\mathsf{ExtEvec}$ is an approximate separation oracle for the set $\mathcal{C}_1$ in the sense of Definition~\ref{def:gauge} (with success probability at least $1-q$), and it also guarantees that $\|\mS\|_F \leq 1$ in Case II.

For the second set $\calC_2$, it turns out that it has a close relation to computing the maximum singular value and the corresponding singular vectors of the input matrix $\mW \in \reals^{d\times d}$. 
Specifically, let $\sigma_{\max}>0$ be the maximal singular value of $\mW$ and let $\vv,\vv' \in \reals^d$ be the associated left and right unit singular vectors. Since $\|\hat{\mB}\|_{\op} = \sigma_{\max}$, we can similarly deduce that: (i) If $\sigma_{\max} \leq 3$, then this certifies that ${\mW} \in \calC_2$. (ii) Otherwise, if $\sigma_{\max} >3$, then we set $\gamma = \frac{\sigma_{\max}}{3}$ and $\mS = \frac{1}{3}\vv'\vv^\top \in \reals^{d \times d}$. In this case, note that $\|\frac{\mW}{\gamma}\|_{\op} \leq 3$ and thus $\frac{\mW}{\gamma} \in \calC_2$. Moreover, note that $\langle \mS, \mW \rangle = \frac{1}{3}\vv^\top \mW \vv' = \frac{\sigma_{\max}}{3} = \gamma$ and $\langle \mS, \hat{\mB}\rangle = \frac{1}{3}\vv^\top \hat{\mB} \vv'\leq \frac{1}{3}\|\hat{\mB}\|_{\op} \leq 1$ for any $\hat{\mB} \in \calC_2$, we obtain that  $\langle \mS, \mW - \hat{\mB} \rangle \geq \gamma - 1$. Thus, this demonstrates that the separation oracle for $\calC_2$ can be constructed if we can compute inexactly the maximal singular value and its associated singular vectors for a given matrix $\mW$. 
Similarly, this can be efficiently implemented by using a randomized Lanczos algorithm, which underpins the $\mathsf{MaxSvec}$ oracle defined below. Detailed discussions are deferred to Section~\ref{sec:implementation}. %
\begin{definition}\label{def:MaxSvec}
  The oracle $\mathsf{MaxSvec}(\mW;\delta,q)$ takes $\mW \in \mathbb{R}^{d\times d}$, $\delta>0$, and $q\in (0,1)$ as input and returns a scalar $\gamma>0$ and a matrix $\mS\in \mathbb{R}^{d\times d}$. With probability at least $1-q$, the returned $\gamma$ and $\mS$ satisfy one of the following properties::
  \begin{itemize}
    \item Case I: $\gamma \leq 1$, which implies that $ \|\mW\|_{\op} \leq 3(1+\delta) $;
    \item Case II: $\gamma>1$, which implies that $ \|\mW/\gamma \|_{\op} \leq 3(1+\delta)$, $\|\mS\|_F \leq 1$ and $\langle \mS,\mW-\hat{\mB}\rangle \geq \gamma -1$ for any $\hat{\mB} \in \reals^{d\times d}$ such that $ \|\hat{\mB}\|_{\op}  \leq 3 $.
  \end{itemize}
\end{definition}
Similarly, we remark that $\mathsf{MaxSvec}$ is an approximate separation oracle for the set $\mathcal{C}_2$ in the sense of Definition~\ref{def:gauge} (with success probability at least $1-q$), and it also guarantees that $\|\mS\|_F \leq 1$ in Case II.

The oracles in Definitions~\ref{def:extevec} and~\ref{def:MaxSvec} provide the approximate separation oracles for $\calC_1$ and $\calC_2$, respectively. Using these two building blocks, we can construct the approximate separation oracle for $\calC$ by following the procedure in Lemma~\ref{lem:sep_intersect}. 
By instantiating Algorithm~\ref{alg:projection_free_online_learning}, we obtain the complete Jacobian approximation update given in Subroutine~\ref{alg:jacobian_approx}.

\subsection{Special cases: minimization, minimax optimization, and sparse nonlinear equations}
\label{subsec:special_cases}

Next, we explore the three special cases from Section~\ref{sec:prelims}, where $\nabla \vF$ exhibits additional structures. We then discuss modifications to our online learning algorithm to enforce these structures on the Jacobian approximation $\{\mB_k\}_{k\geq 0}$.
Specifically, these additional structures amount to choosing different linear subspaces $\calL$ in the feasible set defined in \eqref{eq:feasible_set}. Moreover, since they all satisfy the condition that $\mI \in \calL$, by Lemma~\ref{lem:transform}, the feasible set $\calC$ in our online learning algorithm is given by $\calC = \calC_1 \cap \calC_2 \cap \calL$  under the transformation $\hat{\mB} \mydef \frac{1}{L_1}(\mB-(L_1+\mu)\mI)$. We have discussed the construction of $\SEP_{\calC_1}$ and $\SEP_{\calC_2}$ in Section~\ref{subsec:nonlinear_eqs}, and 
according to Lemma~\ref{lem:sep_intersect}, we only need to specify how to project onto the linear subspace $\calL$.

\myalert{\blue{Minimization.}}
\blue{Consider the minimization problem in \eqref{eq:minimization}. In this case,
we have $\calL = \mathbb{S}^d$, the set of symmetric matrices. For any $\mW \in \reals^{d\times d}$, the projection onto $\calL$ can be computed as $\mP_{\calL}(\mW) = \frac{1}{2}\left(\mW+\mW^\top\right)$. In fact, note that 
since the Jacobians are symmetric, the feasible set $\calZ$ for our online learning problem in Section~\ref{subsec:convergence} can be simplified as in \eqref{eq:feasible_set_min}. 
Moreover, the resulting compact set $\calC$ in \eqref{eq:translated_set} can be simplified as ${\mathcal{C}} \mydef \left\{\hat{\mB} \in \mathbb{S}^d: -\mI \preceq \hat{\mB}\preceq \mI \right\}$. Therefore, the approximate separation oracle for $\calC$ can be directly given by $\mathsf{ExtEvec}$ in Definition~\ref{def:extevec} without relying on Lemma~\ref{lem:sep_intersect}. %
}

\myalert{Minimax optimization.}
In this case, we have $\calL = \{\hat{\mB}\in \reals^{d\times d}: \mJ{\hat\mB} =  \hat{\mB}^\top \mJ\}$, i.e, the set of $\mJ$-symmetric matrices. 
For any $\mW \in \reals^{d\times d}$, the projection onto the set of $\mJ$-symmetric matrices can be easily computed as $\mP_{\calL}(\mW) = \frac{1}{2}(\mW + \mJ \mW^\top \mJ)$. 

\myalert{Sparse nonlinear equations.}
In this case, we have $\calL = \{\mB \in \reals^{d\times d}: [\mB]_{ij} = 0,\;\forall (i,j)\notin \Omega, i\neq j\}$. Given an input matrix $\mW$, the projection onto $\calL$ corresponds to ``zeroing out'' the entries of $\mW$ that are not in $\Omega$. Formally, we have $\mP_{\calL}(\mW) = \tilde{\mW}$, where the matrix $\tilde{\mW}$ is defined as $[\tilde{\mW}]_{ij} = [\mW]_{ij}$ for $(i,j) \in \Omega$ or $ i=j$, and $[\tilde{\mW}]_{ij} =0$ otherwise.

\section{Implementation details of the oracles}
\label{sec:implementation}
By now, we have fully described our proposed QNPE method in Algorithm~\ref{alg:Full_Equasi-Newton}, except for the $\mathsf{LinearSolver}$ oracle, which is required in Subroutine~\ref{alg:ls}, and the $\mathsf{ExtEvec}$ and $\mathsf{MaxSvec}$ oracles, which are required in Subroutine~\ref{alg:jacobian_approx}. In this section, we close these gaps and fully discuss the implementation details of these oracles. 

\begin{subroutine}[!t]\small
  \caption{$\mathsf{LinearSolver}(\mA,\vb; \rho)$}\label{alg:CGLS}
  \begin{algorithmic}[1]
      \STATE \textbf{Input:} $\mA \in \reals^{d\times d}$, $\vb\in \reals^d$, $\rho> 0$
      \STATE \textbf{Initialize:} $\vs_0 \leftarrow 0$, $\vr_0 \leftarrow \vb$, $\vv_0 \leftarrow \mA^\top \vr_0$, $\vp_0 \leftarrow 
      \begin{cases}
        \vv_0  & \text{if }\mA\text{ is non-symmetric}\\
        \vr_0 & \text{if }\mA\text{ is symmetric} 
      \end{cases}$, 
      $\gamma_0 \leftarrow \vv_0^\top \vp_0$, $\vq_0 \leftarrow \vv_0$ if $\mA$ is symmetric
      \FOR{$k=0,\dots$}
      \IF{$\|\vr_k\|_2\leq {\rho}\|\vs_k\|_2$}
        \STATE \textbf{Return} $\vs_k$
      \ENDIF
      \STATE When $\mA$ is non-symmetric: \hspace{7.5em} When $\mA$ is symmetric:\\[-.5em]
      \begin{minipage}[t]{0.45\columnwidth}
        \STATE $\vq_k \leftarrow \mA \vp_k$ \tikzmark{top} \label{line:first_mvp}
      \STATE $ \alpha_k \leftarrow \gamma_{k}/\|\vq_k\|^2$
      \STATE $\vs_{k+1} \leftarrow \vs_{k}+\alpha_k\vp_k$
      \STATE $\vr_{k+1}\leftarrow \vr_{k}-\alpha_k\vq_k$
      \STATE $\vv_{k+1} \leftarrow \mA^\top \vr_{k+1}$ \label{line:second_mvp}
      \STATE \blue{$\gamma_{k+1} \leftarrow \|\vv_{k+1}\|^2$}
      \STATE $\beta_k \leftarrow \gamma_{k+1}/\gamma_{k}$
      \STATE \blue{$\vp_{k+1} \leftarrow \vv_{k+1} + \beta_k \vp_k$} \tikzmark{bottom} \tikzmark{right}
      \end{minipage}
      \begin{minipage}[t]{0.45\columnwidth}
        \makeatletter
        \setcounter{ALC@line}{7}
        \makeatother
        \STATE ($\vq_k = \mA \vp_k$) \tikzmark{top2}
      \STATE $ \alpha_k \leftarrow \gamma_{k}/\|\vq_k\|^2$
      \STATE $\vs_{k+1} \leftarrow \vs_{k}+\alpha_k\vp_k$
      \STATE $\vr_{k+1}\leftarrow \vr_{k}-\alpha_k\vq_k$
      \STATE $\vv_{k+1} \leftarrow \mA \vr_{k+1}$
      \STATE \blue{$\gamma_{k+1} \leftarrow \vv_{k+1}^\top \vr_{k+1}$}
      \STATE $\beta_k \leftarrow \gamma_{k+1}/\gamma_{k}$
      \STATE \blue{$\vp_{k+1} \leftarrow \vr_{k+1} + \beta_k \vp_k$} 
      \STATE \blue{$\vq_{k+1} = \vv_{k+1} + \beta_k \vq_{k}$} \label{line:update_of_q}
      \tikzmark{bottom2} \tikzmark{right2}
      \end{minipage}
      \ENDFOR
  \end{algorithmic}
  \AddNote{top}{bottom}{right}{\color{comment}\textit{\;CGLS}}
  \AddNote{top2}{bottom2}{right2}{\color{comment}\textit{\;Conjugate Residual}}
\end{subroutine}

\myalert{The $\mathsf{LinearSolver}$ oracle}. We first present an implementation of the $\mathsf{LinearSolver}$ oracle in Definition~\ref{def:linear_solver}. At a high level, we propose to run a conjugate gradient-type method to solve the linear system $\mA \vs = \vb$ with the initialization $\vs_0  = 0$, and we return the iterate $\vs_k$ once it satisfies $\|\mA \vs_k - \vb\| \leq \rho \|\vs_k\|$. The specific algorithm we choose depends on whether the input matrix $\mA$ is symmetric or not. Specifically, when $\mA$ is non-symmetric, we adopt the CGLS method~\cite{hestenes1952methods,paige1982lsqr}, which is analytically equivalent to applying the conjugate gradient method to the normal equation $\mA^\top \mA \vs = \mA^\top \vb$ under exact arithmetic, but is more numerically efficient. On the other hand, when $\mA$ is symmetric, we use the conjugate residual method~\cite{stiefel1955relaxationsmethoden,saad2003iterative}, which is designed to minimize the norm of the residual vector $\vr_k = \vb - \mA \vs_k$ over the Krylov subspace $\vspan\{\vb,\mA\vb,\dots,\mA^{k-1}\vb\}$. For completeness, the full algorithm is shown in Subroutine~\ref{alg:CGLS}. 
We also remark that in the conjugate residual method, the relation $\vq_k = \mA \vp_k$ holds, but the matrix-vector product is not computed explicitly. Instead, in Line~\ref{line:update_of_q} we compute $\vq_{k+1}$ directly from $\vv_{k+1}$ and $\vq_k$, without an additional matrix-vector product (noting that $\vv_{k+1} = \mA \vr_{k+1}$, $\vq_k = \mA \vp_k$, and $\vp_{k+1} = \vr_{k+1}+\beta_k \vp_k$). Thus, we observe that $\mathsf{LinearSolver}$ requires at most two matrix-vector products per iteration when $\mA$ is non-symmetric, and only one when $\mA$ is symmetric.

\myalert{The $\mathsf{ExtEvec}$ oracle}.
Next, we discuss implementing the $\mathsf{ExtEvec}$ oracle in Definitions~\ref{def:extevec}. As we mentioned in Section~\ref{subsec:nonlinear_eqs}, it is closely related to computing inexactly the extreme eigenvectors and the extreme eigenvalues of a given matrix. Thus, we build this oracle based on the classical Lanczos method with a random start, where the initial vector is chosen randomly and uniformly from the unit sphere (see, e.g., \cite{saad2011numerical,yurtsever2021scalable}.) 
To begin with, we recall a classical result in \cite{kuczynski1992estimating} on the convergence behavior of the Lanczos method. 
\begin{proposition}[{\cite[Theorem~4.2]{kuczynski1992estimating}}]\label{prop:lanczos}
  Consider a symmetric matrix $\mW$ and let $\lambda_1(\mW)$ and $\lambda_d(\mW)$ denote its largest and smallest eigenvalues, respectively. Then after $k$ iterations of the Lanczos method with a random start, we find unit vectors $\vu^{(1)}$ and $\vu^{(d)}$  such that 
  \begin{align*}
    \mathbb{P}(\langle\mW \vu^{(1)}, \vu^{(1)}\rangle\leq \lambda_1(\mW)-\epsilon(\lambda_1(\mW)-\lambda_d(\mW))) \leq 1.648\sqrt{d}e^{-\sqrt{\epsilon}(2k-1)},\\
    \mathbb{P}(\langle\mW \vu^{(d)}, \vu^{(d)}\rangle\geq \lambda_d(\mW)+\epsilon(\lambda_1(\mW)-\lambda_d(\mW))) \leq 1.648\sqrt{d}e^{-\sqrt{\epsilon}(2k-1)},
  \end{align*}
  As a corollary, to ensure that, with probability at least $1-q$, $ \langle\mW \vu^{(1)}, \vu^{(1)}\rangle> \lambda_1(\mW)-\epsilon(\lambda_1(\mW)-\lambda_d(\mW))$ and $\langle\mW \vu^{(d)}, \vu^{(d)}\rangle< \lambda_n(\mW)+\epsilon(\lambda_1(\mW)-\lambda_d(\mW))$,
  the number of iterations can be bounded by  $\lceil\frac{1}{4}\epsilon^{-1/2}\log(\frac{11d}{q^2})+\frac{1}{2} \rceil$.
\end{proposition}
Now we are ready to describe our procedure detailed in Subroutine~\ref{alg:lanczos}. Specifically, given the input matrix $\mW \in \reals^{d\times d}$, we compute the symmetrized matrix $\bar{\mW} = \frac{1}{2}(\mW+\mW^\top)$ and run the Lanczos method for $N = \lceil\frac{1}{4}\sqrt{2(1+\frac{1}{\delta})}\log(\frac{11d}{q^2})+\frac{1}{2} \rceil$ iterations. This yields approximate unit eigenvectors $\vu^{(1)}$ and $\vu^{(d)}$, which correspond to the largest and the smallest eigenvalues of $\bar{\mW}$, respectively. Moreover, we define $\hat{\lambda}_1 = \langle \bar{\mW} \vu^{(1)}, \vu^{(1)} \rangle$ and $\hat{\lambda}_d = \langle \bar{\mW} \vu^{(d)}, \vu^{(d)} \rangle$ as the approximate largest and smallest eigenvalues of $\bar{\mW}$, respectively. Then we set $\gamma = \max\{\hat{\lambda}_1, \hat{\lambda}_d\}$. To construct the output pair $(\gamma, \mS)$ satisfying the conditions in Definition~\ref{def:extevec}, we distinguish two cases depending on the value of $\gamma$. In Case I where $\gamma \leq 1$, we return $\gamma$ and set $\mS = 0$. Otherwise, in Case II where $\gamma >1$, we return $\gamma$ along with the rank-one matrix $\mS$ given by $$\mS = \begin{cases}
  \vu^{(1)}(\vu^{(1)})^\top, & \text{if }\hat{\lambda}_1 \geq -\hat{\lambda}_d; \\
  -\vu^{(d)}(\vu^{(d)})^\top, & \text{otherwise.}
\end{cases}$$ 
In the following lemma, we prove the correctness of Subroutine~\ref{alg:lanczos}. 

\begin{subroutine}[!t]\small
  \caption{$\mathsf{ExtEvec}(\mW;\delta,q)$}\label{alg:lanczos}
  \begin{algorithmic}[1]
      \STATE \textbf{Input:} $\mW \in \reals^{d\times d}$, $\delta>0$, $q\in (0,1)$
      \STATE \textbf{Initialize:} sample $\vv_1\in \reals^d$ uniformly from the unit sphere, $\beta_1 \leftarrow 0$, $\vv_0\leftarrow 0$
      \STATE Set $\bar{\mW} \leftarrow \frac{1}{2}(\mW + \mW^\top)$ and set the number of iterations 
      $N \leftarrow \Bigl\lceil\frac{1}{4}\sqrt{2(1+\frac{1}{\delta})}\log(\frac{11d}{q^2})+\frac{1}{2} \rceil$
      \FOR{$k=1,\dots,N$ \tikzmark{top}}
      \STATE Set $\vw_k \leftarrow \bar{\mW} \vv_k-\beta_k \vv_{k-1}$ 
      \STATE Set $\alpha_k \leftarrow \langle \vw_k,\vv_k \rangle $ and $\vw_k \leftarrow \vw_k-\alpha_k\vv_k$
      \STATE Set $\beta_{k+1} \leftarrow \|\vw_k\|$ and $\vv_{k+1}\leftarrow \vw_k/\beta_{k+1}$
      \ENDFOR
    \STATE Form a tridiagonal matrix $\mT \leftarrow \mathsf{tridiag}(\beta_{2:N},\alpha_{1:N},\beta_{2:N})$
    \STATE \hspace{-1em}\COMMENT{Use the tridiagonal structure to compute eigenvectors of $\mT$}
    \STATE Compute $(\hat{\lambda}_1,\vz^{(1)}) \leftarrow \mathsf{MaxEvec}(\mT)$ and $(\hat{\lambda}_d, \vz^{(d)}) \leftarrow \mathsf{MinEvec}(\mT)$ \tikzmark{right}
    \STATE Set $\vu^{(1)} \leftarrow \sum_{k=1}^N z^{(1)}_k\vv_k$ and $\vu^{(d)} \leftarrow \sum_{k=1}^N z^{(d)}_k\vv_k$ \tikzmark{bottom}
    \STATE Set $\gamma \leftarrow \max\{\hat{\lambda}_1,-\hat{\lambda}_d\}$ 
    \IF{$\gamma \leq 1$}
    \STATE Return $\gamma$ and $\mS = 0$ \COMMENT{Case I: $\gamma\leq 1$, which implies $\|\mW\|_{\op} \leq 1+\delta$}
    \ELSIF{$\hat{\lambda}_1 \geq -\hat{\lambda}_d$}
      \STATE Return $\gamma$ and $\mS = \vu^{(1)}(\vu^{(1)})^\top$ \hspace{.75em}\COMMENT{Case II: $\gamma> 1$ and $\mS$ defines a separating hyperplane}
    \ELSE
      \STATE Return $\gamma$ and $\mS = -\vu^{(d)}(\vu^{(d)})^\top$ \COMMENT{Case II: $\gamma> 1$ and $\mS$ defines a separating hyperplane}
    \ENDIF
  \end{algorithmic}
  \AddNote{top}{bottom}{right}{\color{comment}\textit{\quad Lanczos method}}
\end{subroutine}
\begin{lemma}\label{lem:extevec_bound}
  Let $\gamma$ and $\mS$ be the output of $\mathsf{ExtEvec}(\mW;\delta,q)$ in Subroutine~\ref{alg:lanczos}. Then with probability at least $1-q$, they satisfy one of the properties given in Definition~\ref{def:extevec}.  
\end{lemma}
\begin{proof}
  Note that in Subroutine~\ref{alg:lanczos}, we run the Lanczos method for $\Bigl\lceil \frac{1}{4}\epsilon^{-1/2}\log\frac{11d}{q^2}+\frac{1}{2}\Bigr\rceil$ iterations, where $\epsilon = \frac{\delta}{2(1+\delta)}$. 
  Thus, by Proposition~\ref{prop:lanczos}, with probability at least $1-q$ we have 
  \begin{align}
    \hat{\lambda}_1 \triangleq \langle\bar{\mW} \vu^{(1)}, \vu^{(1)}\rangle \geq  \lambda_1(\bar{\mW})-\frac{\delta}{2(1+\delta)}(\lambda_1(\bar{\mW})-\lambda_d(\bar{\mW})), \label{eq:lambda_1}\\ 
    \hat{\lambda}_d \triangleq \langle\bar{\mW} \vu^{(d)}, \vu^{(d)}\rangle \leq  \lambda_d(\bar{\mW})+\frac{\delta}{2(1+\delta)}(\lambda_1(\bar{\mW})-\lambda_d(\bar{\mW})). \label{eq:lambda_d}
  \end{align}
  Combining \eqref{eq:lambda_1} and \eqref{eq:lambda_d}, we obtain $\frac{\lambda_1(\bar{\mW})-\lambda_d(\bar{\mW})}{1+\delta}\leq \hat{\lambda}_1-\hat{\lambda}_d$, which further implies that $\lambda_1(\bar{\mW})-\lambda_d(\bar{\mW})\leq (1+\delta)(\hat{\lambda}_1-\hat{\lambda}_d)$. 
  By plugging the above inequality back into \eqref{eq:lambda_1} and \eqref{eq:lambda_d}, we further have 
  \begin{align}
    \lambda_1(\bar{\mW}) \leq \hat{\lambda}_1+\frac{\delta}{2(1+\delta)}(\lambda_1(\bar{\mW})-\lambda_d(\bar{\mW})) \leq \hat{\lambda}_1+\frac{\delta}{2}(\hat{\lambda}_1-\hat{\lambda}_d), \label{eq:upper_bound_lambda_1} \\
    \lambda_d(\bar{\mW}) \geq \hat{\lambda}_d-\frac{\delta}{2(1+\delta)}(\lambda_1(\bar{\mW})-\lambda_d(\bar{\mW})) \geq \hat{\lambda}_d-\frac{\delta}{2}(\hat{\lambda}_1-\hat{\lambda}_d). \label{eq:lower_bound_lambda_d}
  \end{align}
  Recall that $\gamma = \max\{\hat{\lambda}_1,-\hat{\lambda}_d\}$. By \eqref{eq:upper_bound_lambda_1} and \eqref{eq:lower_bound_lambda_d}, 
  we can further bound the eigenvalues of $\bar{\mW}$ by  
  \begin{equation}
    \lambda_1(\bar{\mW}) \leq \gamma +\frac{\delta}{2}\cdot 2\gamma ={(1+\delta)\gamma} \quad \text{and} \quad \lambda_d(\bar{\mW}) \geq -\gamma -\frac{\delta}{2}\cdot 2\gamma = -(1+\delta){\gamma}.
  \end{equation}
  Hence, we can see that $-(1+\delta)\gamma \mI\preceq \bar{\mW} = \frac{1}{2}(\mW + \mW^\top) \preceq (1+\delta)\gamma \mI$. 
  Now we distinguish three cases. 
  \begin{enumerate}[(a)]
    \item If $\gamma \leq 1$, then we are in \textbf{Case I} and the $\mathsf{ExtEvec}$ oracle outputs $\gamma$ and $\mS = 0$. In this case, we indeed have $-(1+\delta) \mI\preceq \frac{1}{2}(\mW + \mW^\top) \preceq (1+\delta) \mI$. 
    \item If $\gamma >1$ and $\hat{\lambda}_1 \geq -\hat{\lambda}_d$, then we are in \textbf{Case II} and the $\mathsf{ExtEvec}$ oracle returns $\gamma $ and $\mS = \vu^{(1)}(\vu^{(1)})^\top$. In this case, since $\max\{\lambda_1(\bar{\mW}),-\lambda_d(\bar{\mW})\} \leq \gamma(1+\delta)$, we have $-(1+\delta) \mI\preceq \frac{1}{2\gamma}(\mW + \mW^\top) \preceq (1+\delta) \mI$. Also, since $\vu^{(1)}$ is a unit vector, we have $\|\mS\|_F = \|\vu^{(1)}\|^2 = 1$. Finally, for any $\hat{\mB}$ such that $-\mI \preceq \frac{1}{2}(\hat{\mB}+\hat{\mB}^\top) \preceq \mI$, 
    we have 
    $\langle \mS,\mW-\hat{\mB}\rangle = \vu_1^\top{\mW} \vu_1-\vu_1^\top \hat{\mB} \vu_1 = \vu_1^\top{\bar{\mW}} \vu_1 -\frac{1}{2}\vu_1^\top (\hat{\mB}+\hat{\mB}^\top) \vu_1 \geq \hat{\lambda}_1-1 = \gamma-1$. Thus, $\gamma$ and $\mS$ satisfy all the properties in \textbf{Case II}. 
    \item If $\gamma >1$ and $-\hat{\lambda}_d \geq \hat{\lambda}_1$, then we are also in \textbf{Case II} and the $\mathsf{ExtEvec}$ oracle returns $\gamma$ and $\mS = -\vu^{(d)}(\vu^{(d)})^\top$. The rest follows similarly to the case above.
  \end{enumerate}
  This completes the proof.
\end{proof}
Finally, note that the Lanczos method computes one matrix-vector product in each iteration. Thus, we conclude that $\mathsf{ExtEvec}(\mW;\delta,q)$ requires $\bigO\left(\sqrt{1+\frac{1}{\delta}}\log(\frac{d}{q^2})\right)$ matrix-vector products in total. 

\myalert{The $\mathsf{MaxSvec}$ oracle}. Finally, we discuss the implementation of $\mathsf{MaxSvec}$ in Definition~\ref{def:MaxSvec}. As we noted in Section~\ref{sec:jacobian_online_learning}, this oracle has a close connection to the problem of computing the maximal singular value and its associated left and right singular vectors of a given matrix. Our key observation is the following fact, which converts a singular value problem for a non-symmetric matrix into an eigenvalue problem for a symmetric matrix. 
\begin{lemma}[{\cite[Section 8.6.1]{Golub2013Matrix}}]\label{lem:singular_to_eigen}
  Let $\mW\in \reals^{d\times d}$ be an arbitrary matrix. Define the augmented matrix $\tilde{\mW}$ by 
  \begin{equation}\label{eq:augmented}
    \tilde{\mW} = \begin{bmatrix} 0 & \mW \\ \mW^\top & 0
    \end{bmatrix} \in \mathbb{S}^{2d}. 
  \end{equation}
  \begin{itemize}
    \item Let $\lambda_1(\tilde{\mW})$ and $\lambda_{2d}(\tilde{\mW})$ denote its largest and smallest eigenvalues, respectively. Then $\lambda_1(\tilde{\mW})$ is the maximal singular value of $\mW$ and $\lambda_{2d}(\tilde{\mW}) = -\lambda_1(\tilde{\mW})$. 
    \item Let $\tilde{\vv} \in \reals^{2d}$ denote the eigenvector associated with $\lambda_1(\tilde{\mW})$ and partition it as $\tilde{\vv} = (\vv,\vv')$, where $\vv,\vv' \in \reals^d$. Then ${\vv}$ and ${\vv'}$ are the left and right singular vectors associated with the maximal singular value of $\mW$, respectively.
  \end{itemize}

\end{lemma}

In light of Lemma~\ref{lem:singular_to_eigen}, to construct the $\mathsf{MaxSvec}$ oracle, we can similarly apply the Lanczos method with a random start to compute inexactly the extreme eigenvectors and the extreme eigenvalues of the augmented matrix in \eqref{eq:augmented}. The procedure is detailed in Subroutine~\ref{alg:lanczos_singular}, and in the following lemma we prove its correctness.  

\begin{subroutine}[!t]\small
  \caption{$\mathsf{MaxSvec}(\mW;\delta,q)$}\label{alg:lanczos_singular}
  \begin{algorithmic}[1]
    \STATE \textbf{Input:} $\mW \in \mathbb{R}^{d\times d}$, $\delta>0$, $q\in (0,1)$
    \STATE Set $\tilde{\mW} \leftarrow \begin{bmatrix} 0 & \mW \\ \mW^\top & 0
    \end{bmatrix} \in \mathbb{S}^{2d}$ and the number of iterations $N  \leftarrow 
    \Bigl\lceil \frac{1}{4}\sqrt{2(1+\frac{1}{\delta})}\log\frac{22d}{q^2}+\frac{1}{2}\Bigr
    \rceil
    $
    
    \STATE Compute approximate largest eigenvalue $\tilde{\lambda}_1$ and its eigenvector $\tilde{\vv} \in \reals^{2d}$, by running the Lanczos method on $\tilde{\mW}$ for $N$ iterations \COMMENT{See Subroutine~\ref{alg:lanczos}}
    \STATE Partition $\tilde{\vv}$ as $\tilde{\vv} = [\vv,\vv']$ where $\vv, \vv' \in \reals^d$ 
    \STATE Set $\gamma \leftarrow \tilde{\lambda}_1/3$ 
    \IF{$\gamma \leq 1$}
    \STATE Return $\gamma$ and $\mS = 0$ \COMMENT{Case I: $\gamma\leq 1$, which implies $\|\mW\|_{\op} \leq 3(1+\delta)$}
    \ELSE
      \STATE Return $\gamma$ and $\mS = \frac{2}{3}\vv'\vv^\top \in \reals^{d\times d}$ \hspace{.75em}\COMMENT{Case II: $\gamma> 1$ and $\mS$ defines a separating hyperplane}
    \ENDIF
  \end{algorithmic}
\end{subroutine}

\begin{lemma}\label{lem:maxsvec_bound}
  Let $\gamma$ and $\mS$ be the output of $\mathsf{MaxSvec}(\mW;\delta,q)$ in Subroutine~\ref{alg:lanczos_singular}.  %
  Then with probability at least $1-q$, they satisfy the properties in Definition~\ref{def:MaxSvec}.  
\end{lemma}
\begin{proof}
  Note that in Subroutine~\ref{alg:lanczos_singular}, we run the Lanczos method for $\Bigl\lceil \frac{1}{4}\epsilon^{-1/2}\log\frac{22d}{q^2}+\frac{1}{2}\Bigr\rceil$ iterations, where $\epsilon = \frac{\delta}{2(1+\delta)}$. 
  Thus, by Proposition~\ref{prop:lanczos}, with probability at least $1-q$ we can find a unit vector $\tilde{\vv} \in \reals^{2d}$ such that 
  $\tilde{\lambda}_1 \triangleq \langle \tilde{\mW} \tilde{\vv}, \tilde{\vv}\rangle \geq  \lambda_1(\tilde{\mW})-\frac{\delta}{2(1+\delta)}(\lambda_1(\tilde{\mW})-\lambda_{2d}(\tilde{\mW}))$, %
  Moreover, notice that $\lambda_{2d}(\tilde{\mW}) = -\lambda_1(\tilde{\mW})$ by Lemma~\ref{lem:singular_to_eigen}. Hence, the above further implies that $\tilde{\lambda}_1 \geq \frac{1}{1+\delta} \lambda_1(\tilde{\mW})$, which is equivalent to $\lambda_1(\tilde{\mW}) \leq (1+\delta)\tilde{\lambda}$. 
  Moreover, note that $\lambda_1(\tilde{\mW}) = \|\mW\|_{\op}$ and $\gamma = \tilde{\lambda}_1/3$. Hence, the above is further equivalent to $\|\mW\|_{\op} \leq 3(1+\delta)\gamma$. 
  Now we distinguish two cases. 
  \begin{enumerate}[(a)]
    \item If $\gamma \leq 1$, then we are in \textbf{Case I} and the $\mathsf{MaxSvec}$ oracle outputs $\gamma$ and $\mS = 0$. In this case, we indeed have $\|\mW\|_\op \leq  3(1+\delta)\gamma \leq 3(1+\delta)$. 
    \item If $\gamma >1$, then we are in \textbf{Case II} and the $\mathsf{MaxSvec}$ oracle returns $\gamma $ and $\mS = \frac{2}{3}\vv'\vv^\top$. 
    In this case, since $\|\mW\|_\op \leq 3(1+\delta)\gamma$, we have $\|\mW/\gamma\|_{\op} \leq 3(1+\delta)$. Also, since $\tilde{\vv}$ is a unit vector, we have $\|\mS\|_F = \frac{2}{3} \|\vv\|\|\vv'\| \leq \frac{1}{3}(\|\vv\|^2 + \|\vv'\|^2) = \frac{1}{3}\|\tilde{\vv}\|^2 =\frac{1}{3}$. Finally, it holds that $\langle \mS, \mW \rangle = \frac{2}{3} \vv^\top \mW \vv' = \frac{1}{3}\tilde{\vv}^\top \tilde{\mW} \tilde{\vv} = \frac{\tilde{\lambda}_1}{3} = \gamma$. Moreover, for any $\hat{\mB}$ such that $\|\hat{\mB}\|_\op \leq 3$, we have
    $\langle \mS,\hat{\mB}\rangle =  \frac{2}{3}\vv^\top \hat{\mB} \vv' \leq \frac{2}{3}\|\vv\| \|\mB\|_{\op} \|\vv'\| \leq 2\|\vv'\| \|\vv\| \leq \|\vv\|^2 + \|\vv'\|^2 = 1$. 
    Hence, we obtain $\langle \mS,\mW-\hat{\mB}\rangle \geq \gamma -1$. This shows that $\gamma$ and $\mS$ satisfy all the properties in \textbf{Case II}. 
  \end{enumerate}
  This completes the proof.
\end{proof}

Finally, similar to the $\mathsf{ExtEvec}$ oracle, 
we note that $\mathsf{MaxSvec}(\mW;\delta,q)$ requires $\bigO\left(\sqrt{1+\frac{1}{\delta}}\log(\frac{d}{q^2})\right)$ matrix-vector products in total.

\section{Convergence rates and complexity analysis of QNPE}
\label{sec:complexity}

So far, we described our proposed QNPE method in Algorithm~\ref{alg:Full_Equasi-Newton} and discussed how the step size $\eta_k$ is selected by Subroutine~\ref{alg:ls} and the Jacobian approximation~$\mB_k$ is updated according to Subroutine~\ref{alg:jacobian_approx}. Moreover, we discussed the required $\mathsf{LinearSolver}$, $\mathsf{ExtEvec}$ and $\mathsf{MaxSvec}$ oracles  in Subroutines~\ref{alg:CGLS}, \ref{alg:lanczos}, and \ref{alg:lanczos_singular}, respectively.  
Next, we will establish the convergence rate and characterize the computational cost of QNPE for the strongly monotone setting (Section~\ref{subsec:strongly_monotone}) and the monotone setting (Section~\ref{subsec:monotone}). 

\subsection{Strongly monotone setting}
\label{subsec:strongly_monotone}

\subsubsection{Convergence rate analysis}
In this section, we present our main convergence result when $\vF$ is strongly monotone. 
\begin{theorem}\label{thm:main}
Suppose Assumptions~\ref{assum:strong_monotone},~\ref{assum:operator_lips} and~\ref{assum:jacobian_lips} hold. Let $\{\vz_k\}$ and $\{\hat{\vz}_k\}$ be the iterates generated by Algorithm~\ref{alg:Full_Equasi-Newton} using the line search scheme in Subroutine~\ref{alg:ls}, where \blue{$\alpha_1, \alpha_2 \in (0,\frac{1}{2})$, $\beta\in(0,1)$,}
and $\sigma_0 \geq \frac{\alpha_2\beta}{7.5L_1}$. In addition, the Jacobian approximation matrices are updated in Subroutine~\ref{alg:jacobian_approx} with \textbf{Option I}, where {$\rho = \frac{1}{121}$}, $\delta_t = {\frac{\mu}{2L_1}}$, and $q_t = \frac{p}{2.5(t+1)\log^2(t+1)}$ for $t \geq 1$. With probability at least $1-p$, the following holds:  
  \begin{enumerate}[(a)]
    \item (Linear convergence) For any $k\geq 0$, we have
        $ \frac{\|\vz_{k+1}-\vz^*\|^2}{\|\vz_k-\vz^*\|^2 }\leq  \left(1+\frac{4\alpha_2\beta\mu}{15L_1}\right)^{-1}$.
    \item (Superlinear convergence) We have $\lim_{k\rightarrow \infty}\frac{\|\vz_{k+1}-\vz^*\|^2}{\|\vz_k-\vz^*\|^2 } = 0$. Furthermore, define the absolute constant $C$ and the quantity $M$ by  
    \begin{align}\label{eq:def_c1}
      C &= \frac{(1+\alpha_1)^2}{2(1-\alpha_1)^2\beta^2 (1-(\alpha_1+\alpha_2))}, \\
       M &= \frac{121\|\mB_0-\nabla^2f(\vz^*)\|_F^2}{L_1^2}+ \left(C+2+\frac{15L_1}{2\alpha_2 \beta \mu}\right)\frac{L_2^2\|\vz_0-\vz^*\|^2}{L_1^2}. \label{eq:def_M}
    \end{align}
    Then for any $k\geq 0$, we have 
    \begin{equation*}
      {%
      \frac{\|\vz_k-\vz^*\|^2}{\|\vz_0-\vz^*\|^2} \leq   \Biggl(1+ 
       \frac{(1-\beta)\alpha_2\beta\mu}{L_1} \min\left\{ \frac{2}{15}k, \sqrt{\frac{k}{M}}\right\}
      \Biggr)^{-k}%
      .}
    \end{equation*}
  \end{enumerate}
\end{theorem}
Theorem~\ref{thm:main} presents two global convergence rates for QNPE. Specifically, Part~(a) shows that it converges linearly and matches the rate of EG, leading to a complexity bound of $\bigO(\frac{L_1}{\mu} \log\frac{1}{\epsilon})$. \blue{According to the lower bound results \cite{Nemirovsky1983Problem,zhang2022lower}, this complexity bound is optimal up to constants in the regime where the number of iterations $k$ is $\mathcal{O}(d)$. }
Additionally, Part~(b) shows an explicit superlinear rate of $(1 + \frac{\mu}{L_1}\sqrt{\frac{k}{M}})^k$, where $ M = \bigO(\frac{1}{L_1^2}\|\mB_0-\nabla \vF(\vz^*)\|^2_F + \frac{L_2^2}{L_1\mu} \|\vz_0-\vz^*\|^2)$. This rate outperforms the linear rate in Part (a) when $k\geq M$. 
By Lemma~\ref{lem:jacobian}, it holds that $\|\mB_0 - \nabla \vF(\vz^*)\|^2 \leq L_1^2 d$, and thus in the worst case $M = \bigO(d)$. 
Therefore, after at most $\bigO(d)$ iterations, our convergence rate is provably faster than the linear rate by EG. To the best of our knowledge, our result is the first to show any acceleration of a quasi-Newton method compared with EG in the regime where $ k =\Omega(d)$.

Now we move to the proof of Theorem~\ref{thm:main}. First of all, by using a simple union bound, we can show that 
 the  $\mathsf{ExtEvec}$ and $\mathsf{MaxSvec}$ oracles used in Subroutine~\ref{alg:jacobian_approx} are successful in all rounds with probability at least $1-p$. 
 Specifically, 
since both  oracles have a failure probability of $ \frac{q_t}{2} = \frac{p}{5(t+1)\log^2(t+1)}$ in the $t$-th round, we can use the union bound to upper bound the total failure probability by 
\begin{equation*}
  \sum_{t=1}^{T-1} q_t %
   \leq  \frac{p}{2.5}\sum_{t=2}^{\infty} \frac{1}{t \log^2 t} \leq \frac{p}{2.5}\left(\frac{1}{2\log^2 2}+\int_{2}^{+\infty}\frac{1}{t\log^2t}\,dt\right) \leq p.
\end{equation*}
 Thus, throughout the proof, we assume every call of $\mathsf{ExtEvec}$ and $\mathsf{MaxSvec}$ is successful.

\begin{proof}[Proof of Theorem~\ref{thm:main}(a)]

We first prove the linear convergence rate in (a). To begin with, note that Subroutine~\ref{alg:jacobian_approx} with \textbf{Option I} guarantees that  $\hat{\mB}_k \in (1+\frac{\mu}{2L_1}) \mathcal{C}$ by Lemma~\ref{lem:regret_reduction}.  As a corollary of Lemma~\ref{lem:transform}, we have the following result showing the boundedness of the Jacobian approximation matrix $\mB_k$.

\begin{corollary}\label{lem:feasible_Bk}
  Let $\delta_t = \frac{\mu}{2L_1}$ for all $t\geq 0$. Then for any $k \geq 0$, we have $\frac{1}{2}(\mB_k + \mB_k^\top )\succeq \frac{\mu}{2} \mI$ and $\|\mB_k\|_{\op} \leq 6.5L_1$. 
\end{corollary}
Combining this with Lemma~\ref{lem:step size_lb}, we obtain the following universal lower bound on the step size $\eta_k$. 
\begin{lemma}\label{lem:stepsize_const_bound}
  In Subroutine~\ref{alg:ls}, choose $\sigma_0 \geq \frac{\alpha_2\beta}{7.5L_1}$. For any $k\geq 0$, we have $\eta_k \geq \frac{\alpha_2\beta}{7.5L_1}$. 
\end{lemma}
\begin{proof}
  Recall that $\calB = \{k: \eta_k < \sigma_k\}$.
  We first establish that $\eta_k \geq \frac{\alpha_2\beta}{7.5L_1}$ for $k\in \mathcal{B}$. To see this, suppose $k\in \mathcal{B}$ and recall from Lemma~\ref{lem:step size_lb} that
\begin{equation}\label{eq:step size_lb_appen}
  \eta_k > \frac{\alpha_2 \beta\|{\tilde{\vz}_k}-\vz_k\|}{\|\vF({\tilde{\vz}_k})-\vF(\vz_k)-\mB_k({\tilde{\vz}_k}-\vz_k)\|}.
\end{equation}
  Moreover, by using the triangle inequality, we have 
  $\|\vF({\tilde{\vz}_k})-\vF(\vz_k)-\mB_k({\tilde{\vz}_k}-\vz_k)\| \leq \|\vF({\tilde{\vz}_k})-\vF(\vz_k)\| + \|\mB_k({\tilde{\vz}_k}-\vz_k)\|$.
  By using Assumption~\ref{assum:operator_lips}, we have $\|\vF({\tilde{\vz}_k})-\vF(\vz_k)\| \leq L_1\|\tilde{\vz}_k-\vz_k\|$. Moreover, since $\|\mB_k\|_{\op} \leq 6.5L_1$ by Corollary~\ref{lem:feasible_Bk}, we also get $\|\mB_k({\tilde{\vz}_k}-\vz_k)\| \leq 6.5 \|{\tilde{\vz}_k}-\vz_k\|$. 
  Thus, this implies that $\|\vF({\tilde{\vz}_k})-\vF(\vz_k)-\mB_k({\tilde{\vz}_k}-\vz_k)\| \leq 7.5\|{\tilde{\vz}_k}-\vz_k\|$,  
  which leads to $\eta_k > \frac{\alpha_2\beta}{7.5L_1}$ from \eqref{eq:step size_lb_appen}. 

  Now we can prove that $\eta_k \geq \frac{\alpha_2\beta}{7.5L_1}$ for all $k\geq 0$ by induction. To show that this holds for $k=0$, we distinguish two cases. If $0\notin \mathcal{B}$, then we have $\eta_0 = \sigma_0 > \frac{\alpha_2\beta}{7.5L_1}$ by our choice of $\sigma_0$. Otherwise, if $0\in \mathcal{B}$, then it directly follows from our result in the previous paragraph. Moreover, assume that $\eta_{l-1} \geq \frac{\alpha_2\beta}{7.5L_1}$ where $l\geq 1$. Similarly, we again distinguish two cases: if $l\notin \mathcal{B}$, then we have $\eta_l = \sigma_l = \frac{\eta_{l-1}}{\beta} >\frac{\alpha_2}{7.5L_1}>\frac{\alpha_2\beta}{7.5L_1}$; otherwise, if $l\in \mathcal{B}$, it follows from the result above that $\eta_l \geq \frac{\alpha_2\beta}{7.5L_1}$. This completes the induction. 
\end{proof}
In light of Lemma~\ref{lem:stepsize_const_bound}, it follows from Proposition~\ref{prop:HPE}(a) that  $\|\vz_{k+1}-\vz^*\|^2 \leq \|\vz_k-\vz^*\|^2 (1+2\eta_k \mu)^{-1} \leq \|\vz_k-\vz^*\|^2 \left(1+\frac{4\alpha_2\beta\mu}{15L_1}\right)^{-1}$. 
This proves the linear convergence result in (a). 
\end{proof}

Next, we prove the superlinear convergence rate in (b). Our starting point is the observations in \eqref{eq:after_jensen} and \eqref{eq:goal}, which demonstrate that it is sufficient to prove an upper bound on the cumulative loss $\sum_{k=0}^{N-1} \ell_k(\mB_k)$. 
Moreover, since $\ell_k(\mB_k) = 0$ when $k \notin \calB$ by the definition in \eqref{eq:loss_of_jacobian}, it is equivalent to bounding the sum $\sum_{k \in \calB} \ell_k(\mB_k)$.  
\blue{Recall that we relabel the indices in $\mathcal{B}$ by $t=0,\dots,T-1$ with $T \leq N$. %
} 

\begin{proof}[Proof of Theorem~\ref{thm:main}(b)] We consider the following three steps.

\vspace{0em}\noindent\textbf{Step 1:} Using regret analysis, we bound the cumulative loss $\sum_{t=0}^{T-1}\ell_{t}(\mB_{t})$ incurred by our online learning algorithm in Subroutine~\ref{alg:jacobian_approx}. 
To begin with, we present the following lemma showing a self-bounding property of the loss function $\ell_t$. {It is similar to the standard inequality $\frac{1}{2L_1} \|\nabla g (\vx)\|^2 \leq g(\vx) - g^* $ for a $L_1$-smooth function $g$, where $g^*$ denotes the minimum of $g$.} This will be the key to proving a constant upper bound on the cumulative loss incurred by Subroutine~\ref{alg:jacobian_approx}.
\begin{lemma}\label{lem:loss}
  Recall the loss function $\ell_k$ defined in \eqref{eq:loss_of_jacobian}. For $k\in \mathcal{B}$, we have
  \begin{equation}\label{eq:gradient}
    \nabla \ell_k(\mB) = - \frac{2(\vu_k-\mB\vs_k)\vs_k^\mathsf{T}}{\|\vs_k\|^2}.
  \end{equation}
  Moreover, for any $\mB\in \reals^{d\times d}$, it holds that 
    \begin{equation}\label{eq:gradient_norm_bound}
        \|\nabla \ell_k(\mB)\|_F \leq \|\nabla \ell_k(\mB)\|_* = 2\sqrt{\ell_k(\mB)},
    \end{equation} 
    where $\|\cdot\|_F$ and $\|\cdot\|_{*}$ denote the Frobenius norm and the nuclear norm, respectively.  
\end{lemma}
\begin{proof}
  The expression in \eqref{eq:gradient} follows from the direct calculation. The first inequality in \eqref{eq:gradient_norm_bound} follows from the fact that $\|\mA\|_F \leq \|\mA\|_*$ for any matrix $\mA\in \reals^{d\times d}$. For the equality, note that 
  \begin{align*}
    \|\nabla \ell_k(\mB)\|_*   = \frac{2}{\|\vs_k\|^2}  \|\vu_k-\mB\vs_k\|\|\vs_k\|  = \frac{2\|\vu_k-\mB\vs_k\|}{\|\vs_k\|} = 2\sqrt{\ell_k(\mB)}, 
  \end{align*}
  where we used the fact that the rank-one matrix $\va\vb^\top$ has only one nonzero singular value $\|\va\|\|\vb\|$. 
\end{proof}

In particular, by exploiting the self-bounding property of the loss function $\ell_t$, we prove a ``small-loss bound'' \cite{srebro2010smoothness} in the following lemma, where the cumulative loss of the learner is bounded by that of a fixed action in the competitor set. 
  \begin{lemma}
  \label{lem:small_loss}
  In Subroutine~\ref{alg:jacobian_approx}, choose \textbf{Option I} with $\rho = \frac{1}{121}$ and $\delta = \frac{\mu}{2L_1}$. 
  For any ${\mH}\in \mathcal{{Z}}$, we have
    $\sum_{t=0}^{T-1} \ell_t({\mB}_t)  \leq 121\|\mB_0-{\mH}\|_F^2+ 2 \sum_{t=0}^{T-1} \ell_t({\mH})$.
\end{lemma}

\begin{proof}
Recall that $\calZ$ is contained in the linear subspace $\calL$ (see \eqref{eq:feasible_set}) and $\mP$ denotes the orthogonal projection matrix associated with $\calL$.    
By letting $\vx_t = \hat{\mB}_t$, $\vu = \hat{\mH} \triangleq \frac{1}{L_1}(\mH-(L_1+\mu)\mI)$, 
$\vg_t = \mG_t \triangleq \frac{1}{L_1}\mP \nabla \ell_t(\mB_t)$, $\tilde{\vg}_t = \tilde{\mG}_t$, $\vw_t = \mW_t$ in Lemma~\ref{lem:regret_reduction}, we obtain:  
\begin{enumerate}[(i)]
  \item $\hat{\mB}_t \in (1+\delta)\mathcal{C}$, which means $\|\hat{\mB}_t\|_{\op} \leq 3(1+\delta)\leq 4.5$ since $\delta = \frac{\mu}{2L_1}\leq \frac{1}{2}$.
  \item It holds that 
  \begin{align}
    \langle \mG_t, \hat{\mB}_t-\hat{\mH} \rangle &\leq \frac{1}{2\rho}\|\mW_t-\hat{\mH}\|_F^2-\frac{1}{2\rho}\|\mW_{t+1}-\hat{\mH}\|_F^2+\frac{\rho}{2}\|\tilde{\mG}_t\|_F^2, \label{eq:linearized_loss_matrix} \\%\\
    \|\tilde{\mG}_t\|_F &\leq \|\mG_t\|_F + |\langle \mG_t, \hat{\mB}_t\rangle|\|\mS_t\|_F. \label{eq:surrogate_gradient_bound}
  \end{align}
\end{enumerate}
First, note that $\|\mS_t\|_F \leq 1$ by Definition~\ref{def:extevec} and
$|\langle \mG_t, \hat{\mB}_t\rangle| \leq \|\mG_t\|_* \|\hat{\mB}_t\|_{\op} \leq 4.5 \|\mG_t\|_*$. %
Together with \eqref{eq:surrogate_gradient_bound}, we get 
\begin{equation}\label{eq:surrogate_bound_loss}
\|\tilde{\mG}_t\|_F \leq \|\mG_t\|_F+4.5\|\mG_t\|_* \leq 5.5\|\mG_t\|_* \leq \frac{5.5}{L_1}\|\nabla \ell_t(\mB_t)\|_* \leq \frac{11}{L_1}\sqrt{\ell_t({\mB}_t)},
\end{equation}
where we used $\mG_t = \frac{1}{L_1}\mP \nabla \ell_t(\mB_t)$ and Lemma~\ref{lem:loss} in the last inequality. 
Furthermore, since $\ell_t$ is convex and $\mB_t,\mH \in \calL$, we have 
\begin{equation}\label{eq:convexity_subspace}
    \ell_t({\mB}_t) - \ell_t(\mH) \!\leq\! \langle \nabla \ell_t(\mB_t), {\mB}_t-\mH \rangle \!=\! \langle \nabla \ell_t(\mB_t), \mP\left({\mB}_t-\mH\right) \rangle \!=\! L_1^2 \langle \mG_t, \hat{\mB}_t-\hat{\mH} \rangle,
\end{equation}
where we used $\mG_t = \frac{1}{L_1}\mP\nabla \ell_t(\mB_t)$, $\hat{\mB}_t \triangleq \frac{1}{L_1}(\mB_t-{(L_1+\mu)}\mI)$, and $\hat{\mH} \triangleq \frac{1}{L_1}(\mH-{(L_1+\mu)}\mI)$. 
Therefore, by \eqref{eq:linearized_loss_matrix} and \eqref{eq:surrogate_bound_loss} we get 
\begin{align*}
\ell_t({\mB}_t) - \ell_t(\mH) &\leq \frac{L_1^2}{2\rho}\|\mW_t-\hat{\mH}\|_F^2-\frac{L_1^2}{2\rho}\|\mW_{t+1}-\hat{\mH}\|_F^2+\frac{\rho L_1^2}{2} \|\tilde{\mG}_t\|_F^2\\
&\leq \frac{L_1^2}{2\rho}\|\mW_t-\hat{\mH}\|_F^2-\frac{L_1^2}{2\rho}\|\mW_{t+1}-\hat{\mH}\|_F^2+60.5\rho\ell_t({\mB}_t). %
\end{align*}
Since $\rho = \frac{1}{121}$, by rearranging and simplifying terms in the above inequality, we obtain
$\ell_t({\mB}_t) \leq 2\ell_t({\mH})+ {121 L_1^2}\|\mW_t-\hat{\mH}\|_F^2-{121 L_1^2}\|\mW_{t+1}-\hat{\mH}\|_F^2$.
By summing the above inequality from $t=0$ to $T-1$, we further have 
\begin{equation*}
  \sum_{t=0}^{T-1} \ell_t({\mB}_t)  \leq {121 L_1^2}\|\mW_0-\hat{\mH}\|_F^2+ 2 \sum_{t=0}^{T-1} \ell_t({\mH}) = 121\|\mB_0-{\mH}\|_F^2+ 2 \sum_{t=0}^{T-1} \ell_t({\mH}),
\end{equation*}
where the last equality is due to ${\mW}_0 \triangleq \frac{1}{L_1}(\mB_0-{(L_1+\mu)}\mI)$ and $\hat{\mH} \triangleq \frac{1}{L_1}(\mH-{(L_1+\mu)}\mI)$. This completes the proof. 
\end{proof}

Note that in Lemma~\ref{lem:small_loss}, we have the freedom to choose any competitor $\mH$ in the set $\mathcal{Z}$. To further obtain an explicit bound, a natural choice would be $\mH^* \triangleq \nabla \vF(\vz^*)$, which leads to our next step.

\vspace{0em}\noindent \textbf{Step 2:} We upper bound the cumulative loss $\sum_{t=0}^{T-1} \ell_t({\mH}^*)$. The proof relies crucially on Assumption~\ref{assum:jacobian_lips} as well as the linear convergence result we proved in (a).
To begin with, we prove the following useful lemma. 
\begin{lemma}\label{lem:sum_of_squares}
 If Assumption~\ref{assum:strong_monotone} holds, then $\sum_{k=0}^{N-1} \|\hat{\vz}_{k}\!-\!\vz_k\|^2 \leq \frac{1}{1-(\alpha_1+\alpha_2)}\|\vz_0\!-\!\vz^*\|^2$. 
\end{lemma}
\begin{proof}
  Recall the inequality~\eqref{eq:one_step_strongly} in the proof of Proposition~\ref{prop:HPE}. Since $\frac{1}{2}\|\vz_{k+1}-\vz^*\|^2 \!\leq \frac{1+2\eta_k\mu}{2}\|\vz_{k+1}-\vz^*\|^2$, we can further derive that 
  \begin{equation}\label{eq:one_step_strongly_coro}
    \frac{1-\alpha_1-\alpha_2}{2}(\|\hat{\vz}_k - \vz_k\|^2 + \|\hat{\vz}_k-\vz_{k+1}\|^2) \leq \frac{1}{2}\|\vz_k- \vz^*\|^2 - \frac{1}{2}\|\vz_{k+1} - \vz^*\|^2.
\end{equation} 
Since $\alpha_1+\alpha_2 < 1$, we further have $\frac{1-\alpha_1-\alpha_2}{2} \|\hat{\vz}_k - \vz_k\|^2 \leq \frac{1-\alpha_1-\alpha_2}{2}(\|\hat{\vz}_k - \vz_k\|^2 + \|\hat{\vz}_k-\vz_{k+1}\|^2) \leq \frac{1}{2}\|\vz_k- \vz^*\|^2 - \frac{1}{2}\|\vz_{k+1} - \vz^*\|^2$. By summing the above inequality from $k=0$ to $k=N-1$, we obtain that 
\begin{equation*}
  \frac{1-\alpha_1-\alpha_2}{2} \sum_{k=0}^{N-1} \|\hat{\vz}_k - \vz_k\|^2 \leq \frac{1}{2}\|\vz_0-\vz^*\|^2 - \frac{1}{2}\|\vz_N-\vz^*\|^2 \leq \frac{1}{2}\|\vz_0-\vz^*\|^2. 
\end{equation*}
Thus, we arrive at Lemma~\ref{lem:sum_of_squares} by dividing both sides by $\frac{1-\alpha_1-\alpha_2}{2}$. 
\end{proof}

With the help of Lemma~\ref{lem:sum_of_squares}, we are ready to upper bound $\sum_{t=0}^{T-1} \ell_t({\mH}^*)$. 
  \begin{lemma}\label{lem:comparator}
  Recall the constant $C$ defined in \eqref{eq:def_c1}. 
  Then 
    $\sum_{t=0}^{T-1} \ell_t(\mH^*) \leq \left(C+2+\frac{15L_1}{2\alpha_2 \beta \mu}\right)L_2^2\|\vz_0-\vz^*\|^2$.
  \end{lemma}
  \begin{proof}
    Recall the definition of $\ell_t$ in \eqref{eq:loss_of_jacobian}. 
    By the fundamental theorem of calculus, we can write $\vu_t = \vF({\tilde{\vz}_t})-\vF(\vz_t) = \bar{\mH}_t (\tilde{\vz}_t-\vz_t)$, where $\bar{\mH}_t = \int_{0}^1 \nabla \vF(\vz_t+ \lambda \vs_t) \,d\lambda$. Moreover, using the triangle inequality, we have 
\begin{align*}
  \|\bar{\mH}_t-\mH^*\|_{\op} &\leq \int_{0}^1 \|\nabla \vF(\vz_t+ \lambda \vs_t)-\nabla \vF(\vz^*)\|_{\op}\, d\lambda \\
  &\leq L_2\int_{0}^1 \|\vz_t-\lambda \vs_t + \vz^*\|\, d\lambda \\
  &\leq L_2\int_{0}^1 (\|\vz_t-\vz^*\| + \lambda \|\vs_t\|)\, d\lambda = L_2\Bigl(\|\vz_t-\vz^*\|+\frac{1}{2}\|\vs_t\|\Bigr),
\end{align*}
where we used Assumption~\ref{assum:jacobian_lips} in the second inequality. 
Therefore, we have $\|\vu_t-\mH^*\vs_t\| = \|(\bar{\mH}_t-\mH^*)\vs_t\| \leq \|\bar{\mH}_t-\mH^*\|_{\op}\|\vs_t\| \leq L_2\|\vs_t\|\Bigl(\|\vz_t-\vz^*\|+\frac{1}{2}\|\vs_t\|\Bigr)$. This further implies that 
\begin{align}
  \sum_{t=0}^{T-1} \ell_t(\mH^*) = \sum_{t=0}^{T-1} \frac{\|\vu_t-\mH^*\vs_t\|^2}{\|\vs_t\|^2} &\leq {L_2^2}\sum_{t=0}^{T-1}\Bigl(\|\vz_t-\vz^*\|+\frac{1}{2}\|\vs_t\|\Bigr)^2 \nonumber\\ 
  &\leq \frac{L_2^2}{2} \sum_{t=0}^{T-1} \|\vs_t\|^2 + 2 L_2^2\sum_{t=0}^{T-1} \|{\vz}_t-\vz^*\|^2. \label{eq:comparator_1}
\end{align}
\blue{Note that by our notation, we have $\sum_{t=0}^{T-1} \|\vs_t\|^2 = \sum_{k \in \calB} \|\vs_k\|^2 = \sum_{k \in \calB} \|\tilde{\vz}_k - \vz_k\|^2 $ and $\sum_{t=0}^{T-1} \|{\vz}_t-\vz^*\|^2 = \sum_{k\in \calB} \|{\vz}_k-\vz^*\|^2$.}
To bound the first sum $\sum_{k \in \calB} \|\tilde{\vz}_k - \vz_k\|^2$, we use the last result in Lemma~\ref{lem:step size_lb} and the inequality in Lemma~\ref{lem:sum_of_squares} to get 
\begin{equation}\label{eq:comparator_2}
  \begin{aligned}
    \sum_{k\in \calB} \|\tilde{\vz}_k-\vz_k\|^2  \leq \frac{(1+\alpha_1)^2}{\beta^2 (1-\alpha_1)^2} \sum_{k\in \calB} \|\hat{\vz}_k-\vz_k\|^2 
    &\leq \frac{(1+\alpha_1)^2}{\beta^2 (1-\alpha_1)^2} \sum_{k=0}^{N-1} \|\hat{\vz}_k-\vz_k\|^2 \\
      &\leq \frac{(1+\alpha_1)^2\|\vz_0-\vz^*\|^2}{(1-\alpha_1)^2\beta^2 (1-(\alpha_1+\alpha_2))} = 2C \|\vz_0-\vz^*\|^2. 
  \end{aligned}
\end{equation}
To bound the second sum $\sum_{k\in \calB} \|{\vz}_k-\vz^*\|^2$, we use the linear convergence result in Part (a) of Theorem~\ref{thm:main}:
\begin{equation}\label{eq:comparator_3}
 \begin{aligned}
  \sum_{k\in \calB} \|{\vz}_k-\vz^*\|^2 \leq \sum_{k=0}^{N-1} \|\vz_k-\vz^*\|^2 &\leq \|\vz_0-\vz^*\|^2\sum_{k=0}^{N-1} \left(1+\frac{4\alpha_2\beta\mu}{15L_1}\right)^{-k}\\ &\leq  \|\vz_0-\vz^*\|^2\left(1+\frac{15L_1}{4\alpha_2 \beta \mu}\right).
 \end{aligned}
\end{equation}
Lemma~\ref{lem:comparator} follows immediately from \eqref{eq:comparator_1}, \eqref{eq:comparator_2}, and \eqref{eq:comparator_3}. 
  \end{proof}

\noindent\textbf{Step 3}: Combining Lemmas~\ref{lem:comparator} and~\ref{lem:small_loss},  we obtain a constant upper bound on the cumulative loss as
$\sum_{t=0}^{T-1} \ell_t({\mB}_t) \leq 121\|\mB_0-{\mH^*}\|_F^2+ \left(C+2+\frac{15L_1}{2\alpha_2 \beta \mu}\right)L_2^2\|\vz_0-\vz^*\|^2 = L_1^2 M$, where we used the definition of $M$ in \eqref{eq:def_M}. 
Together with the fact that $\sigma_0 \geq \frac{\alpha_2\beta}{7.5L_1}$, it follows from  Lemma~\ref{lem:stepsize_bnd} that
\begin{equation*}
  \sum_{i=0}^{k-1}\frac{1}{\eta_i} \leq \frac{7.5L_1}{(1-\beta)\alpha_2\beta}+ \frac{L_1\sqrt{kM}}{(1-\beta)\alpha_2\beta} \leq 
  \frac{2L_1 \max\{7.5, \sqrt{kM}\}}{(1-\beta)\alpha_2\beta}. 
\end{equation*}
This leads to the superlinear convergence result in (b) by Proposition~\ref{prop:HPE} and the observation in \eqref{eq:after_jensen}.
\end{proof}

\subsubsection{Characterizing the computational cost}

In the following theorem, we characterize the number of operator evaluations required by our QNPE method.

\begin{theorem}[Operator evaluaiton complexity]\label{thm:line_search}
  After $N$ iterations, QNPE requires at most $3N+\log_{1/\beta}(\frac{7.5\sigma_0 L_1}{\alpha_2})$ total operator evaluations. 
\end{theorem}
\begin{proof}
  Let $l_k$ denote the number of line search steps in iteration $k$. %
  We first note that $\eta_k = \sigma_k \beta^{l_k-1}$ by our line search subroutine, which implies $l_k = \log_{1/\beta}(\sigma_k/\eta_k)+1$. Thus, the total number of line search steps after $N$ iterations can be bounded by 
    \begin{align}
      \sum_{k=0}^{N-1} l_k = \sum_{k=0}^{N-1} \left(\log_{1/\beta}\frac{\sigma_k}{\eta_k}+1\right) &= N+ \log_{1/\beta}\frac{\sigma_0}{\eta_0}+\sum_{k=1}^{N-1} \log_{1/\beta}\frac{\sigma_k}{\eta_k} \nonumber\\
      &= N+ \log_{1/\beta}\frac{\sigma_0}{\eta_0} + \sum_{k=1}^{N-1} \log_{1/\beta} \frac{\eta_{k-1}}{\beta \eta_k} \label{eq:using_sigma}\\
      &= 2N-1+ \log_{1/\beta}\frac{\sigma_0}{\eta_0} + \sum_{k=1}^{N-1} \log_{1/\beta} \frac{\eta_{k-1}}{\eta_k}  \nonumber\\
      & = 2N-1+ \log_{1/\beta}\frac{\sigma_0}{\eta_{N-1}}, \label{eq:bound_on_calls}
    \end{align} 
    where we used the fact that $\sigma_k = \eta_{k-1}/\beta$ for $k\geq 1$ in \eqref{eq:using_sigma}. Since we have $\eta_{N-1} \geq \frac{\alpha_2\beta}{7.5L_1}$ by Lemma~\ref{lem:stepsize_const_bound},
    we further have $\sum_{k=0}^{N-1} l_k \leq 2N -1 + \log_{1/\beta}(\frac{7.5\sigma_0 L_1}{\alpha_2})$ 
    from~\eqref{eq:bound_on_calls}. Note that each line search step consists of one operator evaluation. Additionally, in each iteration of Algorithm~\ref{alg:Full_Equasi-Newton}, we also need to evaluate $\vF(\vz_k)$. Thus, we conclude that the total number of gradient evaluations is bounded by $3N+\log_{1/\beta}(\frac{7.5\sigma_0 L_1}{\alpha_2})$.  
  \end{proof}

  Theorem~\ref{thm:line_search} shows that
  the total number of operator evaluations is upper bounded by $3N + \bigO(\log(\sigma_0 L_1))$. Thus, when $N$ is sufficiently large, the average number of operator evaluations per iteration can be bounded by a constant close to 3.
   
  In the next theorem, we further characterize the total number of matrix-vector products required in the subroutines. 

  \begin{theorem}[Matrix-vector product complexity]
    \label{thm:computational_cost}
    Let $N_\epsilon$ denote the minimum number of iterations required by Algorithm~\ref{alg:Full_Equasi-Newton} to find an $\epsilon$-accurate solution according to Theorem~\ref{thm:main}. %
    \begin{enumerate}[(a)]
      \item For Problems~\eqref{eq:monotone} and~\eqref{eq:minimax}, QNPE requires ${\bigO\left(N_{\epsilon}{\frac{L_1}{\mu}}\log \left(\frac{L_1 \|\vz_0-\vz^*\|^2}{\mu \epsilon}\right)\right)}$ matrix-vector products for $\mathsf{LinearSolver}$. Moreover, it requires ${\bigO (N_\epsilon \sqrt{\frac{L_1}{\mu}} \log (\frac{dN_{\epsilon}^2}{p^2}))}$ matrix-vector products for $\mathsf{ExtEvec}$ and $\mathsf{MaxSvec}$. 
      \item For Problem~\eqref{eq:minimization},
      QNPE requires ${\bigO\left(N_{\epsilon}\sqrt{\frac{L_1}{\mu}}\log \left(\frac{L_1 \|\vz_0-\vz^*\|^2}{\mu \epsilon}\right)\right)}$ matrix-vector products for $\mathsf{LinearSolver}$. Moreover, it requires computing ${\bigO (N_\epsilon \sqrt{\frac{L_1}{\mu}} \log (\frac{dN_{\epsilon}^2}{p^2}))}$ matrix-vector products for $\mathsf{ExtEvec}$. %
    \end{enumerate}
  \end{theorem}
  For the nonlinear equation in \eqref{eq:monotone} and the minimax problem in \eqref{eq:minimax}, 
  Theorem~\ref{thm:computational_cost}(a) demonstrates that the total number of matrix-vector products can be bounded by  $\tilde{\mathcal{O}}(\frac{L_1}{\mu}N_{\epsilon})$, ignoring logarithmic factors. Moreover, for the minimization problem in~\eqref{eq:minimization}, Theorem~\ref{thm:computational_cost}(b) shows that the dependence on the condition number can be further improved from $\frac{L_1}{\mu}$ to $\sqrt{\frac{L_1}{\mu}}$.

  To prove Theorem~\ref{thm:computational_cost}, we first present the following proposition regarding the convergence of Subroutine~\ref{alg:CGLS}.

  \begin{proposition}\label{prop:CGLS}
    Let $\{\vs_k\}_{k\geq 0}$ and $\{\vr_k\}_{k\geq 0}$ be generated by Subroutine~\ref{alg:CGLS}. Then the following holds:
    \begin{enumerate}[(a)]
      \item If $\mA$ is non-symmetric, define $\kappa(\mA) = \sqrt{\frac{\lambda_{\mathrm{max}}(\mA^\top \mA)}{\lambda_{\mathrm{min}}(\mA^\top \mA)}}$ and we have 
      $\|\vr_k\| \leq 2\left(\frac{{\kappa(\mA)}-1}{{\kappa(\mA)}+1}\right)^k\|\vr_0\|$. %
      \blue{If $\mA$ is symmetric, define $\kappa(\mA) = \frac{\lambda_{\mathrm{max}}(\mA)}{\lambda_{\mathrm{min}}(\mA)}$ and we have 
      $\|\vr_k\| \leq 2\left(\frac{\sqrt{\kappa(\mA)}-1}{\sqrt{\kappa(\mA)}+1}\right)^k\|\vr_0\|$.
      }
      \item {We have $\|\vs_k\| > \|\vs_{k-1}\|$ for all $k\geq 1$.} 
    \end{enumerate}
  \end{proposition}
  \begin{proof}
  When $\mA$ is non-symmetric, Subroutine~\ref{alg:CGLS} employs the CGLS method. See \cite[Chapter 7.4.2]{bjoerck1996numerical} for the proof of (a) and \cite[Theorem 2.1]{steihaug1983conjugate} for the proof of (b). When $\mA$ is symmetric, Subroutine~\ref{alg:CGLS} employs the conjugate residual method. See \cite[{Section 3.1}]{greenbaum1997iterative} for the proof of Part (a) and \cite[{Theorem 2.1.6}]{Fong2011} for the proof of Part (b). 
  \end{proof}

  As a corollary of Proposition~\ref{prop:CGLS}, we can upper bound the total number of matrix-vector products when Subroutine~\ref{alg:CGLS} returns for given inputs $\mA$, $\vb$, and $\rho$. 
  \begin{lemma}\label{lem:conjugate_residual}
    Given the inputs $\mA \in \reals^{d\times d}$, $\vb \in \reals^d$, and $\rho>0$. When Subroutine~\ref{alg:CGLS} returns, the total number of matrix-vector product evaluations can be bounded by  
     $ 2\kappa(\mA) \log \left(\frac{2\lambda_{\max}(\mA^\top\mA)}{\rho\sqrt{\lambda_{\min}(\mA^\top\mA)}}\right)$ if $\mA$ is non-symmetric, and $\sqrt{\kappa(\mA)} \log \left(\frac{2\lambda_{\max}(\mA)}{\rho}\right)$ if $\mA$ is symmetric. 
  \end{lemma}
  \begin{proof}
    We first consider the case where $\mA$ is non-symmetric. From the update rule of~Subroutine~\ref{alg:CGLS}, we can compute that $\vs_1 = \frac{\|\mA^\top \vb\|^2}{\|\mA \mA^\top\vb\|^2} \mA^\top\vb$, which implies that 
    \begin{equation*}
    \|\vs_1\| = \frac{\|\mA^\top\vb\|^3}{\|\mA \mA^\top\vb\|^2} \geq \frac{\|\mA^\top \vb\|}{\lambda_{\max}(\mA \mA^\top )} \geq  \frac{\sqrt{\lambda_{\min}(\mA\mA^\top)}\|\vb\|}{\lambda_{\max}(\mA \mA^\top)} = \frac{\sqrt{\lambda_{\min}(\mA\mA^\top)}\|\vr_0\|}{\lambda_{\max}(\mA \mA^\top)}.
    \end{equation*}
    Since $\|\vs_k\|$ is strictly increasing (cf. Proposition~\ref{prop:CGLS}(b)), we have $\|\vs_k\| \geq \|\vs_1\| = \frac{\sqrt{\lambda_{\min}(\mA\mA^\top)}\|\vr_0\|}{\lambda_{\max}(\mA \mA^\top)}$ for any $k \geq 1$. Thus, if $\|\vr_k\|\leq \rho \frac{\sqrt{\lambda_{\min}(\mA\mA^\top)}\|\vr_0\|}{\lambda_{\mathrm{max}}(\mA \mA^\top)}$, we obtain that $\|\vr_k\|\leq \frac{\rho\sqrt{\lambda_{\min}(\mA\mA^\top)} \|\vr_0\|}{\lambda_{\mathrm{max}}(\mA \mA^\top)} \leq \rho \|\vs_k\|_2$. Moreover, using Proposition~\ref{prop:CGLS}, we obtain that $\|\vr_k\|_2 \leq \rho \|\vs_k\|_2$ if 
    \begin{equation*}
      2\left(\frac{{\kappa(\mA)}-1}{{\kappa(\mA)}+1}\right)^k \leq \frac{\rho \sqrt{\lambda_{\mathrm{min}}(\mA^\top\mA)}}{\lambda_{\mathrm{max}}(\mA^\top\mA)} \quad\Leftrightarrow\quad 
      k \geq \frac{\log \left(\frac{2\lambda_{\max}(\mA^\top\mA)}{\rho\sqrt{\lambda_{\min}(\mA^\top\mA)}}\right)}{\log\left(\frac{{\kappa(\mA)}+1}{{\kappa(\mA)}-1}\right)}.
    \end{equation*}
    Since $\log(x) \geq (x-1)/x$ for all $x >0$, we have  $\log\left(\frac{{\kappa(\mA)}+1}{{\kappa(\mA)}-1}\right) \geq \frac{2}{{\kappa(\mA)}+1} \geq \frac{1}{{\kappa(\mA)}}$. Finally, note that for a non-symmetric matrix $\mA$, Subroutine~\ref{alg:CGLS} requires two matrix-vector products per iteration. This completes the proof in the non-symmetric case.
    
    Next, we consider the case where $\mA$ is symmetric. From the update rule of Subroutine~\ref{alg:CGLS}, we can compute that $\vs_1 = \frac{\vb^\top \mA \vb}{\|\mA\vb\|_2^2} \vb$, which implies
    \begin{equation*}
      \|\vs_1\| = \|\vb\|\cdot \frac{\|\mA^{1/2}\vb\|^2}{(\mA^{1/2}\vb)^\top \mA (\mA^{1/2}\vb)} \geq \frac{\|\vb\|}{\lambda_{\max}(\mA)} = \frac{\|\vr_0\|}{\lambda_{\max}(\mA)}.
    \end{equation*}
    Since $\|\vs_k\|$ is strictly increasing (cf. Proposition~\ref{prop:CGLS}(b)), we have $\|\vs_k\| \geq \|\vs_1\| = \frac{\|\vr_0\|}{\lambda_{\max}(\mA)}$ for any $k \geq 1$. Thus, if $\|\vr_k\|\leq \rho \frac{\|\vr_0\|_2}{\lambda_{\mathrm{max}}(\mA) }$, we obtain that $\|\vr_k\|_2\leq \rho \frac{\|\vr_0\|_2}{\lambda_{\mathrm{max}}(\mA) }\leq \rho \|\vs_k\|_2$. Combining this with Proposition~\ref{prop:CGLS}, we similarly obtain  $\|\vr_k\|_2 \leq \rho \|\vs_k\|_2$ if 
    \begin{equation*}
      2\left(\frac{\sqrt{\kappa(\mA)}-1}{\sqrt{\kappa(\mA)}+1}\right)^k \leq \frac{\rho}{\lambda_{\mathrm{max}}(\mA)} \quad\Leftarrow\quad 
      k \geq {\sqrt{\kappa(\mA)}\log \left(\frac{2\lambda_{\max}(\mA)}{\rho}\right)}
    \end{equation*}
    This completes the proof.
  \end{proof}

Using Lemma~\ref{lem:conjugate_residual}, we are ready to prove Theorem~\ref{thm:computational_cost}. 
\begin{proof}[Proof of Theorem~\ref{thm:computational_cost}] 
First, we consider Problems~\eqref{eq:monotone} and \eqref{eq:minimax}. In this case, the Jacobian approximation matrices $\{\mB_k\}$ are non-symmetric. Consider the $k$-th iteration. Note that in each call of $\mathsf{LinearSolver}$ in Subroutine~\ref{alg:ls}, the inputs are given by $\mA = \mI + \eta_{+} \mB_k$ and $\rho = \alpha_1 \sqrt{1+\eta_+\mu}$, with $\eta_{+} \leq \sigma_k$. 
Moreover, recall from Corollary~\ref{lem:feasible_Bk} that $\frac{1}{2}(\mB_k + \mB_k^\top )\succeq \frac{\mu}{2} \mI$ and $\|\mB_k\|_{\op} \leq 6.5L_1$. 
Thus, we have $ \sqrt{\lambda_{\max}(\mA^\top \mA)} = \|\mA\|_{\op} \leq 1 + \eta_{+}\|\mB_k\|_{\op}\leq 1+6.5\eta_{+} L_1$. Moreover, we also have $ \frac{1}{2}(\mA + \mA^\top) \succeq (1+\eta_+\frac{\mu}{2}) \mI$, which further implies $ \lambda_{\min}( \mA^\top \mA) \geq (1+\eta_+\frac{\mu}{2})^2$ by Lemma~\ref{lem:eig_singular} in Appendix~\ref{appen:supporting}. 
Hence, we can conclude that 
$\kappa(\mA)  = \sqrt{\frac{\lambda_{\max}(\mA^\top \mA)}{\lambda_{\min}(\mA^\top \mA)}} \leq \frac{1+6.5\eta_{+} L_1}{1+\eta_+\frac{\mu}{2}} \leq \frac{13 L_1}{\mu}$. 
Furthermore, by Lemma~\ref{lem:conjugate_residual}, the number of matrix-vector product evaluations in each call of $\mathsf{LinearSolver}$ can be bounded by 
    \begin{align*}
      \overline{\mathsf{MV}}_k &\leq 2\kappa(\mA) \log \left(\frac{2\lambda_{\max}(\mA^\top\mA)}{\rho\sqrt{\lambda_{\min}(\mA^\top\mA)}}\right) \\
       &\leq {\frac{13L_1}{\mu}} \log \left(\frac{2}{\alpha_1} \frac{(1+6.5 \eta_+L_1)^2}{ \sqrt{1+\eta_+ \mu}(1+\eta_+\mu/2)} \right)\\
      &\leq {\frac{13L_1}{\mu}} \log \left(\frac{2}{\alpha_1}\frac{(1+6.5 \eta_+L_1)^2}{(1+\eta_+\mu/2)^2}\right) + {\frac{13L_1}{\mu}} \log \left(\frac{(1+\eta_+\mu/2)^2}{\sqrt{1+\eta_+ \mu}(1+\eta_+\mu/2)}\right) \\
      &\leq {\frac{26L_1}{\mu}} \log \left(\sqrt{\frac{2}{\alpha_1}}\frac{13 L_1}{\mu}\right) + {\frac{13L_1}{2\mu}} \log \left(1+\frac{\eta_+\mu}{2}\right)
      . 
    \end{align*}
    Moreover, since we have $\eta_{+} \leq \sigma_k = \eta_{k-1}/\beta$ for $k\geq 1$, we further get 
    \begin{align*}
      \overline{\mathsf{MV}}_k &\leq {\frac{26L_1}{\mu}} \log \left(\sqrt{\frac{2}{\alpha_1}}\frac{13 L_1}{\mu}\right) + {\frac{13L_1}{2\mu}} \log \left(1+\frac{\eta_{k-1}\mu}{2\beta}\right) \\
      &\leq {\frac{26L_1}{\mu}} \log \left(\sqrt{\frac{2}{\alpha_1 \beta}}\frac{13 L_1}{\mu}\right) + {\frac{13L_1}{2\mu}} \log \left(1+\frac{\eta_{k-1}\mu}{2}\right), 
    \end{align*}
    where the last inequality is due to the fact that $\log \left(1+\frac{\eta_{k-1}\mu}{2\beta}\right) = \log \left(\beta+\frac{\eta_{k-1}\mu}{2}\right) + \log \frac{1}{\beta} \leq \log \left(1+\frac{\eta_{k-1}\mu}{2}\right) + \log \frac{1}{\beta}$. 
    Let $l_k$ denote the number of line search steps in iteration $k$, and then we can bound the total number of matrix-vector products by $\sum_{k=0}^{N_\epsilon-1} l_k \cdot \overline{\mathsf{MV}}_k$. 
    Moreover, from the proof of Theorem~\ref{thm:line_search}, we know that $l_k = \log_{1/\beta}(\frac{\sigma_k}{\eta_k})+1$.
    For $k=0$, we have 
    $l_0 \leq \log_{1/\beta}\Bigl(\frac{\sigma_0}{\eta_0}\Bigr)+1 \leq \log_{1/\beta}\Bigl(\frac{\sigma_0L_1}{\alpha_2\beta}\Bigr)+1$,
    and 
    $\overline{\mathsf{MV}}_0 \leq {\frac{26L_1}{\mu}} \log \left(\sqrt{\frac{2}{\alpha_1}}\frac{13 L_1}{\mu}\right) + {\frac{13L_1}{2\mu}} \log \left(1+\frac{\sigma_0\mu}{2}\right)$,
    where we used that $\eta_0 >\frac{\alpha_2 \beta}{7.5 L_1}$ by Lemma~\ref{lem:stepsize_const_bound}. 
    Furthermore, %
    we first show that  
    \begin{equation}\label{eq:stepsize_upper_bound}
      \prod_{k=0}^{N_\epsilon-2} (1+2\eta_k\mu) \leq \frac{\|\vz_0-\vz^*\|^2}{\epsilon}.
    \end{equation}
    To see this, note that by Proposition~\ref{prop:HPE}, we have $\|\vz_{N}-\vz^*\|^2 \leq \|\vz_0-\vz^*\|^2 \prod_{k=0}^{N-1} (1+2\eta_k\mu)^{-1}$. Then \eqref{eq:stepsize_upper_bound} follows from the fact that $N_{\epsilon}$ is the minimum number of iterations to achieve $\|\vz_N-\vz^*\|^2 \leq \epsilon$. Thus, the total number of matrix-vector products $\mathsf{MV}_{\mathrm{tol}} := \sum_{k=1}^{N_\epsilon-1} l_k \cdot \overline{\mathsf{MV}}_k$ can be bounded by: 
    \begin{align*}
      \mathsf{MV}_{\mathrm{tol}}
      &\leq {\frac{26L_1}{\mu}} \log \left(\sqrt{\frac{2}{\alpha_1 \beta}}\frac{13 L_1}{\mu}\right)\sum_{k=1}^{N_\epsilon-1} l_k + {\frac{13L_1}{2\mu}} \sum_{k=1}^{N_\epsilon-1} \log(1+2{\eta_{k-1}\mu}) \cdot l_k \\
      &\leq {\frac{26L_1}{\mu}} \log \left(\sqrt{\frac{2}{\alpha_1 \beta}}\frac{13 L_1}{\mu}\right) \sum_{k=1}^{N_\epsilon-1} l_k + {\frac{13L_1}{2\mu}}\sum_{k=1}^{N_\epsilon-1} \log(1+2{\eta_{k-1}\mu}) \cdot \sum_{k=1}^{N_\epsilon-1} l_k \\
      &\leq  {\frac{26L_1}{\mu}}\log \left(\frac{13 \sqrt{2} L_1 \|\vz_0-\vz^*\|^2}{\sqrt{\alpha_1\beta }\mu \epsilon}\right) \cdot \left(2N_{\epsilon}+\log_{1/\beta} \frac{7.5 \sigma_0 L_1}{\alpha_2} \right),
    \end{align*}
    where we used \eqref{eq:stepsize_upper_bound} and Theorem~\ref{thm:line_search} in the last inequality. Hence, we conclude that $\textsf{LinearSolver}$ requires ${\bigO\left(N_{\epsilon}{\frac{L_1}{\mu}}\log \left(\frac{L_1 \|\vz_0-\vz^*\|^2}{\mu \epsilon}\right)\right)}$ matrix-vector products in total.  

    Next, we consider Problem~\eqref{eq:minimization}. Note that in this case, the Jacobian approximation matrices $\{\mB_k\}$ are symmetric and positive definite. We follow similar arguments as in the non-symmetric case. Consider the $k$-th iteration and in each call of $\mathsf{LinearSolver}$, the inputs are given by $\mA = \mI + \eta_{+} \mB_k$ and $\rho = \alpha_1 \sqrt{1+\eta_+\mu}$, with $\eta_{+} \leq \sigma_k$. 
    Therefore, we can bound  $\frac{\lambda_{\max}(\mA)}{\lambda_{\min}(\mA)} = \frac{1+\eta_+\lambda_{\max}(\mB_k)}{1+\eta_+\lambda_{\min}(\mB_k)} \leq \frac{\lambda_{\max}(\mB_k)}{\lambda_{\min}(\mB_k)}$. Moreover, Corollary~\ref{lem:feasible_Bk} shows that $\frac{\mu}{2}\mI \preceq \mB_k \preceq 6.5L_1 \mI$, and this further implies that $\kappa(\mA) \leq \frac{13 L_1}{\mu}$. 
    Hence, by Lemma~\ref{lem:conjugate_residual}, the number of matrix-vector product evaluations in each call of $\mathsf{LinearSolver}$ can be bounded by 
    \begin{align*}
      \overline{\mathsf{MV}}_k &\leq \sqrt{\kappa(\mA)} \log \left(\frac{2\lambda_{\max}(\mA)}{\rho}\right)\\
      &\leq \sqrt{\frac{13L_1}{\mu}} \log \left(\frac{2(1+6.5\eta_+L_1)}{\alpha_1\sqrt{1+\eta_+ \mu}}\right) \\
      &\leq \sqrt{\frac{13L_1}{\mu}} \log \left(\frac{2(1+6.5\eta_+L_1)}{\alpha_1(1+\eta_+ \mu)}\right) + \frac{1}{2}\sqrt{\frac{13L_1}{\mu}} \log \left(1+\eta_+ \mu\right). 
    \end{align*}
    Moreover, since $\frac{1+6.5\eta_+L_1}{1+\eta_+ \mu} \leq \frac{6.5 L_1}{\mu}$ and $\eta_{+} \leq \sigma_k = \frac{\eta_{k-1}}{\beta}$ for $k\geq 1$, we further get
    \begin{align*}
      \overline{\mathsf{MV}}_k &\leq \sqrt{\frac{13L_1}{\mu}} \log \left(\frac{13L_1}{\alpha_1\mu}\right) + \frac{1}{2}\sqrt{\frac{13L_1}{\mu}} \log \left(1+ \frac{\eta_{k-1}\mu}{\beta}\right)  \\
      &\leq \sqrt{\frac{13L_1}{\mu}} \log \left(\frac{13 L_1}{\alpha_1 \sqrt{\beta}\mu}\right) + \frac{1}{2}\sqrt{\frac{13L_1}{\mu}} \log \left(1+{\eta_{k-1}\mu}\right). 
    \end{align*}
    Recall that $l_k = \log_{1/\beta}(\frac{\sigma_k}{\eta_k})+1$ denote the number of line search steps in iteration $k$, and the total number of matrix-vector products can be bounded by $\sum_{k=0}^{N_\epsilon-1} l_k \cdot \overline{\mathsf{MV}}_k$. Similar to the non-symmetric setting, for $k=0$, we have $l_0 \leq \log_{1/\beta}\Bigl(\frac{\sigma_0L_1}{\alpha_2\beta}\Bigr)+1$ and $\overline{\mathsf{MV}}_0  \leq \sqrt{\frac{13L_1}{\mu}} \log \left(\frac{13L_1}{\alpha_1\mu}\right) + \frac{1}{2}\sqrt{\frac{13L_1}{\mu}} \log \left(1+ {\sigma_0\mu}\right)$. Furthermore, 
    the total number of matrix-vector products $\mathsf{MV}_{\mathrm{tol}} := \sum_{k=1}^{N_\epsilon-1} l_k \cdot \overline{\mathsf{MV}}_k$ can be bounded by:
    \begin{align*}
      \mathsf{MV}_{\mathrm{tol}}
      &\leq \sqrt{\frac{13L_1}{\mu}} \log \left(\frac{13 L_1}{\alpha_1 \sqrt{\beta}\mu}\right)\sum_{k=1}^{N_{\epsilon}-1} l_k + \sqrt{\frac{13L_1}{4\mu}}  \sum_{k=1}^{N_\epsilon-1} \log(1+2{\eta_{k-1}\mu}) \cdot l_k \\
      &\leq \sqrt{\frac{13L_1}{\mu}} \log \left(\frac{13 L_1}{\alpha_1 \sqrt{\beta}\mu}\right) \sum_{k=1}^{N_\epsilon-1} l_k + \sqrt{\frac{13L_1}{4\mu}}\sum_{k=1}^{N_\epsilon-1} \log(1+2{\eta_{k-1}\mu}) \cdot \sum_{k=1}^{N_\epsilon-1} l_k \\
      &\leq  \sqrt{\frac{13L_1}{\mu}}\log \left(\frac{13 L_1 \|\vz_0-\vz^*\|^2}{\alpha_1\sqrt{\beta }\mu \epsilon}\right) \cdot \left(2N_{\epsilon}+\log_{1/\beta} \frac{7.5 \sigma_0 L_1}{\alpha_2} \right),
    \end{align*}
    where we used \eqref{eq:stepsize_upper_bound} and Theorem~\ref{thm:line_search} in the last inequality. Hence, we obtain that the $\textsf{LinearSolver}$ oracle requires ${\bigO\left(N_{\epsilon}\sqrt{\frac{L_1}{\mu}}\log \left(\frac{L_1 \|\vz_0-\vz^*\|^2}{\mu \epsilon}\right)\right)}$ matrix-vector products in total.  

    Finally, we bound the number of matrix-vector products for $\mathsf{ExtEvec}$ and $\mathsf{MaxSvec}$ in all cases. Since the parameters are selected as $\delta = \frac{\mu}{2L_1}$ and $q_t = \frac{p}{2.5(t+1)\log^2(t+1)} \leq \frac{p}{2.5N_{\epsilon}\log^2(N_{\epsilon})}$, it follows from Lemmas~\ref{lem:extevec_bound} and~\ref{lem:maxsvec_bound} that the total number of matrix-vector products is bounded by 
    $\bigO\left(\sum_{t=1}^T \sqrt{\frac{1+\delta_t}{\delta_t}}\log\frac{d}{q_t^2} \right) = \bigO\left(N_\epsilon \sqrt{\frac{L_1}{\mu}} \log (\frac{dN_{\epsilon}^2}{p^2})\right)$.
    The proof is complete.
\end{proof}

\subsection{Monotone setting}
\label{subsec:monotone}

\subsubsection{Convergence rate analysis}

Next, we present our convergence result when $\vF$ is only monotone.
\begin{theorem}\label{thm:main_monotone}
  Suppose Assumptions~\ref{assum:monotone}, ~\ref{assum:operator_lips} and~\ref{assum:jacobian_lips} hold. Let $\{\vz_k\}$ and $\{\hat{\vz}_k\}$ be the iterates generated by Algorithm~\ref{alg:Full_Equasi-Newton} using the line search scheme in Subroutine~\ref{alg:ls}, where $\alpha_1, \alpha_2 \in (0, \frac{1}{2})$, $\beta \in (0,1)$, and $\sigma_0 \geq  \frac{\alpha_2\beta}{5L_1}$.
  In addition, the Jacobian approximation matrices are updated in Subroutine~\ref{alg:jacobian_approx} with \textbf{Option II}, where {$\rho = \frac{1}{121}$}, $\delta_t = {\frac{1}{2(t+1)^{\nicefrac{1}{4}}}}$, and $q_t = \frac{p}{2.5(t+1)\log^2(t+1)}$ for $t \geq 1$. With probability at least $1-p$, the following holds:  
    \begin{enumerate}[(a)]
      \item For any $k\geq 0$, we have $\|\vz_{k+1}-\vz^*\| \leq \|\vz_k-\vz^*\|$. Moreover, define the averaged iterate $\bar{\vz}_k$ by $\bar{\vz}_k = \frac{\sum_{i=0}^{k-1} \eta_i \hat{\vz}_i}{\sum_{i=0}^{k-1} \eta_i}$. Then for any compact set $\calD \subset \reals^d$, we have
      $\mathrm{Gap}(\bar{\vz}_k; \calD) \leq \frac{5 L_1 \max_{\vz \in \calD}\|\vz_0-\vz\|^2}{2\alpha_2\beta k}$. 
      \item Moreover, recall the definition of $C$ in \eqref{eq:def_c1} and define $Q_1$ and $Q_2$ as 
      \begin{align}
        Q_1 &= {9\sqrt{2}\frac{\|\mB_0-{\nabla \vF(\vz_0)}\|_F}{L_1}} + {\sqrt{2C}\frac{L_2 \| \vz_0-\vz^*\|}{L_1}}, \label{eq:def_Q1} \\
         Q_2 &= \sqrt{\frac{648\sqrt{2} d L_2\|\vz_0-\vz^*\|}{\sqrt{1-\alpha_1-\alpha_2}L_1}} + 6\sqrt{2}.  \label{eq:def_Q2}
      \end{align}
      Then for any $k\geq 0$, we have 
      \begin{equation*}
        \mathrm{Gap}(\bar{\vz}_k; \calD)\leq  \frac{L_1\max_{\vz \in \calD}\|\vz_0-\vz\|^2}{2(1-\beta)\alpha_2 \beta} \left(\frac{5}{k^2} + \frac{Q_1}{k^{1.5}} + \frac{Q_2}{k^{1.25}} \right).  
      \end{equation*}
    \end{enumerate}
  \end{theorem}
  Similar to the strongly monotone setting, Theorem~\ref{thm:main_monotone} demonstrates two global convergence rates for QNPE. Part (a) shows that QNPE converges at least at a rate of $\bigO\left(\frac{1}{k}\right)$, which matches the rate of EG and is known to be optimal in the regime when the number of iterations $k$ is $\bigO(d)$~\cite{ouyang2021lower}. Moreover, Part (b) shows that QNPE achieves a rate of $\bigO\left(\frac{1}{k^2} + \frac{Q_1}{k^{1.5}} + \frac{Q_2}{k^{1.25}}\right)$, where $Q_1 = \bigO\left(\frac{\|\mB_0 - \nabla \vF(\vz_0)\|_F +L_2 \|\vz_0-\vz^*\|}{L_1}\right)$ and $Q_2 = \bigO\left(\sqrt{\frac{d L_2\|\vz_0-\vz^*\|}{L_1}}\right)$. Since Lemma~\ref{lem:jacobian} implies that $\|\mB_0 - \nabla \vF(\vz_0)\|_F \leq L_1\sqrt{d}$, in the worst case, $ Q_1 = \bigO(\sqrt{d})$. Hence, the leading term is $\bigO(\frac{\sqrt{d}}{k^{1.25}})$, which outperforms the rate in Part (a) when $k = \Omega(d^2)$. To the best of our knowledge, this is the first result demonstrating a theoretical advantage of quasi-Newton methods over EG for solving monotone nonlinear equations. 

  The rest of this section is devoted to the proof of Theorem~\ref{thm:main_monotone}. Using the same union bound argument as in Theorem~\ref{thm:main}, we assume that every call of $\mathsf{ExtEvec}$ and $\mathsf{MaxSvec}$ is successful, which holds with probability at least $1-p$. 

  \begin{proof}[Proof of Theorem~\ref{thm:main_monotone}(a)]
    To begin with, note that Subroutine~\ref{alg:jacobian_approx} with \textbf{Option II} guarantees that  $\hat{\mB}_k \in \mathcal{C}$ by Lemma~\ref{lem:regret_reduction}.  As a corollary of Lemma~\ref{lem:transform}, we have 
    $\frac{1}{2}(\mB_k + \mB_k^\top) \succeq 0$ and $\|\mB_k\|_{\op} \leq 4L_1$ for all $k\geq 0$. Thus, following the same arguments as in Lemma~\ref{lem:stepsize_const_bound}, we can show that $\eta_k \geq \frac{\alpha_2 \beta}{5L_1}$ by induction.  
    This leads to $\sum_{k=0}^{N-1} \eta_k \geq \frac{\alpha_2\beta}{5 L_1} N$, and we obtain from Proposition~\ref{prop:HPE}(b) that $\gap(\bar{\vz}_N; \calD) \leq \frac{\max_{\vz \in \calD}\,\|\vz_0-\vz\|^2}{2 \sum_{k=0}^{N-1} \eta_k} \leq \frac{5 L_1\|\vz_0-\vz\|^2}{2\alpha_2\beta N}$. This proves the first convergence rate in (a). 
  \end{proof}

  Next, we prove the second convergence rate in (b). Similar to the proof of Theorem~\ref{thm:main}, our starting point is the observations in~\eqref{eq:after_Cauchy_Schwarz} and~\eqref{eq:goal}, which implies that it is sufficient to prove an upper bound on the cumulative loss $\sum_{k=0}^{N-1} \ell_k(\mB_k)$. Again, recall that $\ell_k(\mB_k) = 0$ when $k \notin \calB$ by the definition in \eqref{eq:loss_of_jacobian}, 
  and we relabel the indices in $\calB$ by $t = 0, \dots, T-1$ with $T \leq N$. 

  \begin{proof}[Proof of Theorem~\ref{thm:main_monotone}(b)]
    
  We consider the following three steps.

\vspace{.1em}
\noindent\textbf{Step 1:} 
First, we bound the cumulative loss $\sum_{t=0}^{T-1}\ell_t(\mB_t)$ incurred by our online learning algorithm in Subroutine~\ref{alg:jacobian_approx}. 
However, different from the proof in Theorem~\ref{thm:main}, we will upper bound the \emph{dynamic regret} instead of the static regret. This is because, in the strongly monotone setting, the iterates are guaranteed to converge at least linearly to the optimal solution $\vz^*$ (Theorem~\ref{thm:main}(a)). This results in less variation in the loss functions $\{\ell_t\}_{t=0}^{T-1}$, and in particular we can show that $\sum_{t=0}^{T-1} \ell_t(\mH^*)$ is bounded (Lemma~\ref{lem:comparator}). In contrast, without linear convergence, we need to consider a time-varying sequence $\{\mH_t\}_{t=0}^{T-1}$ to control the cumulative loss. Specifically, we present a ``small-loss bound'' similar to Lemma~\ref{lem:small_loss} in the following lemma. %

\begin{lemma}
  \label{lem:dynamic_small_loss}
  In Subroutine~\ref{alg:jacobian_approx}, choose \textbf{Option II} with $\rho = \frac{1}{81}$ and $\delta_t = \frac{1}{2(t+1)^{\nicefrac{1}{4}}}$. 
  For any sequence of matrices $\{\mH_t\}_{t=0}^{T-1}$ such that $\mH_t \in \calZ$, we have 
    $\sum_{t=0}^{T-1} \ell_t({\mB}_t)  \leq 4\sum_{t=0}^{T-1} \ell_t({\mH_t}) + 162\|\mB_0-{\mH_0}\|_F^2  + 648\sqrt{d}L_1\sum_{t=0}^{T-2} \|\mH_{t+1} - \mH_t\|_{F} + 72L_1^2\sqrt{T}$.
\end{lemma}

\begin{proof}
Recall that $\calZ$ is contained in the linear subspace $\calL$ (see \eqref{eq:feasible_set}) and $\mP$ denotes the orthogonal projection matrix associated with $\calL$.
By letting $\vx_t = \hat{\mB}_t$, $\vu = \hat{\mH}_t \triangleq \frac{1}{L_1}(\mH_t-L_1\mI)$, $\vg_t = \mG_t \triangleq \frac{1}{L_1}\mP \nabla \ell_t(\mB_t)$, $\tilde{\vg}_t = \tilde{\mG}_t$, $\vw_t = \mW_t$ in Lemma~\ref{lem:regret_reduction}, we obtain:  
\begin{enumerate}[(i)]
  \item $\hat{\mB}_t \in \mathcal{C}$, which means $\|\hat{\mB}_t\|_{\op} \leq 3$.
  \item It holds that 
  \begin{align}
    \!\!\!\!\!\!\!\!\!\langle \mG_t, \hat{\mB}_t-\hat{\mH}_t \rangle & \!\leq\! \frac{\|\mW_t-\hat{\mH}_t\|_F^2}{2\rho}\!-\!\frac{\|\mW_{t+1}\!-\!\hat{\mH}_t\|_F^2}{2\rho}\!+\!\frac{\rho}{2}\|\tilde{\mG}_t\|_F^2 \!-\! \delta_t \langle \mG_t, \hat{\mB}_t \rangle, \label{eq:linearized_loss_matrix_monotone} \\%\\
    \|\tilde{\mG}_t\|_F &\leq \|\mG_t\|_F + (1+\delta_t)|\langle \mG_t, \hat{\mB}_t\rangle|\|\mS_t\|_F . \label{eq:surrogate_gradient_bound_monotone}
  \end{align}
\end{enumerate}
First, note that $\|\mS_t\|_F \leq 1$ by Definition~\ref{def:extevec} and
$|\langle \mG_t, \hat{\mB}_t\rangle| \leq \|\mG_t\|_* \|\hat{\mB}_t\|_{\op} \leq 3 \|\mG_t\|_*$. %
By using $\delta_t = \frac{1}{2{(t+1)^{\nicefrac{1}{4}}}} \leq \frac{1}{2}$ together with \eqref{eq:surrogate_gradient_bound_monotone}, we get 
\begin{equation}\label{eq:surrogate_bound_loss_monotone}
\|\tilde{\mG}_t\|_F \leq \|\mG_t\|_F+3(1+\delta_t)\|\mG_t\|_* \leq 4.5\|\mG_t\|_* \leq \frac{9}{L_1}\sqrt{\ell_t({\mB}_t)},
\end{equation}
where we used $\mG_t = \frac{1}{L_1}\mP \nabla \ell_t(\mB_t)$ and Lemma~\ref{lem:loss} in the last inequality. 
Furthermore, since $\ell_t$ is convex and $\mB_t,\mH_t \in \calL$, we have $\ell_t({\mB}_t) - \ell_t(\mH_t) \leq L_1^2 \langle \mG_t, \hat{\mB}_t-\hat{\mH}_t \rangle$ as shown in \eqref{eq:convexity_subspace}.  
Therefore, by \eqref{eq:linearized_loss_matrix_monotone} and \eqref{eq:surrogate_bound_loss_monotone} we get $\ell_t({\mB}_t) - \ell_t(\mH_t) \leq \frac{L_1^2}{2\rho}\|\mW_t-\hat{\mH}_t\|_F^2-\frac{L_1^2}{2\rho}\|\mW_{t+1}-\hat{\mH}_t\|_F^2+\frac{\rho L_1^2}{2} \|\tilde{\mG}_t\|_F^2 +3L_1^2\delta_t\|\mG_t\|_{*} \leq \frac{L_1^2}{2\rho}\|\mW_t-\hat{\mH}_t\|_F^2-\frac{L_1^2}{2\rho}\|\mW_{t+1}-\hat{\mH}_t\|_F^2+40.5\rho\ell_t({\mB}_t) + 6L_1 \delta_t \sqrt{\ell_t({\mB}_t)}$. 
Since $\rho = \frac{1}{81}$, by rearranging and simplifying terms in the above inequality, we obtain
\begin{equation}\label{eq:implicit_loss_bound}
\ell_t({\mB}_t) \leq 2\ell_t({\mH_t})+ {81 L_1^2}\|\mW_t-\hat{\mH}_t\|_F^2-{81 L_1^2}\|\mW_{t+1}-\hat{\mH}_t\|_F^2 + 12L_1 \delta_t\sqrt{\ell_t({\mB}_t)}.
\end{equation}
We observe that the above upper bound on $\ell_t(\mB_t)$ is implicit since it appears on both sides of \eqref{eq:implicit_loss_bound}. To derive an explicit upper bound, we apply the following lemma. 
\begin{lemma}\label{lem:x_sqrtx}
  If the real number $x$ satisfies $x \leq A+B\sqrt{x}$, then we have $x \leq 2A+B^2$. 
\end{lemma}
\begin{proof}
  By using the assumption, we have $\left(\sqrt{x}-\frac{B}{2}\right)^2 \leq A+\frac{B^2}{4}$. 
  Hence, we obtain that $x \leq \left(\sqrt{A+\frac{B^2}{4}}+\frac{B}{2}\right)^2 \leq 2A+B^2$.
\end{proof}
By applying Lemma~\ref{lem:x_sqrtx} to \eqref{eq:implicit_loss_bound} with $A = 2\ell_t({\mH_t})+ {81 L_1^2}\|\mW_t-\hat{\mH}_t\|_F^2-{81 L_1^2}\|\mW_{t+1}-\hat{\mH}_t\|_F^2$ and $B =12L_1 \delta_t$, we get  
\begin{equation}\label{eq:upper_bound_on_ellt_dynamic}
  \ell_t(\mB_t) \leq 4\ell_t({\mH_t})+ {162 L_1^2}\left(\|\mW_t-\hat{\mH}_t\|_F^2-\|\mW_{t+1}-\hat{\mH}_t\|_F^2\right) + 144 L_1^2 \delta_t^2. 
\end{equation}
Furthermore, note that $\|\mW_{t+1}-\hat{\mH}_{t+1}\|_F^2-\|\mW_{t+1}-\hat{\mH}_t\|_F^2=(\|\mW_{t+1}-\hat{\mH}_{t+1}\|_F+\|\mW_{t+1}-\hat{\mH}_t\|_F)(\|\mW_{t+1}-\hat{\mH}_{t+1}\|_F-\|\mW_{t+1}-\hat{\mH}_t\|_F) \leq 4\sqrt{d} \|\hat{\mH}_{t+1}-\hat{\mH}_t\|_F = \frac{4\sqrt{d}}{L_1} \|{\mH}_{t+1}-{\mH}_t\|_F$,  
where in the last inequality we used that $\hat{\mH}_{t},\hat{\mH}_{t+1},\mW_{t+1}\in \mathcal{B}_{\sqrt{d}}(0)$ and the triangle inequality. 
  Therefore, we can write 
  $\sum_{t=0}^{T-1}\bigl(\|\mW_t-\hat{\mH}_t\|_F^2-\|\mW_{t+1}-\hat{\mH}_t\|_F^2\bigr) \leq \|\mW_0 - \hat{\mH}_0\|_F^2 + \sum_{t=0}^{T-2}\bigl(\|\mW_{t+1}-\hat{\mH}_{t+1}\|_F^2-\|\mW_{t+1}-\hat{\mH}_t\|_F^2\bigr) \leq \|\mW_0 - \hat{\mH}_0\|_F^2 + \frac{4\sqrt{d}}{L_1} \sum_{t=0}^{T-2} \|\mH_{t+1}-\mH_t\|_F$. 
By summing the inequality in \eqref{eq:upper_bound_on_ellt_dynamic} from $t=0$ to $T-1$, we obtain that $\sum_{t=0}^{T-1} \ell_t({\mB}_t) \leq 4 \sum_{t=0}^{T-1} \ell_t({\mH_t})+ {162 L_1^2}\|\mW_0-\hat{\mH}_0\|_F^2 + {648\sqrt{d}}L_1\sum_{t=0}^{T-2}  \|{\mH}_{t+1}-{\mH}_t\|_F + 144 L_1^2 \sum_{t=0}^{T-1} \delta_t^2$.  
Finally, note that $\mW_0 - \hat{\mH}_0 = \frac{1}{L_1}(\mB_0 - \mH_0)$ and $\sum_{t=0}^{T-1} \delta_t^2 = \frac{1}{4}\sum_{t=1}^T \frac{1}{t^{\nicefrac{1}{2}}} \leq \frac{1}{2}\sqrt{T}$. 
This completes the proof by substituting these results into the above inequality. 
\end{proof}

Note that in Lemma~\ref{lem:dynamic_small_loss}, we have the freedom to choose any competitor sequence $\{\mH_t\}_{t\geq 0}$ in the set $\mathcal{Z}$. To further obtain an explicit bound, we propose to select $\mH_t = \nabla \vF(\vz_t)$ for $t=0,\dots,T-1$, which leads to our next step.

\textbf{Step 2:} For the choice of $\mH_t = \nabla^2 F(\vz_t)$ for all $t\geq 0$, we upper bound the cumulative loss $\sum_{t=0}^{T-1} \ell_t({\mH}_t)$ and the path-length $\sum_{t=0}^{T-2} \|\mH_{t+1}-\mH_t\|_F$. To begin with, we present the following lemma. 

\begin{lemma}\label{lem:sum_of_squares_monotone}
  If Assumption~\ref{assum:monotone} holds, then $\sum_{k=0}^{N-1} \|\hat{\vz}_{k}-\vz_k\|^2 \!\leq\! \frac{\|\vz_0-\vz^*\|^2}{1-(\alpha_1+\alpha_2)}$ and $\sum_{k=0}^{N-1} \|{\vz}_{k+1}-\vz_k\|^2 \leq \frac{2\|\vz_0-\vz^*\|^2}{1-(\alpha_1+\alpha_2)}$. 
\end{lemma}
\begin{proof}
  Recall the inequality~\eqref{eq:one_step_monotone} in the proof of Proposition~\ref{prop:HPE} and that $\alpha = \alpha_1 + \alpha_2 < 1$. By rearranging the terms, we can further derive that 
  \begin{equation*}
    \frac{1-\alpha}{2}(\|\hat{\vz}_k - \vz_k\|^2 + \|\hat{\vz}_k-\vz_{k+1}\|^2) \leq \frac{1}{2}\|\vz_k- \vz^*\|^2 - \frac{1}{2}\|\vz_{k+1} - \vz^*\|^2.
\end{equation*} 
By summing the above inequality from $k=0$ to $k=N-1$, we obtain that 
\begin{equation}\label{eq:sum_of_squares_last}
  \frac{1-\alpha}{2} \sum_{k=0}^{N-1} (\|\hat{\vz}_k - \vz_k\|^2 + \|\hat{\vz}_k-\vz_{k+1}\|^2) 
  \leq \frac{1}{2}\|\vz_0-\vz^*\|^2.
\end{equation}
Since $\alpha < 1$, by dropping the second non-negative term from the left-hand side of~\eqref{eq:sum_of_squares_last},  this immediately leads to
the first inequality in Lemma~\ref{lem:sum_of_squares_monotone}.
 In addition, note that $\|\vz_{k}-\vz_{k+1}\|^2 = \|\vz_k-\hat{\vz}_{k} + \hat{\vz}_k-\vz_{k+1}\|^2 \leq 2\left(\|\vz_k-\hat{\vz}_{k}\|^2 + \|\hat{\vz}_k-\vz_{k+1}\|^2\right)$. 
 Thus, from~\eqref{eq:sum_of_squares_last}, we also have 
\begin{equation*}
  \sum_{k=0}^{N-1} \|\vz_{k}-\vz_{k+1}\|^2 \leq 2\sum_{k=0}^{N-1} \left(\|\hat{\vz}_{k}-\vz_k\|^2 +  \|\hat{\vz}_k-\vz_{k+1}\|^2 \right)\leq \frac{2\|\vz_0-\vz^*\|^2}{1-\alpha},
\end{equation*}
which proves the second inequality in Lemma~\ref{lem:sum_of_squares_monotone}.
The proof is now complete.
\end{proof}

With the help of Lemma~\ref{lem:sum_of_squares_monotone}, we are set to upper bound the cumulative loss $\sum_{t=0}^{T-1} \ell_t({\mH}_t)$ and the path-length $\sum_{t=0}^{T-2} \|\mH_{t+1}-\mH_t\|_F$.

  \begin{lemma}\label{lem:comparator_dynamic}
  Recall that $\mH_t = \nabla F(\vz_t)$ for $t=0,\dots,T-1$, and the constant $C$ is defined in \eqref{eq:def_c1}. Also, let $N$ be the total number of iterations. Then $\sum_{t=0}^{T-1} \ell_t(\mH_t) \leq \frac{C L_2^2}{2} \|\vz_0-\vz^*\|^2$ and $\sum_{t=0}^{T-2}  \|{\mH}_{t+1}-{\mH}_t\|_F \leq   L_2\sqrt{\frac{2d N}{1-(\alpha_1+\alpha_2)}}\|\vz_0-\vz^*\|$. 
  \end{lemma}
  \begin{proof}
    By the fundamental theorem of calculus, we have $\vu_t = \vF({\tilde{\vz}_t})-\vF(\vz_t) = \bar{\mH}_t \vs_t$, where $\bar{\mH}_t = \int_{0}^1 \nabla \vF(\vz_t+ \lambda \vs_t) \,d\lambda$. 
    Moreover, using the triangle inequality, we have $\|\bar{\mH}_t-\mH_t\|_{\op} 
    \leq  \int_{0}^1 \left\|\nabla \vF(\vz_t+ \lambda \vs_t) - \nabla \vF(\vz_t) \right\|_{\op}\,d\lambda %
    \leq \int_{0}^1 L_2\lambda\|\vs_t\| \,d\lambda = \frac{L_2}{2} \|\vs_t\|$, 
    where we used Assumption~\ref{assum:jacobian_lips} in the second inequality. Furthermore, we can bound $\ell_t(\mH_t) =  \frac{\|\vu_t-\mH_t\vs_t\|^2}{\|\vs_t\|^2} = \frac{\|(\bar{\mH}_t-\mH_t)\vs_t\|^2}{\|\vs_t\|^2} \leq \frac{\|\bar{\mH}_t-\mH_t\|^2_{\op}\|\vs_t\|^2}{\|\vs_t\|^2} = \|\bar{\mH}_t-\mH_t\|^2_{\op}$. 
    Combining the two inequalities above, we get  
    $\sum_{t=0}^{T-1} \ell_t(\mH_t) \leq\sum_{t=0}^{T-1} \|\bar{\mH}_t-\mH_t\|^2_{\op} \leq  \frac{L_2^2}{4}\sum_{t=0}^{T-1} \|\vs_t\|^2$. 
    Note that by our notation, $\sum_{t=0}^{T-1} \|\vs_t\|^2 = \sum_{k \in \calB} \|\vs_k\|^2 = \sum_{k \in \calB} \|\tilde{\vz}_k - \vz_k\|^2$. Using the same arguments as in \eqref{eq:comparator_2} of Lemma~\ref{lem:comparator}, we have $\sum_{k\in \calB} \|\tilde{\vz}_k-\vz_k\|^2  \leq 2C \|\vz_0-\vz^*\|^2$. This proves that $\sum_{t=0}^{T-1} \ell_t(\mH_t) \leq \frac{CL_2^2}{2} \|\vz_0-\vz^*\|^2$. 

    Next, using the fact that $\|\mA\|_F \leq \sqrt{d} \|\mA\|_{\op}$ for any matrix $\mA \in \reals^{d\times d}$, we have 
    \begin{equation*}
      \sum_{t=0}^{T-2} \|\mH_{t+1}-\mH_t\|_F \leq \sqrt{d} \sum_{t=0}^{T-2} \|\mH_{t+1} - \mH_t\|_{\op} \leq \sqrt{d} L_2 \sum_{t=0}^{T-2} \|\vz_{t+1} - \vz_t\|,
    \end{equation*}
    where we used Assumption~\ref{assum:jacobian_lips} in the last inequality.
    Note that by our notation, we denote the indices in $\calB$ as $\{k_0,k_1,\dots,k_{T-1}\}$, and $\vz_t$ is a shorthand for $\vz_{k_t}$. Thus, the sum $\sum_{t=0}^{T-2} \|\vz_{t+1} - \vz_t\|$ in the above inequality becomes $\sum_{t=0}^{T-2} \|\vz_{k_{t+1}} - \vz_{k_t}\|$, and by the triangle inequality we have $ \|\vz_{k_{t+1}} - \vz_{k_t}\| \leq \sum_{k = k_t}^{k_{t+1}-1} \|\vz_{k+1} - \vz_k\|$. 
    Hence, combining this with the second inequality in Lemma~\ref{lem:sum_of_squares_monotone}, we further have $\sum_{t=0}^{T-2} \|\vz_{k_{t+1}} - \vz_{k_t}\| \leq \sum_{k=0}^{N-1} \|\vz_{k+1}-\vz_k\|  \leq \sqrt{N \sum_{k=0}^{N-1} \|\vz_{k+1}-\vz_k\|^2}  \leq  \sqrt{\frac{2 N}{1-(\alpha_1+\alpha_2)}}\|\vz_0-\vz^*\|$,
    where we used Cauchy-Schwarz inequality in the last inequality. 
  \end{proof}

  \noindent\textbf{Step 3}: Combining Lemma~\ref{lem:comparator_dynamic} and  Lemma~\ref{lem:dynamic_small_loss},  we obtain that $\sum_{t=0}^{T-1} \ell_t({\mB}_t) \leq {2CL_2^2} \|\vz_0-\vz^*\|^2 + 162\|\mB_0-{\mH_0}\|_F^2  + \frac{648\sqrt{2}d L_1L_2\|\vz_0-\vz^*\|}{\sqrt{1-(\alpha_1+\alpha_2)}} \sqrt{N} + 72L_1^2\sqrt{N}
  $. Given the definitions of $Q_1$ and $Q_2$ in \eqref{eq:def_Q1} and \eqref{eq:def_Q2}, the above upper bound can be further simplified as $\sum_{t=0}^{T-1} \ell_t({\mB}_t) \leq L_1^2 Q_1^2 + L_1^2Q_2^2 \sqrt{N}$. 
Together with the fact that $\sigma_0 = \frac{\alpha_2\beta}{5L_1}$, if follows from Lemma~\ref{lem:stepsize_bnd} that 
\begin{equation*}
  \sum_{k=0}^{N-1} \frac{1}{\eta_k} \leq \frac{5L_1}{(1-\beta)\alpha_2 \beta} + \frac{\sqrt{N(L_1^2 Q_1^2 + L_1^2Q_2^2 \sqrt{N})}}{(1-\beta)\alpha_2\beta} \leq  \frac{L_1(5+Q_1 \sqrt{N}+Q_2 N^{\nicefrac{3}{4}})}{(1-\beta)\alpha_2 \beta}.
\end{equation*}
In light of Proposition~\ref{prop:HPE}(b) and the observation in \eqref{eq:after_Cauchy_Schwarz}, we obtain the second rate in Part (b) of Theorem~\ref{thm:main_monotone}.  
\end{proof}

\subsubsection{Charaterizing the computational cost}
In this section, we characterize the computation cost of QNPE in the monotone setting. Since the proof techniques are similar to those used in the strongly monotone setting, we summarize the computational cost in the following theorem and defer the proof to Appendix~\ref{appen:computational_cost_monotone}. 
\begin{theorem}
  \label{thm:computational_cost_monotone}
  Let $N_\epsilon$ denote the minimum number of iterations required by Algorithm~\ref{alg:Full_Equasi-Newton} to find an $\epsilon$-accurate solution according to Theorem~\ref{thm:main_monotone}. %
  \begin{enumerate}[(a)]
    \item QNPE requires at most $3N_\epsilon+\log_{1/\beta}(\frac{5\sigma_0 L_1}{\alpha_2})$ operator evaluations. 

    \item For Problems~\eqref{eq:monotone} and~\eqref{eq:minimax}, QNPE requires ${\bigO\left(N_{\epsilon} + \frac{L_1\max_{\vz \in \calD}\|\vz_0-\vz\|^2}{\epsilon} + \sigma_0 L_1\right)}$ matrix-vector products for $\mathsf{LinearSolver}$. Moreover, it requires $\bigO (N_\epsilon^{1.125} \log (\frac{dN_{\epsilon}^2}{p^2}))$ matrix-vector products for $\mathsf{ExtEvec}$ and $\mathsf{MaxSvec}$. 
    \item For Problem~\eqref{eq:minimization},
    QNPE requires ${\bigO\left(N_{\epsilon} + \sqrt{N_{\epsilon}\frac{L_1\max_{\vz\in \calD}\|\vz_0-\vz\|^2}{2\epsilon}} + \sqrt{\sigma_0 L_1}\right)}$ matrix-vector products for $\mathsf{LinearSolver}$. Moreover, it requires $\bigO (N_\epsilon^{1.125} \log (\frac{dN_{\epsilon}^2}{p^2}))$ matrix-vector products for $\mathsf{ExtEvec}$. %
  \end{enumerate}
\end{theorem}

As in the case of strongly monotone problems (cf. Theorem~\ref{thm:line_search}), Theorem~\ref{thm:computational_cost_monotone} shows that the average number of operator evaluations per iteration is bounded by a constant close to 3 for sufficiently large $N$. 
Moreover, note that Theorem~\ref{thm:computational_cost_monotone} implies that $N_{\epsilon} = \bigO\left(\min\{\frac{1}{\epsilon}, \frac{d^{\nicefrac{2}{5}}}{\epsilon^{\nicefrac{4}{5}}}\}\right)$. Hence, for the nonlinear equation in \eqref{eq:monotone} and the minimax problem in \eqref{eq:minimax}, 
Theorem~\ref{thm:computational_cost_monotone}(a) demonstrates that the total number of matrix-vector products can be bounded by  $\mathcal{O}(\frac{1}{\epsilon})$ and $\tilde{\bigO}\left(\min\{\frac{1}{\epsilon^{\nicefrac{9}{8}}}, \frac{d^{\nicefrac{9}{20}}}{\epsilon^{\nicefrac{9}{10}}}\}\right)$ for $\mathsf{LinearSolver}$ and $\mathsf{ExtEvec}$ and $\mathsf{MaxSvec}$, respectively. Moreover, for the minimization problem in~\eqref{eq:minimization}, Theorem~\ref{thm:computational_cost}(b) shows that the total number of matrix-vector products can be bounded by $\bigO\left(\min \{\frac{1}{\epsilon}, \frac{d^{\nicefrac{1}{5}}}{\epsilon^{\nicefrac{9}{10}}}\}\right)$ and $\tilde{\bigO}\left(\min\{\frac{1}{\epsilon^{\nicefrac{9}{8}}}, \frac{d^{\nicefrac{9}{20}}}{\epsilon^{\nicefrac{9}{10}}}\}\right)$ for $\mathsf{LinearSolver}$ and $\mathsf{ExtEvec}$, respectively. In particular, the computational cost is dominated by the latter in this case. 

\section{Conclusion}
\label{sec:conclusion}
We proposed a novel quasi-Newton proximal extragradient method for solving smooth and monotone nonlinear equations, particularly relevant to unconstrained minimization optimization and minimax optimization. 
We also demonstrated how to exploit structures in the Jacobian matrices, such as symmetry and sparsity, leading to more efficient implementation. 
In the strongly monotone setting, we established a global linear convergence rate of $(1+ (\frac{\mu}{30L_1}))^{-k}$, comparable to the extragradient method, and an explicit superlinear convergence rate of $(1+\Omega(\sqrt{k}))^{-k}$. 
Moreover, in the monotone setting, we demonstrated a global convergence rate of $\mathcal{O}\left(\min\{\frac{1}{k},\frac{\sqrt{d}}{k^{1.25}}\}\right)$, which is also superior to the convergence rate of $\bigO(\frac{1}{k})$ by the extragradient method. 
Our results are the first to show a better global complexity bound for using a quasi-Newton method over first-order methods such as the extragradient method, without requiring access to the operator's Jacobian.

\newpage 
\appendix
\section*{Appendix}

\section{The line search scheme termination}
\label{appen:terminate}

In our backtracking line search scheme, we repeatedly reduce the step size ${\eta}_+$ by a factor of $\beta$ until we find a pair $({\eta}_+,\hat{\vz}_+)$ that satisfies the condition in \eqref{eq:step size_condition} (refer also to Lines~\ref{line:check_condition} and~\ref{line:decreasing_stepsize} in Subroutine~\ref{alg:ls}). The following lemma demonstrates that once the step size ${\eta}_+$ falls below a specific threshold, the pair $({\eta}_+,\hat{\vz}_+)$ will meet both conditions in \eqref{eq:x_plus_update} and \eqref{eq:step size_condition}. Consequently, this ensures that Subroutine~\ref{alg:ls} will terminate within a finite number of steps.

\begin{lemma}\label{lem:ls_terminates}
  Suppose Assumption~\ref{assum:operator_lips} holds. If ${\eta}_+ <\frac{\alpha_2}{L_1+\|\mB\|_{\op}}$ and $\hat{\vz}_+$ is computed according to \eqref{eq:linear_solver_update}, then the pair $({\eta}_+,\hat{\vz}_+)$ satisfies the conditions in \eqref{eq:x_plus_update} and \eqref{eq:step size_condition}. 
\end{lemma}

\begin{proof}
  Since $\hat{\vz}_+$ follows the update rule in \eqref{eq:linear_solver_update}, 
  by Definition~\ref{def:linear_solver}, the pair $({\eta}_+,\hat{\vz}_+)$ will always satisfy the condition in \eqref{eq:x_plus_update}. 
  Hence, it is sufficient to prove that the condition in \eqref{eq:step size_condition} also holds. 
  Recall that $\vg$ is defined as $\vg = \vF(\vz)$. By Assumption~\ref{assum:operator_lips}, the operator $\vF$ is $L_1$-Lipschitz and thus it follow that  
  $\|\vF(\hat{\vz}_+)-\vg\| = \|\vF(\hat{\vz}_+)-\vF(\vz)\| \leq L_1 \|\hat{\vz}_+ - \vz\|$. 
  Further, applying the triangle inequality yields 
  \begin{equation*}
      \|\vF(\hat{\vz}_+)-\vg-\mB(\hat{\vz}_+-\vz)\| \leq \|\vF(\hat{\vz}_+)-\vg\| + \|\mB(\hat{\vz}_+-\vz)\| \leq (L_1  + \|\mB\|_{\op}) \|\hat{\vz}_+ - \vz\|.
  \end{equation*}
  Hence, if ${\eta}_+ \leq \frac{\alpha_2}{L_1+\|\mB\|_{\op}}$, we have 
     \begin{equation}\label{eq:approx_error}
         {\eta}_+ \|\vF(\hat{\vz}_+)-\vg-\mB(\hat{\vz}_+-\vz)\| \leq \alpha_2 \|\hat{\vz}_+ - \vz\|.
     \end{equation}
  Finally, by using the triangle inequality again,  we combine \eqref{eq:x_plus_update} and \eqref{eq:approx_error} to establish that  
  \begin{align*}
      \|\hat{\vz}_{+}-\vz+{\eta}_+ \vF(\hat{\vz}_{+})\| &= \|\hat{\vz}_{+}-\vz+{{\eta}_+}(\vg+\mB({\hat{\vz}_{+}}-\vz)) + {\eta}_+(\vF(\hat{\vz}_+)-\vg-\mB(\hat{\vz}_+-\vz))\| \\
      &\leq \|\hat{\vz}_{+}-\vz+{{\eta}_+}(\vg+\mB({\hat{\vz}_{+}}-\vz))\| + \|{\eta}_+(\vF(\hat{\vz}_+)-\vg-\mB(\hat{\vz}_+-\vz))\| \\
      &\leq \alpha_1 \sqrt{1+\eta_+ \mu}\|\hat{\vz}_+-\vz\|+\alpha_2 \|\hat{\vz}_+-\vz\| \\
      &\leq (\alpha_1+\alpha_2)\sqrt{1+\eta_+ \mu}\|\hat{\vz}_+-\vz\|, 
  \end{align*}
  which confirms that the condition in \eqref{eq:step size_condition} is satisfied. This completes the proof. 
  \end{proof}

\section{A supporting lemma for Theorem~\texorpdfstring{\ref{thm:computational_cost}}{6.9}}
\label{appen:supporting}
\begin{lemma}\label{lem:eig_singular}
  Suppose $ \mA \in \reals^{d\times d}$ satisfies that $ \frac{1}{2}(\mA + \mA^\top) \succeq \lambda \mI$. Then we have $ \lambda_{\min}(\mA^\top \mA) \geq \lambda^2$. 
\end{lemma}
\begin{proof}
  Let $\sigma_{\min}$ be the minimum singular value of $ \mA$ and let $ \vu,\vv \in \reals^d$ be the corresponding left and right singular vectors, respectively, where $\|\vu\| = 1$ and $\|\vv\|=1$.  Therefore, we have $\mA \vv = \sigma_{\min} \vu$ and $\mA^\top \vu = \sigma_{\min} \vv$. Moreover, we have 
  \begin{equation*}
    \vv^\top \mA \vv = \sigma_{\min} \vu^\top \vv \leq \frac{1}{2}\sigma_{\min} (\|\vu\|^2+\|\vv\|^2) \leq \sigma_{\min}.
  \end{equation*} 
  Furthermore, we also have $ \vv^\top \mA \vv = \frac{1}{2}\vv^\top (\mA+\mA^\top) \vv \geq \lambda \|\vv\|^2 = \lambda$. This proves that $\sigma_{\min} \geq \lambda$. Since $\mA$ is non-singular, we have $ \lambda_{\min}(\mA^\top \mA) = \sigma_{\min}^2 = \lambda^2$. This completes the proof. 
\end{proof}

\section{Proof of Theorem~\texorpdfstring{\ref{thm:computational_cost_monotone}}{6.17}} 
\label{appen:computational_cost_monotone}
In this section, we present the proof of Theorem~\ref{thm:computational_cost_monotone}. To begin with, we prove the result in Part (a), which follows similarly to Theorem~\ref{thm:line_search}.  

\begin{proof}[Proof of Theorem~\ref{thm:computational_cost_monotone}(a)]
  Recall that $l_k$ denoted the number of line search steps in iteration $k$. Using the same arguments as in the proof of Theorem~\ref{thm:line_search}, the total number of line search steps after $N$ iterations can be bounded by $\sum_{k=0}^{N-1} l_k = 2N-1+ \log_{1/\beta}\frac{\sigma_0}{\eta_{N-1}}$. Since we have $\eta_{N-1} \geq \frac{\alpha_2\beta}{5L_1}$ as shown in the proof of Theorem~\ref{thm:main_monotone}(a), we obtain that $\sum_{k=0}^{N-1} l_k = 2N-1+ \log_{1/\beta}\frac{5\sigma_0 L_1}{\alpha_2\beta}$. Counting the additional operator evaluation for $\vF(\vz_k)$ in each iteration, we prove the result in Part (a). 
\end{proof}

In order to characterize the number of matrix-vector products, we present the following proposition regarding the convergence of Subroutine~\ref{alg:CGLS} in the monotone setting. 
\begin{proposition}\label{lem:convergence_CR}
  Let $\vs^*$ be any optimal solution of $\mA\vs^* = \vb$. Moreover, let $\{\vs_k\}_{k\geq 0}$ and $\{\vr_k\}_{k\geq 0}$ be generated by Subroutine~\ref{alg:CGLS}. If $\mA$ is non-symmetric, then we have $\|\vr_k\|_2 \leq \frac{\sqrt{2}\|\mA\|_{\op}\|\vs^*\|_2}{k+1}$. Otherwise, if $\mA$ is symmetric, then $\|\vr_k\|_2\leq \frac{\|\mA\|_{\op}\|\vs^*\|_2}{(k+1)^2}$. 

\end{proposition}

\begin{proof}
  When $\mA$ is non-symmetric, Subroutine~\ref{alg:CGLS} executes the CGLS algorithm on the least-squares problem $\|\mA \vs - \vb\|^2$. Moreover, note that CGLS is analytically equivalent to applying the conjugate gradient method on the normal equation $\mA^\top \mA \vs = \mA^\top \vb$. Thus, it follows from standard results on conjugate gradient methods that (see, e.g., \cite[Section B.2]{d2021acceleration}): $\|\mA \vs_k - \vb\|^2 \leq \frac{{2\lambda_{\max}(\mA^\top \mA)}\|\vs_0-\vs^*\|_2^2}{(k+1)^2}$. 
  By taking the square roots of both sides and noting that $\vs_0 = 0$, we obtain the first result in Proposition~\ref{lem:convergence_CR}. 
  When $\mA$ is symmetric, Subroutine~\ref{alg:CGLS} executes the conjugate residual method and
  the second result in Proposition~\ref{lem:convergence_CR} follows from \cite[Chapter 12.4]{nemirovski1995information}. 
\end{proof}

As a corollary of Proposition~\ref{lem:convergence_CR}, we upper bound the total number of matrix-vector products in Subroutine~\ref{alg:ls} for given inputs $\mA$, $\vb$, and $\rho$. 
\begin{lemma}\label{lem:mv_bound_monotone}
  Given the inputs $\mA \in \reals^{d\times d}$, $\vb\in \reals^d$, and $\rho>0$. Suppose that $\frac{1}{2}(\mA + \mA^\top) \succeq \mI$. When Subroutine~\ref{alg:ls} returns, the total number of matrix-vector product evaluations can be bounded by $\frac{2\sqrt{2}(\rho+1)}{\rho}\|\mA\|_{\op}$ if $\mA$ is non-symmetric, and $\sqrt{\frac{(\rho+1)\|\mA\|_{\op}}{\rho}}$ if $\mA$ is symmetric.  
\end{lemma}
\begin{proof}
  First of all, we shall prove that Subroutine~\ref{alg:ls} terminates when $\|\vr_k\|_2 \leq \frac{\rho}{\rho+1}\|\vs^*\|_2$. To see this, note that $\|\vs^*\|_2 \leq 
  \|\vs_k\|_2 + \|\vs_k-\vs^*\|_2$ by the triangle inequality. 
  Also, since $\frac{1}{2}(\mA + \mA^\top ) \succeq \mI$, we have 
  $\|\mA(\vs_k-\vs^*)\|_2\|\vs_k-\vs^*\|_2\geq (\vs_k-\vs^*)^\top \mA (\vs_k-\vs^*) \geq \|\vs_k-\vs^*\|_2^2$, which implies that $\|\vr_k\|_2 = \|\mA(\vs_k-\vs^*)\|_2 \geq \|\vs_k-\vs^*\|_2$. Hence, combining the two inequalities leads to $\|\vs^*\|_2 \leq \|\vs_k\|_2 + \|\vr_k\|_2$. 
  Thus, if $\|\vr_k\|_2 \leq \frac{\rho}{\rho+1}\|\vs^*\|_2$, we obtain that $\|\vr_k\|_2 \leq \frac{\rho}{\rho+1}(\|\vs_k\|_2 + \|\vr_k\|_2)$, which is equivalent to $\|\vr_k\|_2 \leq \rho \|\vs_k\|_2$ and hence the termination criterion of Subroutine~\ref{alg:ls} is satisfied.
  
  When $\mA$ is non-symmetric, by Proposition~\ref{lem:convergence_CR}, Subroutine~\ref{alg:ls} returns if $\frac{\sqrt{2}\|\mA\|_{\op}\|\vs^*\|_2}{k+1} \leq \frac{\rho}{\rho+1}\|\vs^*\|_2$, which is equivalent to $k \geq \frac{\sqrt{2}(\rho+1)}{\rho}\|\mA\|_{\op} - 1$. Since Subroutine~\ref{alg:ls} requires two matrix-vector products per iteration for a non-symmetric matrix $\mA$, the total number of matrix-vector products can be bounded by $\frac{2\sqrt{2}(\rho+1)}{\rho}\|\mA\|_{\op}$. Similarly, when $\mA$ is symmetric, it follows from Proposition~\ref{lem:convergence_CR} that Subroutine~\ref{alg:ls} returns if $\frac{{\|\mA\|_{\op}}\|\vs^*\|_2}{(k+1)^2} \leq \frac{\rho}{\rho+1}\|\vs^*\|_2$, which is equivalent to $k \geq \sqrt{\frac{(\rho+1)\|\mA\|_{\op}}{\rho}}- 1$. Since Subroutine~\ref{alg:ls} requires one matrix-vector product per iteration for a symmetric matrix $\mA$, the total number of matrix-vector products can be bounded by $\sqrt{\frac{(\rho+1)\|\mA\|_{\op}}{\rho}}$.
\end{proof}

Now we are ready to prove the remaining of Theorem~\ref{thm:computational_cost_monotone}. 
\begin{proof}[Proof of Theorem~\ref{thm:computational_cost_monotone}(b) and (c)]
First, we consider Problems~\eqref{eq:monotone} and~\eqref{eq:minimax}, where the Jacobian approximation matrices $\{\mB_k\}$ is non-symmetric. Consider the $k$-th iteration of Algorithm~\ref{alg:Full_Equasi-Newton} and let $l_k$ denote the total number of line search steps in Subroutine~\ref{alg:ls}. 
For notational convenience, we set $\eta_{-1} = \beta \sigma_0$. Note that at the $i$-th line search step ($i\geq 1$), we call the $\mathsf{LinearSolver}$ oracle with 
$\mA = \mI+\eta_{k-1} \beta^{i-2} \mB_k$ and $\rho = \alpha_1$. Moreover, recall from the proof of Theorem~\ref{thm:main_monotone}(a)that $\frac{1}{2}(\mB_k + \mB_k^\top) \succeq 0$ and $\|\mB_k\|_{\op} \leq 4L_1$.  Since $\frac{1}{2}(\mI+\eta_{k-1} \beta^{i-2} \mB_k + (\mI+\eta_{k-1} \beta^{i-2} \mB_k)^\top ) \succeq \mI$ and $\|\mI+\eta_{k-1} \beta^{i-2} \mB_k\|_{\op} \leq 1 + \eta_{k-1} \beta^{i-2} \|\mB_k\|_{\op} \leq 1 + 4\eta_{k-1} \beta^{i-2} L_1$, by using Lemma~\ref{lem:mv_bound_monotone}, we can bound the total number of matrix-vector products at the $k$-th iteration by 
\begin{equation*}
  \mathsf{MV}_k \leq \sum_{i=1}^{l_k}\frac{2\sqrt{2}(\alpha_1+ 1)}{\alpha_1} (1+4\eta_{k-1} \beta^{i-2} L_1) \leq \frac{2\sqrt{2}(\alpha_1 + 1)}{\alpha_1} l_k  + \frac{8\sqrt{2}(\alpha_1 + 1)}{\alpha_1\beta(1-\beta)} \eta_{k-1} L_1.
\end{equation*}
Thus, the total number of matrix-vector products is bounded by 
\begin{equation}\label{eq:mv_bound_monotone_last_1}
  \sum_{k=0}^{N_{\epsilon}-1} \mathsf{MV}_k \leq \frac{2\sqrt{2}(\alpha_1 + 1)}{\alpha_1} \sum_{k=0}^{N_{\epsilon}-1} l_k + \frac{8\sqrt{2}(\alpha_1 + 1) L_1}{\alpha_1\beta(1-\beta)}  \sum_{k=0}^{N_{ \epsilon}-1}\eta_{k-1}.
\end{equation}
By Theorem~\ref{thm:computational_cost_monotone}(a), we have $\sum_{k=0}^{N_{\epsilon}-1} l_k \leq 3N_{\epsilon} + \log_{1/\beta}(\frac{5\sigma_0 L_1}{\alpha_2})$.
Further, we can show that $\sum_{k=0}^{N_{\epsilon}-2} \eta_k \leq \frac{\max_{\vz \in \calD}\|\vz_0-\vz\|^2}{2\epsilon}$ by the definition of $N_{\epsilon}$. To see this, by Proposition~\ref{prop:HPE}, it holds that $\gap(\bar{\vz}_N; \calD) \leq \frac{\max_{\vz \in \calD}\,\|\vz_0-\vz\|^2}{2 \sum_{k=0}^{N-1} \eta_k}$. Moreover, since $N_{\epsilon}$ is the minimum number of iterations to achieve $\gap(\bar{\vz}_N; \calD) \leq \epsilon$, we must have $\gap(\bar{\vz}_{N_{\epsilon}-1}; \calD) > \epsilon$ and this implies the desired result. Hence, we further obtain from \eqref{eq:mv_bound_monotone_last_1} that 
\begin{equation*}
  \begin{aligned}
    \sum_{k=0}^{N_{\epsilon}-1} \mathsf{MV}_k &\leq \frac{6\sqrt{2}(\alpha_1 + 1)}{\alpha_1}N_{\epsilon} + \frac{2\sqrt{2}(\alpha_1 + 1)}{\alpha_1}\log_{1/\beta}(\frac{5\sigma_0 L_1}{\alpha_2}) \\
    &\phantom{{}\leq{}}+ \frac{4\sqrt{2}(\alpha_1 + 1) L_1 \max_{\vz \in \calD}\,\|\vz_0-\vz\|^2}{\alpha_1\beta(1-\beta) \epsilon} + \frac{8\sqrt{2}(\alpha_1 + 1) L_1 \sigma_0}{\alpha_1(1-\beta)}. 
  \end{aligned}
\end{equation*}

Next, we consider Problem~\eqref{eq:minimization}, where the Jacobian approximation matrices $\{\mB_k\}$ is symmetric and positive semi-definite. We follow similar arguments as in the non-symmetric case. Consider the $k$-th iteration of Algorithm~\ref{alg:Full_Equasi-Newton}. Recall that at the $i$-th line search step ($i\geq 1$), we call the $\mathsf{LinearSolver}$ oracle with 
$\mA = \mI+\eta_{k-1} \beta^{i-2} \mB_k$ and $\rho = \alpha_1$. Since $0 \preceq \mB_k \preceq 4L_1 \mI$, by using Lemma~\ref{lem:mv_bound_monotone}, we can bound the total number of matrix-vector products at the $k$-th iteration by 
\begin{align*}
  \mathsf{MV}_k \leq  \sum_{i = 1}^{l_k} \sqrt{\frac{(\alpha_1+1)(1+4\eta_{k-1}\beta^{i-2}L_1)}{\alpha_1}} &\leq \sum_{i=1}^{l_k} \sqrt{\frac{\alpha_1+1}{\alpha_1}} + \sum_{i=1}^{l_k} 2\sqrt{\frac{\alpha_1+1}{\alpha_1}}\sqrt{\eta_{k-1}L_1} \beta^{\frac{i}{2}-1} \\
  &\leq \sqrt{\frac{\alpha_1+1}{\alpha_1}} l_k + \frac{2\sqrt{\alpha_1+1}}{\sqrt{\alpha_1\beta}(1-\sqrt{\beta})}\sqrt{\eta_{k-1}L_1}. 
  \end{align*}
  Hence, the total number of matrix-vector products is bounded by 
\begin{equation*}
  \sum_{k=0}^{N_{\epsilon}-1} \mathsf{MV}_k \leq \sqrt{\frac{\alpha_1+1}{\alpha_1}} \sum_{k=0}^{N_{\epsilon}-1} l_k + \frac{2\sqrt{(\alpha_1+1)L_1}}{\sqrt{\alpha_1\beta}(1-\sqrt{\beta})} \sum_{k=0}^{N_{\epsilon}-1} \sqrt{\eta_{k-1}} .  
\end{equation*}
Again, by Theorem~\ref{thm:computational_cost_monotone}(a), we have $\sum_{k=0}^{N_{\epsilon}-1} l_k \leq 3N_{\epsilon} + \log_{1/\beta}(\frac{5\sigma_0 L_1}{\alpha_2})$.
Moreover, it follows from the definition of $N_{\epsilon}$ that $\sum_{k=0}^{N_{\epsilon}-2} \sqrt{\eta_k} \leq \sqrt{N_{\epsilon}} \sqrt{\sum_{k=0}^{N_\epsilon-2} \eta_k} \leq  \sqrt{N_{\epsilon}\frac{\max_{\vz \in \calD}\|\vz_0-\vz\|^2}{2\epsilon}}$. Hence, we obtain  
\begin{align*}
  \sum_{k=0}^{N_{\epsilon}-1} \textsf{MV}_k &\leq \sqrt{\frac{\alpha_1+1}{\alpha_1}} \Bigl(3N_{\epsilon} + \log_{\nicefrac{1}{\beta}}(\frac{5\sigma_0 L_1}{\alpha_2})\Bigr) + \frac{2\sqrt{(\alpha_1+1)}}{\sqrt{\alpha_1\beta}(1-\sqrt{\beta})} \sqrt{N_{\epsilon}\frac{L_1\max_{\vz \in \calD}\|\vz_0-\vz\|^2}{2\epsilon}} \\
  &\phantom{{}={}}  + \frac{2\sqrt{(\alpha_1+1)\sigma_0 L_1}}{\sqrt{\alpha_1}(1-\sqrt{\beta})}.  
\end{align*}

Finally, we bound the number of matrix-vector products for $\mathsf{ExtEvec}$ and $\mathsf{MaxSvec}$ in all cases. Since the parameters are selected as $\delta_t = {\frac{1}{2(t+1)^{\nicefrac{1}{4}}}}$ and $q_t = \frac{p}{2.5(t+1)\log^2(t+1)} \leq \frac{p}{2.5N_{\epsilon}\log^2(N_{\epsilon})}$, it follows from Lemmas~\ref{lem:extevec_bound} and~\ref{lem:maxsvec_bound} that the total number of matrix-vector products can be bounded by 
    $\bigO\left(\sum_{t=1}^T \sqrt{\frac{1+\delta_t}{\delta_t}}\log\frac{d}{q_t^2} \right) = \bigO\left(\sum_{t=1}^T (t+1)^{\nicefrac{1}{8}} \log (\frac{dN_{\epsilon}^2}{p^2})\right) = \bigO\left(N_{\epsilon}^{1.125} \log (\frac{dN_{\epsilon}^2}{p^2})\right)$.
    The proof is complete.
  \end{proof}
\newpage
\printbibliography

\end{document}

%% file: math_commands.tex
\newcommand\Vector[1]{\bm{#1}}

\newcommand\va{{\Vector{a}}}
\newcommand\vb{{\Vector{b}}}

\newcommand\vg{{\Vector{g}}}

\newcommand\vp{{\Vector{p}}}
\newcommand\vq{{\Vector{q}}}
\newcommand\vr{{\Vector{r}}}
\newcommand\vs{{\Vector{s}}}

\newcommand\vu{{\Vector{u}}}
\newcommand\vv{{\Vector{v}}}
\newcommand\vw{{\Vector{w}}}
\newcommand\vx{{\Vector{x}}}
\newcommand\vy{{\Vector{y}}}
\newcommand\vz{{\Vector{z}}}

\newcommand\MATRIX[1]{\mathbf{#1}}

\newcommand\mA{{\MATRIX{A}}}
\newcommand\mB{{\MATRIX{B}}}

\newcommand\mF{{\MATRIX{F}}}
\newcommand\mG{{\MATRIX{G}}}
\newcommand\mH{{\MATRIX{H}}}
\newcommand\mI{{\MATRIX{I}}}
\newcommand\mJ{{\MATRIX{J}}}

\newcommand\mM{{\MATRIX{M}}}

\newcommand\mP{{\MATRIX{P}}}

\newcommand\mS{{\MATRIX{S}}}
\newcommand\mT{{\MATRIX{T}}}

\newcommand\mW{{\MATRIX{W}}}

\newcommand\sS{{\mathbb{S}}}

\newcommand\bigO{\mathcal{O}}

\DeclareMathOperator*{\E}{\mathbb{E}}

\newcommand\reals{\mathbb{R}}
\newcommand\op{{\mathrm{op}}}

\DeclareMathOperator*{\argmax}{arg\,max}

\DeclareMathOperator{\sign}{sign}